\documentclass[11pt,reqno]{amsart}

\usepackage[foot]{amsaddr}
\usepackage[margin = 1in]{geometry}
\usepackage{xcolor}
\usepackage{mathtools}
\usepackage[sort,numbers]{natbib}
\usepackage{bm}
\usepackage{bbm}
\usepackage{amsfonts}       
\usepackage{amsmath} 
\numberwithin{equation}{subsection}
\usepackage{amsthm}
\usepackage{amssymb}
\usepackage{enumitem}
\usepackage{mathrsfs}
\usepackage{comment}

\usepackage{tikz}

\usepackage{graphicx} 
\graphicspath{ {./images/} }
\usepackage{subcaption}

\usepackage[linktocpage = true]{hyperref}
\hypersetup{colorlinks = true, linkcolor = blue, citecolor = blue, urlcolor = blue}

\usepackage{centernot}

\theoremstyle{plain}

\newtheorem{thm}{Theorem}[section]      
\newtheorem{prop}[thm]{Proposition}
\newtheorem{lemma}[thm]{Lemma}
\newtheorem{cor}[thm]{Corollary}

\theoremstyle{definition}
\newtheorem{defi}[thm]{Definition}
\newtheorem{assumption}[thm]{Assumption}

\theoremstyle{remark}
\newtheorem{remark}[thm]{Remark}

\newcommand{\N}{\mathbb{N}}
\newcommand{\Z}{\mathbb{Z}}

\renewcommand{\Pr}{\mathbb{P}}
\newcommand{\E}{\mathbb{E}}

\renewcommand{\d}{\mathrm{d}}
\newcommand{\1}{\mathbbm{1}}

\newcommand{\leqst}{\leq_{\mathrm{st}}}
\newcommand{\geqst}{\geq_{\mathrm{st}}}

\DeclareMathOperator{\V}{Var}

\DeclareMathOperator{\MRSC}{MRSC}
\DeclareMathOperator{\LM}{LM}

\newcommand{\e}{\epsilon}

\newcommand{\lp}{\left(}
\newcommand{\rp}{\right)}
\newcommand{\lb}{\left[}
\newcommand{\rb}{\right]}

\newcommand{\sd}[1]{\footnote{Souvik: #1}}

\def\sss{\scriptscriptstyle}

\title{Giants through Higher-Order paths in Random Simplicial Complexes}

\author[Dhara]{Souvik Dhara$^\dagger$}
\author[Kang]{Taegyu Kang $^\ddagger$}

\address{$^\dagger$ Georgia Institute of Technology. \emph{Email:} \href{mailto:sdhara@gatech.edu}{\tt sdhara@gatech.edu}}
\address{$^\ddagger$ Georgia Institute of Technology. \emph{Email:} \href{mailto:tkang81@gatech.edu}{\tt tkang81@gatech.edu}}


\thanks{\emph{Acknowledgements.} Part of Taegyu Kang's research was done during his Ph.D. at Purdue University, where he was partially supported by the AFOSR grant FA9550-22-1-0238 at Purdue University under the supervision of Takashi Owada.
}

\begin{document}

\begin{abstract}
We investigate the giant component formed via higher-dimensional paths in the multi-parameter random simplicial complex (MRSC) model. For a $d$-dimensional simplicial complex, we define $d$-dimensional connectivity through incidence between $(d-1)$- and $d$-dimensional simplices. 
The phase transition of the largest $d$-dimensional connected component is determined in terms of the parameter~$\lambda$ that governs the number of $d$-simplices incident to a typical $(d-1)$-simplex. In the subcritical regime, we show that the largest component contains $\Theta(\log n)$ many $(d-1)$-simplices with high probability in the MRSC model. In the supercritical regime, we determine the asymptotic proportion of $1$-simplices in the giant component in dimension $2$, for $\lambda_c < \lambda < \bar{\lambda}$, where $\bar{\lambda} > 4$ is an explicit constant. In particular, for Linial--Meshulam complexes, this result holds throughout the entire supercritical regime. Additionally, we show that the number of vertices in the giant component undergoes a discontinuous phase transition in $d$-dimensional Linial--Meshulam complexes, in the sense that the asymptotic proportion of vertices in the giant jumps from $0$ to $1$. Our approach is based on local-weak convergence. We establish local-weak convergence in probability for the MRSC model and prove the concentration result via a refined analysis of the breadth-first exploration process, that tracks contributions from newly discovered and previously explored vertices.

\end{abstract}

\maketitle
\section{Introduction}
\label{sec: intro}
The discovery of the giant component phase transition by Erd\H{o}s and R\'enyi~\cite{ER60} gave rise to one of the central themes in random graph theory. 
A version of this classical result states that, for the Erd\H{o}s-R\'enyi random graph $G(n,p)$ with $p = \frac{\lambda}{n}$, when $\lambda > 1$, the number of vertices in the largest connected component concentrates around $n\zeta(\lambda)$, where $\zeta(\lambda)$ is the survival probability of a Poisson$(\lambda)$ branching process. Moreover, the size of the second largest component is of order $\log n$ with high probability. In contrast, when $\lambda < 1$, the largest component has size $O(\log n)$ with high probability. This result has driven major subsequent developments in the study of large-scale networks; see, for example, the books~\cite{Draief;Massoulie:2009, Jackson:2010, Hofstad:2017, Hofstad:2024}, where these developments and their applications are discussed extensively. The primary goal of this paper is to establish an analogous phase transition and concentration result for the giant component in higher-dimensional analogues of random graphs, namely multi-parameter random simplicial complexes.

There have been many attempts to generalize graphs and their properties to higher-dimensional settings. One natural generalization is given by simplicial complexes. A simplicial complex is a collection of simplices; in particular, a singleton set (a vertex) is a $0$-dimensional simplex, a two-element set (an edge) is a $1$-dimensional simplex, and a three-element set (a triangle) is a $2$-dimensional simplex, and so on. Historically, simplicial complexes arose in algebraic topology as a combinatorial and algebraic framework for describing topological spaces. More recently, the relevance of higher-order connectivity has been demonstrated across a wide range of fields, including physics~\cite{Lambiotte:2019,Bobrowski;Krioukov:2022}, neuroscience~\cite{Sizemore;PhillipsCremins;Ghrist;Bassett:2019,Lord:2016}, human behavior~\cite{Sizemore:2018}, and the social sciences~\cite{Sarker:2024}. 
In the mathematical literature, two prominent perspectives on giant components and higher-order connectivity arise from homology in algebraic topology and from hypergraphs in combinatorics.
\\

\noindent \textbf{Homology and the giant shadow.}
The most elementary model generalizing Erd\H{o}s-R\'enyi random graphs to simplicial complexes is the \emph{Linial--Meshulam complex}, introduced by Linial and Meshulam~\cite{Linial;Meshulam:2006} for dimension $2$ and subsequently extended to general dimension $d \geq 2$ by Meshulam and Wallach~\cite{Meshulam;Wallach:2009}. The model, denoted by $\text{LM}_d(n,p)$, is constructed by including all simplices up to dimension $d-1$ on $n$ vertices and then including each $d$-dimensional simplex independently with probability $p$. Linial and Meshulam~\cite{Linial;Meshulam:2006} first identified a phase transition for homological $1$-connectivity of $\text{LM}_2(n,p)$ over $\mathbb{Z}/2\mathbb{Z}$, where connectivity is defined via the vanishing of the first homology group. In the case of graphs, this is closely related to the classical connectivity threshold, since a graph is connected if and only if its reduced $0$-th homology group is trivial. This result was subsequently extended to general dimensions by \cite{Meshulam;Wallach:2009,Luczak;Peled:2018,Newman;Paquette:2023}.
A more general model, referred to as the \emph{multi-parameter random simplicial complex (MRSC)} model (see Definition~\ref{defn:MRSC}), was introduced by Costa and Farber~\cite{Costa;Farber:2016, Costa;Farber:2017A, Costa;Farber:2017B}. This framework generalizes the Linial--Meshulam complex and encompasses other models, such as the random clique complex~\cite{Kahle:2014}. 

The notion of the \emph{giant shadow} was introduced in the homological setting to capture a phase transition analogous to the giant component phenomenon. The shadow of a $d$-dimensional simplicial complex $X$ is defined as the set of all $d$-simplices $\sigma$ such that $\sigma \notin X$, but whose addition creates a new $d$-cycle. Linial and Peled~\cite{Linial;Peled:2016} identified the phase transition corresponding to having a shadow of positive density in $\text{LM}_d(n,p)$. In the terminology of random graphs and percolation theory, in dimension $d=1$, this corresponds to the susceptibility function exhibiting a jump from $O(1)$ to $\Omega(n)$, which is a well-known alternative characterization of the classical giant component problem. However, in higher dimensions, this notion is not equivalent to connectivity via higher-order paths, which is the focus of our work.
For a comprehensive overview of random simplicial complexes and their properties, we refer to the survey~\cite{Bobrowski;Krioukov:2022}.\\

\noindent \textbf{Giant in Hypergraphs.}
In combinatorial approaches, the random $(d+1)$-uniform hypergraph model $H \sim H_{d+1}(n,p)$ serves as a canonical generalization of the Erd\H{o}s--R\'enyi random graph, where each $(d+1)$-hyperedge (i.e., a subset of $d+1$ vertices) appears independently with probability $p$. This model can be viewed as a simplicial complex via its \emph{downward closure} (i.e., by including all non-empty subsets of each hyperedge), which is closely related to $\text{LM}_d(n,p)$ up to the presence of additional $k$-dimensional simplices for $k \leq d-1$.
Classical work of Schmidt-Pruzan and Shamir~\cite{SS85} identified the phase transition in terms of the number of vertices in the largest connected component, while Karo\'nski and {\L}uczak~\cite{KL02} determined its size in the barely supercritical regime. In these works, connectivity and paths are defined via hyperedges having only non-empty intersections, rather than through higher-order paths (see Definition~\ref{defn:higher-order-path}), leading to a continuous phase transition for the number of vertices, in contrast to the discontinuous behavior observed in our setting. More refined results on this notion of connectivity including centrala and local limit theorems were obtained in~\cite{RR06,BCK10,BR12}.

To the best of our knowledge, the only works that study the giant component phase transition in hypergraphs via higher-order paths are~\cite{Cooley;Kang;Koch:2018, Cooley;Kang;Person:2018}. In~\cite{Cooley;Kang;Person:2018}, Cooley, Kang, and Person identified the threshold for the emergence of a giant component in $H \sim H_{d+1}(n,p)$. The asymptotic size of the giant component was subsequently studied by Cooley, Kang, and Koch~\cite{Cooley;Kang;Koch:2016}, but their results and methods are restricted to the barely supercritical regime, where $p = p_c(1+\varepsilon)$ with $\varepsilon \to 0$.\\

\noindent \textbf{Our Contribution.} In this paper, we establish a sequence of results on the phase transition and concentration of the giant component formed through higher-order paths in multi-parameter random simplicial complexes (MRSC), which includes Linial–Meshulam complexes as a special case. We identify a phase transition governed by a parameter $\lambda$, which can be interpreted as the expected number of $d$-simplices incident to a typical $(d-1)$-simplex.

\begin{enumerate}
    \item \emph{Subcritical Phase:} For $\lambda < \frac{1}{d}$, we show that the largest component contains between $a \log n$ and $b \log n$ many $(d-1)$-simplices with high probability, for some explicit constants $a,b > 0$ (Theorem~\ref{thm: subcritical regime}). \vspace{.1cm}

    \item \emph{Supercritical Phase:} In dimension 2, we show that the proportion of $1$-simplices (or edges) in the largest component concentrates around $\zeta_\lambda$, where $\zeta_\lambda$ is the survival probability of a branching process with offspring distribution $2\times \mathrm{Poisson}(\lambda)$. 
    This result holds for the entire supercritical regime in the Linial-Meshulam complex~$\text{LM}_2(n,\frac{\lambda}{n})$, while, for the general MRSC model, this concentration holds for $1/2 < \lambda < \bar{\lambda}$ where $\bar{\lambda} > 4$ (Theorem~\ref{thm: supercritical MRSC}). The regime it becomes more permissive as the edge occupancy leads to denser graphs.
    \vspace{.1cm}

    \item \emph{Discontinuous Transition:} We show that the proportion of vertices in the giant component undergoes a discontinuous phase transition at $\lambda = \frac{1}{d}$ in $\text{LM}_d(n,\frac{\lambda}{n})$. Specifically, this proportion converges in probability to $0$ for $\lambda < \frac{1}{d}$ and to $1$ for $\lambda > \frac{1}{d}$ (Theorem~\ref{thm: vertex}).
\end{enumerate}

\noindent\textbf{Methods and Challenges.} Our approach builds on a recent general framework of van der Hofstad~\cite{Hofstad:2021,Hofstad:2024} for proving concentration of giant components in random graphs. We extend this framework to simplicial complexes and show that, for any random simplicial complex model, the proportion of $(d-1)$-simplices in the $d$-dimensional giant component concentrates provided the following two conditions hold: 
(1) the underlying simplicial complexes converge in probability in the sense of \emph{Benjamini--Schramm local weak convergence}; 
(2) for two uniformly chosen $(d-1)$-simplices, conditioned on their neighborhoods being sufficiently large, there exists a higher-order path connecting these neighborhoods (Proposition~\ref{prop: side condition for the giant}). For the MRSC model, we first establish local weak convergence in probability (Theorem~\ref{thm: LWC in Prob of MRSC}), which is a stronger version of that proved by Kanazawa~\cite{Kanazwa:2022}. An immediate corollary of this result is Proposition~\ref{prop: number of components}, which provides asymptotic estimates for the number of higher-order connected components, which is of independent interest.

The main technical challenge lies in verifying the second condition. In random graphs, this condition is typically verified by showing that once the neighborhoods of two randomly chosen vertices become sufficiently large, they grow exponentially, and once each boundary contains on the order of $L\sqrt{n}$ vertices (for large $L$), a connection between the two neighborhoods occurs with high probability. Thus, it suffices to explore neighborhoods up to size $\Theta(\sqrt{n})$ to verify this condition without exploring the whole graph. However, this intuition breaks down fundamentally in simplicial complexes. For instance, in two-dimensional Linial--Meshulam complexes, ensuring the existence of a higher-order path requires exploring neighborhoods until $\omega(n)$ edges are discovered. A purely \emph{forward exploration} (see Section~\ref{subsec: breadth-first exploration}), which tracks connections from the explored set to previously unseen vertices, is insufficient, as it can yield at most $O(n)$ edges. Instead, it is necessary to incorporate a \emph{backward exploration}, which accounts for connections formed within the currently explored vertex set.

Our analysis reveals a two-phase structure of the exploration process: in the initial phase, forward exploration dominates, while once $\Omega(n)$ vertices have been discovered, backward exploration becomes the primary mechanism. Throughout, we must carefully track the number of edges incident to the explored set. This contribution is non-negligible and, in fact, can grow exponentially according to a different branching process when $\lambda$ is large. This phenomenon is a key source of technical difficulty and explains why our analysis is currently restricted to $\lambda < \bar{\lambda}$. Extending the concentration results to the regime $\lambda > \bar{\lambda}$ remains an interesting direction for future work.
\\

\noindent\textbf{Organization.}
The remaining part of this paper is organized as follows. In Section \ref{subsec: intro simplicial complex} and \ref{subsec: intro random simplicial complex}, we define $d$-dimensional connectivity in a simplicial complex and introduce the multi-parameter random simplicial complex (MRSC) model. In Section \ref{subsec: main results}, we present the main theorems of this paper. 
In Section \ref{sec: prelim}, we introduce the local-weak convergence of random simplicial complexes and provide general properties connecting the local-weak convergence of random simplicial complexes and their higher-dimensional connected components, which are applicable to any random simplicial complex models. Although the local-weak convergence of simplicial complexes has recently been utilized as a versatile tool to analyze random simplicial complexes, we have not found a detailed description of its metric formulation, and proofs of basic facts like completeness and seperability of the underlying metric space. Hence, we add its detailed formulation in Appendix \ref{appendix: formulation of LWC}. In Section \ref{subsec: breadth-first exploration} and \ref{subsec: breadth-first exploration-stochastic properties}, we describe the breadth-first exploration of a higher-order component in a simplicial complex and present various probabilistic estimates. In Section \ref{subsec: LWC of MRSC}, we establish the local-weak convergence of MRSCs (Theorem~\ref{thm: LWC in Prob of MRSC}) by using the breadth-first exploration on local neighborhoods. The formal proofs of Theorem \ref{thm: subcritical regime}, \ref{thm: supercritical MRSC}, and \ref{thm: vertex} are presented in Section \ref{sec: proof}.
\\

\noindent\textbf{General notation and terminology.} Throughout this paper, we often use the Bachmann--Landau asymptotic notation $o(1), O(1), \Theta(1)$ etc. Sometimes we express $a = O(b)$ as $a \lesssim b$ for convenience. For two sequences $(a_n)_{n\geq 1}$ and $(b_n)_{n\geq 1}$, we write $a_n \asymp b_n$ as a shorthand for $\lim_{n\to \infty}\frac{a_n}{b_n} =1$. 
Given a sequence of probability measures $(\Pr_n)_{n\geq 1}$, a sequence of events $(\mathcal{E}_n)_{n\geq 1}$ is said to hold \emph{with high probability} (\textit{w.h.p.} for brevity) if $\lim_{n\to\infty} \Pr_n(\mathcal{E}_n) = 1$. We use $\xrightarrow{\sss\Pr}$ and $\xrightarrow{\sss D}$ to denote convergence in probability and in distribution, respectively. We write $[n]$ for the set of $n$ integers $\{1, 2, \dots, n\}$. Additionally, we write $\mathsf{Ber}$, $\mathsf{Bin}$, and $\mathsf{Poi}$ for Bernoulli, Binomial, and Poisson distributions, respectively.

\subsection{Simplicial complexes and higher-order connectivity}
\label{subsec: intro simplicial complex}
We start by formally defining simplicial complexes:
\begin{defi}[Simplicial  complexes]
\label{def: simplicial complex}
A simplicial complex $X$ is a collection of finite sets, called \emph{simplices}, satisfying the following \emph{subset closure property}: if $\tau \in X$ and $\sigma$ is a non-empty subset of~$\tau$, then $\sigma \in X$. We say that a simplex $\sigma \in X$ is a $d$\emph{-dimensional simplex} (or simply $d$-simplex) if $|\sigma| =d+1$ and write $\dim (\sigma) = d$, where $| \cdot |$ denotes the cardinality of a set. A simplicial complex is said to be $d$-dimensional if $\max_{\sigma\in X} \dim (\sigma) = d$, and write $\dim X = d$.
\end{defi}

For a simplicial complex $X$, we let $S_d(X)$ denote the set of all $d$-dimensional simplices in $X$ for each $d \in \N \cup \{0\}$, and let $s_d(X)$ denote the cardinality of $S_d(X)$. In particular, $S_0(X)$ denotes the set of vertices and $S_1(X)$ denotes the set of edges in $X$. The \textit{$d$th skeleton} $X^{(d)}$ of $X$ is defined to be the set of all simplices in $X$ up to dimension $d$, i.e., $X^{(d)} := \cup_{k=0}^d S_{k}(X)$. For $\sigma \in X$ and each $0 \leq k \leq \dim \sigma$, we call a $k$-dimensional simplices contained in $\sigma$ a $k$-dimensional \textit{face} (or simply a $k$-face) of $\sigma$. We say that a $d$-dimensional simplicial complex is \textit{pure} if every simplex in $X$ is contained in some $d$-dimensional simplex in $X$, i.e., every simplex in $X$ is a face of some $d$-dimensional simplex of $X$. Readers are referred to \cite{Munkres:1996} for a thorough exposition on these concepts.

\begin{figure}[t]
    \centering
    \includegraphics[width=0.6\textwidth]{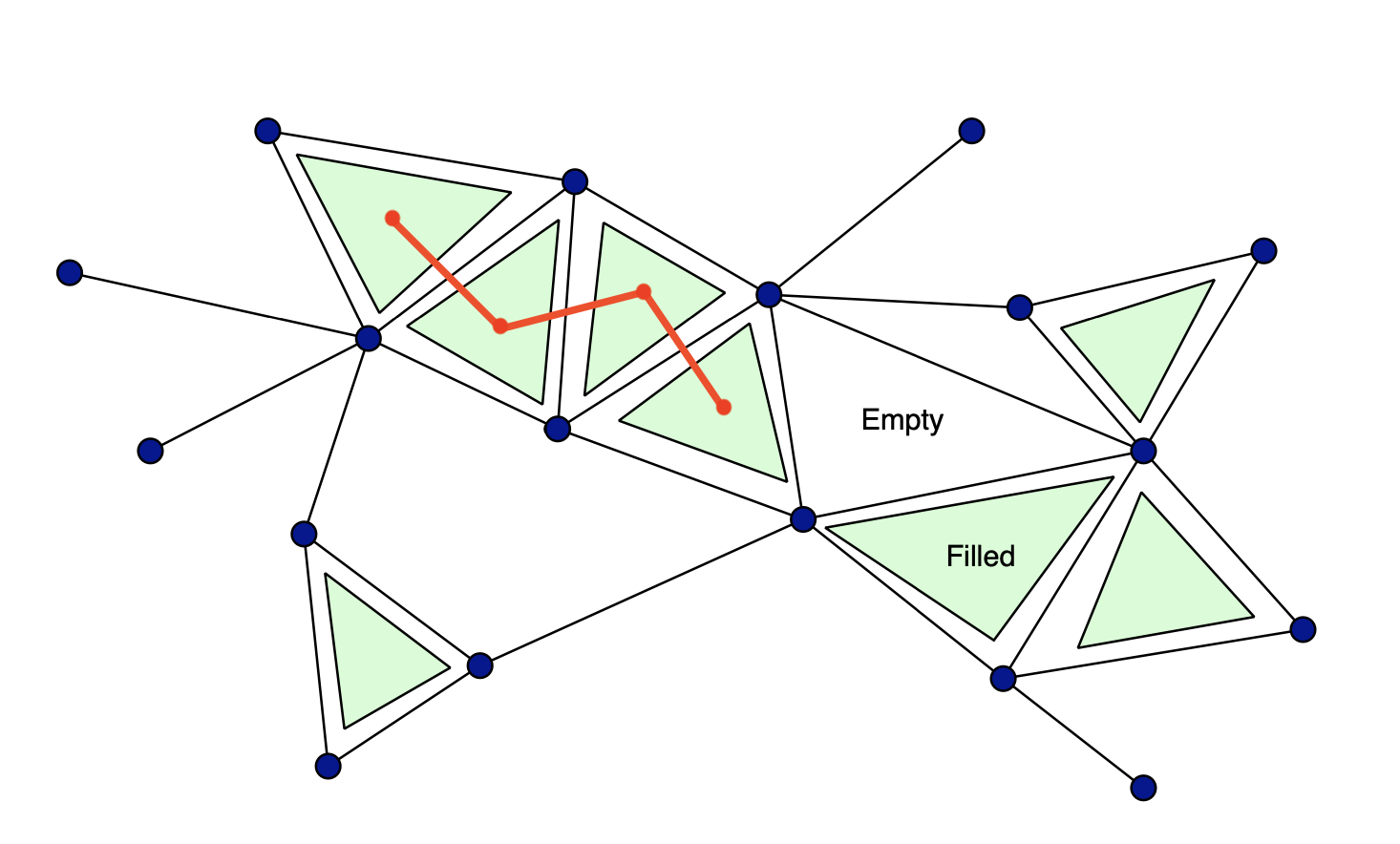}
    \caption{Visualization of a 2-dimensional simplicial complex. The green triangles are part of the simplicial complex while the empty triangles are not. The red line represents a $2$-dimensional path in the simplicial complex. The 1-skeleton consists of the nodes and edges and excludes the triangles.}
    \label{fig: complex}
\end{figure}

Next, we define $d$-dimensional connectivity. This notion of connectivity appeared in other contexts in~\cite{Horak;Jost:2013, Owada;Samorodnitsky:2025, Cooley;Kang;Koch:2018}.

\begin{defi}[$d$-dimensional paths]\label{defn:higher-order-path}
For any pair of $(d-1)$-dimensional simplices $\sigma, \sigma' \in S_{d-1}(X)$, we say that they are \textit{adjacent} if there exists a $d$-dimensional simplex $\tau \in S_d(X)$ such that $\sigma \subseteq \tau$ and $\sigma' \subseteq \tau$.  
For any pair of $(d-1)$-dimensional simplices $\sigma_1, \sigma_2 \in S_{d-1}(X)$, we say that they are \textit{adjacent} if there exists a $d$-dimensional simplex $\tau \in S_d(X)$ such that $\sigma_1 \subseteq \tau$ and $\sigma_2 \subseteq \tau$. Also, we say that $\sigma$ and $\sigma'$ are connected through a \textit{$d$-dimensional path}, denoted by $\sigma \overset{\sss d}{\longleftrightarrow} \sigma'$ if there exist distinct elements $(\sigma_i)_{i=0}^m\subseteq S_{d-1}(X)$ with $\sigma_0 = \sigma$ and $\sigma_m = \sigma'$ such that $\sigma_{i-1}$ and $\sigma_{i}$ are adjacent for all $i = 1, \dots, m$.

The smallest simplicial complex containing all $d$-dimensional simplices $\{\tau_i\}_{i=1}^m$ each of which defines the adjacency $\sigma_{i-1} \sim \sigma_i$ is called a $d$-dimensional path of length $m$ connecting $\sigma = \sigma_0$ and $\sigma' = \sigma_m$. The length of the shortest path between $\sigma,\sigma' \in S_{d-1}(X)$ is the minimum value of $m$ such that there exists a $d$-dimensional path between $\sigma,\sigma'$ of length $m$.
\end{defi}

Note that the relation $\overset{\sss d}{\longleftrightarrow}$ is an equivalence relation on $S_{d-1}(X)$, and we call its equivalence classes the \textit{$d$-dimensional connected components} in $X$. In particular, for every $\sigma \in S_{d-1}(X)$, let $\mathcal{C}(\sigma)$ denote the $d$-dimensional connected component in $X$ containing $\sigma$, i.e.,
\[
\mathcal{C}(\sigma) := \{ \sigma' \in S_d(X) \: : \: \sigma \overset{\sss{d}}{\longleftrightarrow} \sigma' \}.
\]
Furthermore, we write $\mathcal{C}_{(1)}, \mathcal{C}_{(2)}, \dots$ for the connected components in $X$ sorted in the decreasing order of the number of $(d-1)$-dimensional simplices they contain with ties being broken arbitrarily. We write $\mathcal{C}_{\mathrm{max}} := \mathcal{C}_{(1)}$. For example, for $d = 2$, two edges ($1$-simplices) are defined to be adjacent if there exists a triangle containing them and a $2$-dimensional path is a simplicial complex that consists of adjacent edges and triangles defining their adjacency. $2$-dimensional connected components are sorted by the number of edges in them (See Figure~\ref{fig: complex}).

\subsection{Probabilistic model of simplicial complexes}
\label{subsec: intro random simplicial complex}
We define the random simplicial complex model of our interest, called the multi-parameter random simplicial complex which was first introduced by Costa and Farber in \cite{Costa;Farber:2016}.

\begin{defi}[Multi-parameter random simplicial complex (MRSC)] \label{defn:MRSC}
The $d$-dimensional multi-parameter random simplicial complex model with vertex set $[n]$ and probability parameters $\bm{p} = (p_1, \dots, p_d) \in [0,1]^d$ is constructed recursively as follows. The set $[n]$ is the $0$-dimensional skeleton and each pair of vertices is connected by an edge with probability $p_1$, independent of all the other pairs. In general, for each $k$, let $S_{k}'$ denote the set of all potential $k$-dimensional simplices upon the constructed $(k-1)$-dimensional skeleton. Each $k$-dimensional simplex in $S_{k}^{'}$ is included with probability $p_{k}$. This goes on until $k = d$ or $S_{k}' = \emptyset$ for some $k \leq d$. This model is denoted by $\MRSC_d(n; \bm{p})$. Throughout, the abbreviation MRSC is used to refer to the multi-parameter random simplicial complex. 
\end{defi}

Finally, for $\MRSC(n; \bm{p})$, we define two key parameters $q_d$ and $r_d$ as follows:
\begin{equation} \label{eq: 1.2.1 q_k} 
q_d := \prod_{k = 1}^{d-1} p_{k}^{\binom{d}{k +1}}
\end{equation}
and
\begin{equation} \label{eq: 1.2.2 r_k} 
r_d := \prod_{k = 1}^{d} p_{k}^{\binom{d}{k}}.
\end{equation}
The two parameters can be interpreted as follows: Let $X_n \sim \MRSC(n; \bm{p})$, let $\sigma$ be an arbitrary $(d-1)$-dimensional simplex, and $\tau$ be an arbitrary $d$-dimensional simplex containing $\sigma$, both are on the vertex set $[n]$. Then we can write $q_d = \Pr \lp \sigma \in X_n \rp$ and $r_d = \Pr \lp \tau \in X_n \: | \: \sigma \in X_n \rp$.

\subsection{Main results}
\label{subsec: main results}
Our main results are going to be discussed under the following assumption for the parameters $q_d$ and $r_d$:

\begin{assumption}
\label{assumption: prob parameters}
The probability parameters $\bm{p} = (p_1, \dots, p_d) \in [0,1]^d$ governing the multi-parameter random simplicial complex satisfies the following:
\[ \label{eq: 1.3.1 assumption 1} \tag{1.3.1}
n^d q_d \to \infty, \text{ as } n\to \infty
\]
and
\[ \label{eq: 1.3.2 assumption 2} \tag{1.3.2}
n r_d \to \lambda, \text{ for some constant }\lambda > 0.
\]

\end{assumption}

\noindent One may notice that this regime is related to the so-called critical dimension of the MRSC~\cite{Costa;Farber:2017B}.

In fact, the condition (\ref{eq: 1.3.1 assumption 1}) means that the expected number of $(d-1)$-dimensional simplices in  MRSC tends to infinity. For the special case $d=2$, (\ref{eq: 1.3.1 assumption 1}) reduces to $p_1 = \omega( n^{-1/2})$. The condition (\ref{eq: 1.3.2 assumption 2}) means that the typical degree of a $(d-1)$-simplex in the given MRSC is $\lambda$. Thus, (\ref{eq: 1.3.2 assumption 2}) can be viewed as a generalization of the sparsity condition from the giant component problem in random graphs. These two conditions were introduced by Kanazawa in \cite{Kanazwa:2022}, and, for the Linial–Meshulam model, they essentially coincide with the assumption $p_d = \lambda/n$, which is the condition used in \cite{Linial;Peled:2016}. Recall that $s_{d-1}(\cdot)$ denotes the number of $(d-1)$-dimensional simplices in a given simplicial complex.

\begin{thm}
\label{thm: subcritical regime}
Let $X_n \sim \MRSC_d(n; \bm{p})$ satisfying Assumption~\ref{assumption: prob parameters}. Assume that $\lambda < 1/d$. Then, for any positive constants $C > d^2/ I_{\lambda}$ and $c < 1/I_{\lambda/2}$, where $I_{\lambda} := \lambda - 1 - \log \lambda$, it holds that
\[ \label{thm eq: subcritical} \tag{1.3.3}
\Pr \lp c \log n \leq s_{d-1}(\mathcal{C}_{\max}) \leq C \log n \rp = 1 - O(n^{-\delta}).
\]
\end{thm}

For $d = 2$, we establish that the $2$-dimensional giant component emerges in the supercritical regime $\lambda > 1/2$. Under appropriate conditions, the giant component is unique and its size concentrates around a positive constant $\zeta_{\lambda}$. 

\begin{thm}
\label{thm: supercritical MRSC}
Let $X_n \sim \MRSC_2(n; \bm{p})$ with $\bm{p} = (p_1, p_2)$ satisfying the assumptions (\ref{eq: 1.3.1 assumption 1}) and (\ref{eq: 1.3.2 assumption 2}) with parametrization $p_1 = \Theta(n^{-\alpha_1})$ and $p_2 = \Theta(n^{-\alpha_2})$, $\alpha_1, \alpha_2 \geq 0$, i.e., $2 \alpha_1 + \alpha_2 = 1$. Assume that $1/2 < \lambda < 2^{\frac{2(1-\alpha_1)}{\alpha_1}}$ with convention that $2^{2/0} := \infty$. Then,
\[ \label{thm eq: supercritical C_max} \tag{1.3.4}
\frac{s_1 \lp \mathcal{C}_{\max} \rp}{q_1 \binom{n}{2}} \overset{\Pr}{\to} \zeta_{\lambda}
\]
and
\[ \label{thm eq: supercritical C_2} \tag{1.3.5}
\frac{s_1 \lp \mathcal{C}_{(2)} \rp}{q_2 \binom{n}{2}} \overset{\Pr}{\to} 0
\]
as $n \to \infty$. Here, $\zeta_{\lambda}$ is given by $\zeta_{\lambda} = 1 - \gamma_{\lambda}$, where $\gamma_{\lambda}$ is the smallest solution to $\gamma = e^{-\lambda ( 1- \gamma^2)}$. In particular, when $\alpha_1 = 0$, i.e., $X_n \sim \LM_2(n, \lambda/n)$, the result holds for the entire supercritical regime.
\end{thm}

Additionally, we obtain the following discontinuous phase transition for the number of vertices in the largest $d$-dimensional connected component in $X_n \sim \LM_d(n, \lambda/n)$:

\begin{thm}
\label{thm: vertex}
Let $X_n \sim \LM_d(n, \lambda/n)$. If $\lambda < 1/d$, then
\[ \label{thm eq: vertex subcritical} \tag{1.3.6}
\frac{s_0 \lp \mathcal{C}_{\max} \rp}{n} \overset{\Pr}{\to} 0
\]
as $n \to \infty$. If $\lambda > 1/d$, then
\[ \label{thm eq: vertex supercritical C_max} \tag{1.3.7}
\frac{s_0\lp \mathcal{C}_{\max} \rp}{n} \overset{\Pr}{\to} 1
\]
and
\[ \label{thm eq: vertex supercritical C_2} \tag{1.3.8}
\frac{s_0 \lp \mathcal{C}_{(2)} \rp}{n} \overset{\Pr}{\to} 0
\]
as $n \to \infty$.
\end{thm}

\section{Local-weak convergence and giant higher-order components}
\label{sec: prelim}

In this section, we introduce the local-weak convergence of random simplicial complexes as an extension of the local-weak convergence of random graphs. First, we formally define the local-weak convergence of random simplicial complexes and discuss some general results about $d$-dimensional connected components as a result of it, such as the number of $d$-dimensional connected components and the length of the typical path joining two $(d-1)$-simplices. After that, we introduce how the emergence and the concentration of the size of the $d$-dimensional giant component are characterized via local-weak convergence.

\subsection{Local-weak convergence in probability of random simplicial complexes}
\label{subsec: LWC}

Local-weak convergence was introduced by Benjamini and Schramm~\cite{Benjamini;Schramm:2001} and Aldous and Steele~\cite{Aldous;Steele:2004} independently. Roughly speaking, local-weak convergence means that any finite neighbourhood of the typical vertex is convergent. In most cases, the limiting object of local-weak convergence has relatively simpler shapes, such as trees, and the asymptotic behavior of a given sequence of graphs can be analyzed via the properties of local limit. For a comprehensive overview of local-weak convergence of random graphs, we refer readers to \cite{Hofstad:2024}.

Local-weak convergence of random simplicial complexes has been developed relatively recently. The series of papers \cite{Aronshtam;Linial;Luczak;Meshulam:2013,Aronshtam;Linial:2015,Aronshtam;Linial:2016, Linial;Peled:2016} formulated the local-weak convergence of the Linial-Meshulam random simplicial complex $Y_d(n, p)$ to discuss its Betti numbers. Particularly, in \cite{Linial;Peled:2016}, the local-weak convergence of $Y_d(n, p)$ with $p = \lambda/n$ was formulated as the local-weak convergence of random bipartite graphs whose vertex sets are the sets of $(d-1)$-dimensional simplices and $d$-dimensioal simplices and edges given by the incidence between them. Local-weak convergence has also been extended to MRSCs. In \cite{Kanazwa:2022}, the local-weak convergence of MRSCs under Assumption~\ref{assumption: prob parameters} was established and used to analyze the Betti number of them.

We introduce the concept called the local-weak convergence in probability of random simplicial complexes. A pair $(X, \sigma)$ of a $d$-dimensional simplicial complex and a $(d-1)$-dimensional simplex $\sigma$ in $X$ is called a $(d-1)$\emph{-rooted simplicial complex}. Recall that, for a given simplicial complex $X$ and any $0 \leq k \leq d$, $S_k(X)$ denotes the set of all $k$-dimensional simplices in $X$. We say two $(d-1)$-rooted simplicial complexes $(X_1, \sigma_1)$ and $(X_2, \sigma_2)$ are isomorphic, denoted by $(X_1, \sigma_1) \cong (X_2, \sigma_2)$, if there is a bijection $\phi: S_0(X_1) \to S_0(X_2)$ that preserves all simplices and the roots, i.e., $\{v_0, \dots, v_k\} \in X_1$ if and only if $\{\phi(v_0), \dots, \phi(v_k)\} \in X_2$ and $\phi: \sigma_1 = \{ v_0, \dots, v_{d-1} \} \mapsto \{\phi(v_0), \dots, \phi(v_{d-1})\} = \sigma_2$. Additionally, for any simplex $\sigma$, we define $Q(\sigma)$ to be the minimal simplicial complex containing $\sigma$, i.e., the simplicial complex obtained by the downward closure of $\sigma$. For example, if $\sigma = \{v_0, v_1, v_2\}$, then
\[
Q(\sigma) = \left\{ \{ v_0 \}, \{ v_1 \}, \{v_2\}, \{v_0, v_1\}, \{v_1, v_2\}, \{v_2, v_0\}, \{v_0, v_1, v_2\} \right\}.
\]

For each $0 \leq k \leq d-1$ and $\sigma \in S_k(X)$, the \emph{(up-)degree} of $\sigma$ in $X$ is defined to be the number of $(k+1)$-dimensional simplices in $X$ containing $\sigma$ as a $k$-dimensional sub-simplex. 
We say that a simplicial complex $X$ is \emph{locally-finite above dimension} $k$ if every $p$-dimensional simplex $\sigma \in S_p(X)$ with $p \geq k$ has finite up-degree. Now, let $X$ be a $d$-dimensional simplicial complex, $\sigma \in S_{d-1}(X)$ and $r \in \N \cup \{0 \}$. The neighborhood $B_X(\sigma; r)$ of $\sigma$ in $X$ with radius $r$ is defined in the following way: First, let $B(\sigma; 0) := Q(\sigma)$. For a constructed $B_X(\sigma; \ell)$ with $0 \leq \ell \leq r-1$, let us define the boundary $\partial B_X(\sigma; \ell+1)$ by
\[ \label{def: boundary of nbd 2.1.1} \tag{2.1.1}
\partial B_X(\sigma; \ell + 1) := \bigcup_{\sigma' \in \mathcal{J}_{\ell}} \lb Q(\sigma') \cap B_X(\sigma; \ell)^c \rb,
\]
where $\mathcal{J}_{\ell}$ denotes the set of all $d$-dimensional simplices in $X \setminus B_X(\sigma; \ell)$ that contain at least one $(d-1)$-dimensional simplex in $B_X(\sigma; \ell)$. In other words, a $d$-dimensional simplicial complex $\tau \in \mathcal{J}_{\ell}$ if and only if $\tau \not\in B_X(\sigma; \ell)$ and there exists a $(d-1)$-dimensional simplex $\sigma' \in B_X(\sigma; \ell)$ such that $\sigma' \subseteq \tau$. We set $B_X(\sigma; \ell) := B_X(\sigma; \ell -1) \cup \partial B_X(\sigma; \ell)$. We would like to remark that if $X$ is locally-finite above dimension $(d-1)$, every neighborhood of a $(d-1)$-dimensional simplex $\sigma$ with finite radius is a finite simplicial complex. In fact, $\sigma$ is contained in up to finitely many $d$-dimensional simplices, and each of $(d-1)$-simplices contained in them---they are $d-1$-simplices that are apart from $\sigma$ by distance $1$---is again contained up to finitely many $d$-dimensional simplices, and so forth. Thus, for any finite $r > 0$, $B_X(\sigma; r)$ is either the singleton set $\{\sigma\}$ or a pure $d$-dimensional simplicial complex containing finitely many $d$-dimensional simplices.

\begin{remark}
We would like to remark that a lower dimensional local structure of a locally finite simplicial complex can be dense. To illustrate this, let $X$ be a $2$-dimensional simplicial complex constructed as follows. We start from an edge $\{v_0, v_1\}$, and a $\{v_0, v_1, v_2\}$. Among two two newly added edges, pick the one containing $v_0$, and add a new triangle to it, say $\{v_0, v_2, v_4\}$. Repeat this process so that every new triangle will contain $v_0$. Then, the obtained simplicial complex $X$ is locally finite above dimension $1$, and every finite-radius neighborhood of an edge is a finite simplicial complex. However, the degree of the vertex $v_0$ in the $1$-skeleton of $X$ is infinite.
\end{remark}

We write $\mathcal{S}_{d-1}$ for the set of the isomorphism classes of $(d-1)$-rooted simplicial complexes that are locally finite above dimension $(d-1)$. The metric $d_{\mathcal{S}_{d-1}}$ on $\mathcal{S}_{d-1}$ is given by
\[ \label{def: metric 2.1.2} \tag{2.1.2}
d_{\mathcal{S}_{d-1}} \lp (X_1, \sigma_1), (X_2, \sigma_2) \rp := \frac{1}{1 + R^*},
\]
where
\[
R^* := \sup \{R \geq 0 \: : \: B_{X_1}(\sigma_1; R) \cong B_{X_2}(\sigma_2; r) \}.
\]
The metric space $(\mathcal{S}_{d-1}, d_{\mathcal{S}_{d-1}})$ is a Polish space. This property has been taken for granted in most literature considering the local-weak convergence of simplicial complexes. However, we have not found an explicit proof for this. To make the paper self-contained, we provide the proof of this property in Appendix~\ref{appendix: formulation of LWC}.

Now, we are ready to define the local-weak convergence of random simplicial complexes. we introduce a more stronger notion called local-weak convergence in probability.

\begin{defi}[The local-weak convergence in probability of simplicial complexes]
We say that a sequence of $d$-dimensional finite random simplicial complexes $(X_n)_{n \geq 1}$ \textit{converges locally weakly in probability} to a $(d-1)$-rooted random simplicial complex $(X, o)$ with (possibly random) law $\mu$, when, for every $r \geq 0$ and $H_* \in \mathcal{S}_{d-1}$,
\[ \label{eq: 2.1.3 local conv} \tag{2.1.3}
\frac{1}{s_{d-1} \lp X_n \rp} \sum_{\sigma \in S_{d-1}(X_n)} \1 \left\{ B_{X_n}(\sigma; r) \cong H_* \right\} \overset{\Pr}{\to} \mu \lp B_X(o; r) \cong H_* \rp.
\]
\end{defi}

Typically, we say that $(X_n)_{n \geq 1}$ converges locally weakly if the expectation of the left-hand side of \eqref{eq: 2.1.3 local conv} converges to the right-hand side of some deterministic law $\mu$. The notion of local-weak convergence in probability provided here is stronger than the usual notion in this sense, and provides more suitable framework to discuss the concentration of functionals of random simplicial complexes.

\subsection{Local-weak convergence and emergence of the giant component in random simplicial complexes}
\label{subsec: LWC and connected components}

In this subsection, we consider how the local-weak convergence of random simplicial complexes leads us to the emergence and the size of the giant $d$-dimensional connected component in them. Although the general properties we discuss here can be viewed as straightforward extensions of their counterparts for random graphs, we provide their proof in Appendix \ref{appendix: formulation of LWC} to make the present paper self-contained. First, the number of $d$-dimensional connected components can be considered. 

\begin{prop}
\label{prop: number of components}
Let $(X_n)_{n \geq 1}$ be a sequence of random finite $d$-dimensional simplicial complexes that converges locally weakly in probability to a random $(d-1)$-rooted simplicial complex $(X, o) \sim \mu$. Let $C_n$ denote the number of $d$-dimensional connected components in $X_n$. Then, it holds that
\[ \label{eq: 2.2.1} \tag{2.2.1}
\frac{C_n}{s_{d-1}(X_n)} \overset{\Pr}{\to} \E_{\mu} \lb \frac{1}{s_{d-1}(\mathcal{C}(o))} \rb,
\]
where $\mathcal{C}(o)$ denotes the $d$-dimensional connected component in $X$ containing $o$. Here, by convention, we set $1/s_{d-1}(\mathcal{C}(o)) = 0$ when $s_{d-1}(\mathcal{C}(o)) = \infty$.
\end{prop}

Next, we consider the size of the largest $d$-dimensional connected components. The size of the largest component is not a local functional, however, local-weak convergence in probability provides useful partial information about it: its upper bound.

\begin{prop}
\label{prop: upper bound for the giant}
Let $(X_n)_{n \geq 1}$ be a sequence of $d$-dimensional simplicial complexes with $s_{d-1}(X_n) \to \infty$ as $n \to \infty$. Assume that $X_n$ converges locally weakly in probability to $(X, o) \sim \mu$. Let $\zeta := \Pr \lp s_{d-1} \lp \mathcal{C}(o) \rp = \infty \rp$, the survival probability of the $d$-dimensional connected component of the root $o$ in $X$. Then, for every $\e > 0$, it holds that
\[ \label{eq: 2.2.2} \tag{2.2.2}
\Pr \lp s_{d-1} \lp \mathcal{C}_{\max} \rp \leq s_{d-1}(X_n) (\zeta + \e) \rp \to 1
\]
as $n \to \infty$.
\end{prop}

Local-weak convergence itself is not enough to identify the global information of the largest $d$-dimensional connected component. However, under the following additional conditions that governs how big local neighborhoods merge into a large connected component, we can identify the asymptotic size of the largest connected component.

\begin{prop}
\label{prop: side condition for the giant}
Let $(X_n)_{n \geq 1}$ be a sequence of random finite $d$-dimensional simplicial complexes that converges locally weakly in probability to $(X, o) \sim \mu$. Assume further that 
\[ \label{eq: 2.2.3} \tag{2.2.3}
\lim_{k \to \infty} \limsup_{n \to \infty} \Pr \lp \sigma_1 \overset{\sss{d}}{\longleftrightarrow} \sigma_2, s_{d-1}(\mathcal{C}(\sigma_1)), s_{d-1}(\mathcal{C}(\sigma_2)) \geq k \rp = 1,
\]
where $\sigma_1$ and $\sigma_2$ are chosen uniformly at random from $X_n$. Then, it holds that
\[
\frac{s_{d-1}(\mathcal{C}_{\max})}{s_{d-1}(X_n)} \overset{\Pr}{\to} \mu \lp s_{d-1}(\mathcal{C}(o)) = \infty \rp 
\]
and
\[
\frac{s_{d-1}(\mathcal{C}_{(2)})}{s_{d-1}(X_n)} \overset{\Pr}{\to} 0.
\]
\end{prop}

\section{Exploration and the local-weak limit of MRSCs}
\label{sec: exploration and LWC}

In this section, we first describe the breadth-first exploration process of a $d$-dimensional connected component in a random $d$-dimensional simplicial complex and analyze the stochastic behavior of the exploration process. The general description of the process on $X_n \sim \MRSC_d(n; \bm{p})$ is similar as that introduced in \cite{Kanazwa:2022}. However, we investigate it much more meticulously so as to utilize it for explorations much beyond a local neighborhood, which is necessary for the proof of Theorem~\ref{thm: supercritical MRSC}. In Section~\ref{subsec: LWC of MRSC}, we prove that $X_n \sim \MRSC_d(n; \bm{p})$ under Assumption~\ref{assumption: prob parameters} converges locally weakly in probability to the so-called $(d-1)$-rooted Poisson tree with parameter $\lambda$. The main tool for the proof is the breadth-first exploration on local neighborhoods.

\subsection{Breadth-first exploration on a $d$-dimensional simplicial complex}
\label{subsec: breadth-first exploration}

Let $\Delta_n$ denote the complete $d$-dimensioanl simplicial complex on the vertex set $[n]$, the simplicial complex containing all possible simplices up to dimension $d$ upon the vertex set $[n]$. In particular, $\Delta_n$ contains all $\binom{n}{d}$ $(d-1)$-dimensional simplices and $\binom{n}{d+1}$ $d$-dimensional simplices on $[n]$. For each $0 \leq k \leq d$, let $S_k(\Delta_n)$ be equipped with the lexicographic ordering induced by the ordering of the vertices $1 < 2 < \cdots < n$.

We describe the breadth-first exploration on a simplicial complex $\Delta_n^{(d-1)} \subseteq X \subseteq \Delta_n$. Fix a $(d-1)$-dimensional simplex in $X$, and then explore the $(d-1)$-dimensional simplices adjacent to it one by one. The final output of the exploration is the $d$-dimensional connected component containing the fixed $(d-1)$-simplex at the beginning. 

Let $\sigma$ be a $(d-1)$-dimensional simplex in $X$, which we call the root $(d-1)$-simplex. At Step $0$, set
\[
\mathcal{X}_0 := Q(\sigma),
\]
and declare that $\sigma$ is active. At Step $k \geq 1$, we choose the active $(d-1)$-dimensional simplex, denoted by $\sigma_k$, which is closest to the root simplex $\sigma$ and is minimal with respect to the ordering on $S_{d-1}(\Delta_n)$. More formally, let $\mathcal{A}_{k-1}$ be the set of all active $(d-1)$-dimensional simplices at the end of Step $k-1$. Then $\sigma_k \in \mathcal{A}_{k-1}$ satisfies
\[
\mathrm{dist}_{\mathcal{X}_{k-1}} (\sigma, \sigma_k) = \min_{\sigma' \in \mathcal{A}_{k-1}} \mathrm{dist}_{\mathcal{X}_{k-1}} (\sigma, \sigma')
\]
and
\[
\sigma_k < \sigma' \; \text{ for every } \sigma' \in \mathcal{A}_{k-1} \setminus \{\sigma_k\},
\]
Next, detect all $(d-1)$-dimensional simplices in $X \setminus \mathcal{X}_{k-1}$ which are adjacent to $\sigma_k$ in $X \setminus \mathcal{X}_{k-1}$. The newly detected simplices are classified into three categories, which we refer to ``forward $(d-1)$-simplices'', ``backward $(d-1)$-simplices'', and ``sibling $(d-1)$-simplices'', in the following criteria: write $\sigma_k = \{v_{1,k}, \dots, v_{d,k}\}$, then 

\begin{itemize}

\item If a new $(d-1)$-simplex is a part of a $d$-simplex $\{v_{1,k}, \dots,  v_{d,k}, w \}$, where $w \not\in \mathcal{X}_{k-1}$, then the new $(d-1)$-simplex is called a {\it forward $(d-1)$-simplex}, and the $d$-simplex is called a {\it forward $d$-simplex}. We write $\mathcal{F}_{d-1, k}$ and $\mathcal{F}_{d,k}$ for the set of forward $(d-1)$-simplices and the set of forward $d$-simplices detected at Step $k$, respectively.

\item If a new $(d-1)$-simplex is a part of a $d$-simplex $\{v_{1,k}, \dots, v_{d,k}, w \}$, where $w \in \mathcal{X}_{k-1}$, but none of $(d-1)$-dimensional faces of the $d$-simplex containing $w$ is in $\mathcal{X}_{k-1}$, then the new $(d-1)$-simplex is called a {\it backward $(d-1)$-simplex}, and the $d$-simplex is called a {\it backward $d$-simplex}. We write $\mathcal{B}_{d-1,k}$ and $\mathcal{B}_{d,k}$ for the set of backward $(d-1)$-simplices and the set of backward $d$-simplices detected at Step $k$, respectively.

\item If a new $(d-1)$-simplex is a part of a $d$-simplex $\{v_{1,k}, \dots, v_{d,k}, w \}$, where $w \in \mathcal{X}_{k-1}$ and if at least one of the $(d-1)$-dimensional faces of the $d$-simplex containing $w$ is in $\mathcal{X}_{k-1}$, then the new $(d-1)$-simplex is called a {\it sibling $(d-1)$-simplex}, and the $d$-simplex is called a {\it sibling $d$-simplex}. We write $\mathcal{H}_{d-1,k}$ and $\mathcal{H}_{d,k}$ for the set of sibling $(d-1)$-simplices and the set of sibling $d$-simplices detected at Step $k$, respectively.

\end{itemize}
Note that every forward or backward $d$-simplex adds $d$ new $(d-1)$-simplices; whereas sibling $d$-simplices add less than that.

Finally, we declare all the new $(d-1)$-simplices are active and $\sigma_k$ becomes explored, and set 
\[
\mathcal{X}_k = \mathcal{X}_{k-1} \bigcup \lb \bigcup_{\tau \in \mathcal{F}_{d,k} \cup \mathcal{B}_{d,k} \cup \mathcal{H}_{d,k}} Q(\tau) \rb.
\]
The exploration process terminates when there is no longer an active $d$-simplex, i.e., $\mathcal{A}_k = \emptyset$.

\subsection{Properties of the breadth-first exploration on MRSC}
\label{subsec: breadth-first exploration-stochastic properties}

Next, we examine some probabilistic properties of this breadth-first exploration process on $X_n \sim \MRSC_d(n, \bm{p})$, which explores the $d$-dimensional connected component of each $(d-1)$-dimensional simplex $\sigma$ in $X_n$. To do so, we define the following general version of the degree of simplices in a simplicial complex.

\begin{defi}[Degree of simplices]
\label{def: degree}
Let $X$ be a $d$-dimensional simplicial complex and $\sigma$ be a $k$-dimensional simplex in $X$, $0 \leq k < d$. For $k \leq \ell \leq d$, the $\ell$-degree of $\sigma$ is defined to be the number of $\ell$-dimensional simplices in $X$ that contain $\sigma$, i.e., 
\[ \label{eq: degree 3.2.1} \tag{3.2.1}
\deg_{k, \ell}^X(\sigma) := \left| \{ \tau \in X \: : \: \dim (\tau) = \ell, \sigma \subseteq \tau \} \right|.
\]
In other words, $\deg_{k, \ell}$ is equal to the cardinality of the $(\ell - k)$-dimensional link of $\sigma$ in $X$.
\end{defi}

Note that $\deg^X_{0, 1}(v)$ for each $v \in S_0(X)$ is defined to be the vertex-edge degree of the vertex $v$ in the $1$-dimensional skeleton (graph) of $X$. Let $\sigma$ be an arbitrary $(d-1)$-dimensional simplex in $X_n$ at which we start our exploration, and let $E_k$ be the number of new active $(d-1)$-dimensional simplices obtained at Step $k$. Let us write $F_k = |\mathcal{F}_{d-1, k}|$, the number of forward $(d-1)$-dimensional simplices; $H_k := |\mathcal{H}_{d-1, k}|$, the number of sibling $(d-1)$-dimensional simplices; and $B_k := |\mathcal{B}_{d-1,k}|$, the number of backward $(d-1)$-dimensional simplices newly obtained at Step $k$. Note that $E_k = F_k + H_k + B_k$. Additionally, let us define $\tilde{E}_k := F_k + B_k$, the number of new forward and backward $(d-1)$-simplices. Clearly, $E_k \geq \tilde{E}_k$. We approximate the probability distributions of $F_k$ and $B_k$, which contribute primarily to adding new $(d-1)$-dimensional simplices to $\mathcal{C}(\sigma)$. For each $k \geq 0$, let $V_k := s_0(\mathcal{X}_k)$ be the number of vertices in $\mathcal{X}_k$.

We introduce additional parameter $p_{d, k,\ell}$ for $0 \leq \ell \leq d-1$ defined as follows.  Let $\tau, \sigma, \rho$ be arbitrary~$d$, $(d-1)$, and $\ell$-dimensional simplices on $[n]$, respectively, satisfying $\sigma, \rho \subset \tau$ and $\rho  \not\subseteq \sigma$. Then, we define
\[ \label{eq: 3.2.2} \tag{3.2.2}
p_{d, \ell} := \Pr \lp \tau \in X_n \: | \: \sigma, \rho \in X_n \rp.
\]
Note that $p_{d, 0} = r_d = \lambda/n$. The following lemma provides estimates for the probability of creating new forward or backward $d$-dimensional simplices.

\begin{lemma}
\label{lemma: forward estimate}
Let $k \geq 0$ and let $v$ be a vertex that can create a forward $d$-dimensional simplex with a $(d-1)$-simplex $\sigma_k$ that is being explored at Step $k$. Then, given $\mathcal{X}_{k-1}$, the probability that $\sigma_k \cup \{v\}$ appears at Step $k$ is lower bounded by
\[ \label{eq: forward prob 3.2.3} \tag{3.2.3}
p_k^F := \lp 1 - (k-1) \frac{\lambda}{n} - \sum_{\ell = 1}^{d-1} \sum_{\rho \subseteq \sigma_k, \dim(\rho) = \ell-1} \deg_{\ell-1, d-1}^{\mathcal{X}_{k-1}} (\rho)  p_{d, \ell}  \rp \lp \frac{\lambda}{n} + o(1) \rp.
\]
In particular, if the given MRSC model is the Linial-Meshulam model, the probability is exactly $\lambda/n + o(1)$.
\end{lemma}

\begin{proof}
Let $k \geq 0$ and let $T \subseteq \Delta_n$ be a $d$-dimensional simplicial complex with $\Pr (\mathcal{X}_{k-1} = T) > 0$. Note that under $\mathcal{X}_{k-1} = T$, the simplicial complexes $(\mathcal{X}_{j = 0}^{k-1}$ that have appeared throughout the exploration is equal to the simplicial complexes obtained by the breadth-first exploration on $T$ and $v \not\in S_0(T)$. The probability we consider can be written as
\[ \label{eq: 3.2.4} \tag{3.2.4}
\Pr \lp \sigma_k \cup \{v\} \in X_n \: | \: \mathcal{X}_{k-1} = T \rp.
\]
Let $E$ be the event defined as follows:
\[ \label{eq: 3.2.5} \tag{3.2.5}
E := \bigcap_{1 \leq j \leq k-1} \left\{ \sigma_j \cup \{v\} \not\in X_n \right\},
\]
Additionally, let $\sigma_1, \dots, \sigma_{k-1}$ be the explored $(d-1)$-simplices until Step $k-1$. For each $u \in [n]$, let the event $F_u$ be defined by
\[ \label{eq: 3.2.6} \tag{3.2.6}
F_u := \bigcap_{\{ j \: : \: 1 \leq j \leq k-1, u \not\in S_0(\mathcal{X}_{j-1}) \}} \left\{ \sigma_j \cup \{u \} \not\in X_n \setminus T \right\},
\]
and set
\[ \label{eq: 3.2.7} \tag{3.2.7}
F := \bigcap_{u \in [n] \setminus \{v\}} F_u.
\]
Here, $E$ is the event that any previously explored $(d-1)$-simplex did not create a $d$-simplex with the vertex $v$. $F$ is the event that all the other vertices did not show up before the step they were supposed to show up. Then, we can write
\[ \label{eq: 3.2.8} \tag{3.2.8}
\{\mathcal{X}_{k-1} = T \} = \{T \subseteq X_n \} \cap E \cap F.
\]

Following the argument of \cite[Lemma 16]{Kanazwa:2022}, we can write
\[ \label{eq: 3.2.9} \tag{3.2.9}
\begin{aligned}
\Pr \lp \sigma_k \cup \{v\} \in X_n \: | \: \mathcal{X}^{k-1} = T \rp & = \frac{\Pr \lp  \{ \sigma_k \cup \{v\} \in X_n \} \cap \{\mathcal{X}_{k-1} = T \} \rp}{\Pr \lp \mathcal{X}_{k-1} = T \rp} \\
& = \frac{\Pr \lp \{\sigma_k \cup \{v\} \in X_n\} \cap E \cap F \: | \: | T \subseteq X_n \rp }{\Pr \lp E \cap F \: | \: T \subseteq X_n \rp}.
\end{aligned}
\]
Because the event $F$ is independent of both $E$ and $\{\sigma_k \cup \{v\} \in X_n\}$ given $T \subseteq X_n$, the last expression of (\ref{eq: 3.2.9}) can be written as
\[ \label{eq: 3.2.10} \tag{3.2.10}
\frac{\Pr \lp \{\sigma_k \cup \{v\} \in X_n \} \cap E \: | \: T \subseteq X_n \rp \Pr \lp F \: | \: T \subseteq X_n \rp}{\Pr \lp E \: | \: T \subseteq X_n \rp \Pr \lp F \: | \: T \subseteq X_n \rp} = \frac{\Pr \lp \{\sigma_k \cup \{v\} \in X_n \} \cap E \: | \: T \subseteq X_n \rp}{\Pr \lp E \: | \: T \subseteq X_n \rp }.
\]
The right-hand side of (\ref{eq: 3.2.10}) is lower bounded by
\[ \label{eq: 3.2.11} \tag{3.2.11}
\Pr \lp \{ \sigma_k \cup \{v\} \in X_n \} \cap E \: | \: T\subseteq X_n \rp,
\]
which is equal to
\[ \label{eq: 3.2.12} \tag{3.2.12}
\Pr \lp E \: | \: \sigma_k \cup \{v\} \in X_n, T \subseteq X_n \rp \Pr \lp \sigma^k \cup \{v\} \in X_n \: | \: T \subseteq X_n \rp.
\]
By the assumption (\ref{eq: 1.3.2 assumption 2}), the latter term of (\ref{eq: 3.2.12}) equals $\lambda/n+ o(1)$. The former term can be estimated as follows: 
\[
\Pr \lp E \: | \: \sigma^k \cup \{v\} \in X_n, T \subseteq X_n \rp = 1 - \Pr \lp \bigcup_{1 \leq j \leq k-1} \left\{ \sigma_j \cup \{v \} \in X_n \right\} \: | \: \sigma_k \cup \{v\} \in X_n, T \subseteq X_n \rp \\
\]
Using the union bound, we can show that the right-hand side of the above is lower bounded by
\[ \label{eq: 3.2.13} \tag{3.2.13}
1 - \sum_{1 \leq j \leq k-1} \Pr \lp \sigma_j \cup \{v\} \in X_n \: | \: \sigma_k \cup \{v\} \in X_n, T \subseteq X_n \rp.
\]
Note that each probability term in the summation in (\ref{eq: 3.2.13}) is equal to
\[
\Pr \lp \sigma_j \cup \{v\} \in X_n \: | \: \sigma_j \in X_n, \{v\} \cup (\sigma_k \cap \sigma_j) \in X_n \rp.
\]
This probability depends on the intersection of the two $d$-simplices $\sigma_j \cup \{v\}$ and $\sigma_k \cup \{v\}$. For each $0 \leq \ell \leq d-1$, if 
\[
\dim \lp (\sigma_j \cup \{v\}) \cap (\sigma_k \cup \{v\}) \rp =  \left| \lp \sigma_j \cup \{v\} \rp \cap \lp \sigma_k \cup \{v\} \rp \right| - 1 = | \sigma_j \cap \sigma_k| = \ell,
\]
we have
\[
\Pr \lp \sigma_j \cup \{v\} \in X_n \: | \: \sigma_j \in X_n, \{v\} \cup (\sigma_k \cap \sigma_j) \in X_n \rp = p_{d, \ell},
\]
where $p_{d, \ell-1}$ is defined in (\ref{eq: 3.2.2}). In particular, if $\sigma_j$ and $\sigma_k$ are disjoint ($\ell = 0$), then the probability is equal to $\lambda/n$; whereas, if $\sigma_j$ and $\sigma_k$ share a $(d-2)$-dimensional face ($\ell =d-1$), then the probability is equal to $p_{d, d-1}$. Thus, under the event $\{\sigma_k \cup \{v\} \in X_n, T \subseteq X_n \}$, the right-hand side of (\ref{eq: 3.2.13}) is equal to
\[ \label{exp: 3.2.14} \tag{3.2.14}
1 - \sum_{0 \leq \ell \leq d-1} \sum_{1 \leq j \leq k-1} \1 \{ |\sigma_j \cap \sigma_k| = \ell  \} p_{d, \ell}.
\]
For $\ell =0$, we use the trivial upper bound $\sum_{1 \leq j \leq k-1} \1\{ |\sigma_j \cap \sigma_k| = 0 \} \leq k-1$. For $1 \leq \ell \leq d-1$ , we have
\[ \label{eq: 3.2.15} \tag{3.2.15}
\sum_{1 \leq j \leq k-1} \1\{ |\sigma_j \cap \sigma_k| = \ell \} \leq  \sum_{\substack{\rho \subseteq \sigma_k, \\ \dim (\rho) = \ell -1}} \deg_{\ell-1, d-1}^{T}(\rho).
\]
Then, the expression in (\ref{eq: 3.2.12}) is lower bounded by
\[ \label{eq: 3.2.16} \tag{3.2.16}
1 - (k-1) \frac{\lambda}{n} - \sum_{1 \leq \ell \leq d-1} \sum_{\substack{\rho \subseteq \sigma_k, \\ \dim(\rho) = \ell-1}} \deg_{\ell-1, d-1}^T (\rho) p_{\ell, d}.
\]
This proves the general statement of the lemma.

When the given model is the Linial-Meshulam model, the event $E$ is independent of the event $\sigma^k \cup \{v\} \in X_n \}$. Thus, the right-hand side of (\ref{eq: 3.2.10}) becomes
\[ \label{eq: 3.2.17} \tag{3.2.17}
\begin{aligned}
\frac{\Pr \lp \{\sigma_k \cup \{v\} \in X_n \} \cap E \: | \: T \subseteq X_n \rp}{\Pr \lp E \: | \: T \subseteq X_n \rp} & = \frac{\Pr \lp \{\sigma_k \cup \{v\} \in X_n \} \: | \: T \subseteq X_n \rp \Pr \lp E \: | \: T \subseteq X_n \rp}{\Pr \lp E \: | \: T \subseteq X_n \rp} \\
& = \Pr \lp \sigma_k \cup \{v\} \in X_n \: | \: T \subseteq X_n \rp \\
& = \frac{\lambda}{n}.
\end{aligned}
\]
This completes the proof of the second assertion of the lemma.
\end{proof}

The probability distribution of the backward exploration can be estimated as follows:

\begin{lemma}
\label{lemma: backward estimate}
Let $k \geq 0$ and $v$ be a vertex that can create a backward $d$-dimensional simplex a $(d-1)$-simplex $\sigma^k$ that is being explored at Step $k$. Then, the probability that $\sigma^k \cup \{v\}$ appears at Step $k$ is lower bounded by
\[ \label{eq: 3.2.18} \tag{3.2.18}
p_k^B := \bigg( 1- \sum_{\ell = 1}^{d-1} \sum_{\substack{\rho \subseteq \sigma^k, \\ \dim(\rho) = \ell-1}} \deg_{\ell-1, d-1}^{\mathcal{X}_{k-1}}(\rho) p_{d, \ell} \bigg) \lp \frac{\lambda}{n} + o(1) \rp.
\]
In particular, if the given MRSC model is the Linial-Meshulam model, the probability is exactly $\lambda/n + o(1)$.
\end{lemma}

\begin{proof}
Let $k \geq 0$ and let $T \subseteq \Delta_n$ be a $d$-dimensional simplicial complex $\Pr(\mathcal{X}^{k-1} = T) > 0$. The vertex $v$ satisfies $v \in S_0(T)$ but no other faces of $\sigma_k \cup \{v\}$ containing $v$ belongs to $T$. The probability we consider is the following:
\[ \label{eq: 3.2.19} \tag{3.2.19}
\Pr \lp \sigma^k \cup \{v\} \in X_n \: | \: \mathcal{X}_{k-1} = T \rp.
\]
Let us define the following two events $E'$ and $E''$:
\[ \label{eq: 3.2.20} \tag{3.2.20}
E' := \{ v \in \sigma_1 \} \cup \lp \bigcup_{1 \leq j \leq k-1} \{ \sigma_j \cup \{v\} \in X_n \} \rp
\]
and
\[ \label{eq: 3.2.21} \tag{3.2.21}
E'' := \bigcap_{0 \leq \ell \leq d-2} \: \: \bigcap_{\substack{1 \leq j \leq k-1, \\ \dim( \sigma_j \cap \sigma_k) = \ell}} \{\sigma_j \cup \{v\} \not\in X_n \}.
\]
In plain parlance, the event $E'$ is the event that the vertex $v$ belongs to the explored part $\mathcal{X}_{k-1}$ up to Step $(k-1)$ and the event $E''$ is the event that none of lower-dimensional-faces of $\sigma_k \cup \{v\}$ containing $v$ has appeared up to Step $(k-1)$. Similarly as we derived (\ref{eq: 3.2.11}), we can obtain
\[ \label{eq: 3.2.22} \tag{3.2.22}
\begin{aligned}
\Pr \lp \sigma^k \cup \{v\} \in X_n \: | \: \mathcal{X}^{k-1} = T \rp & = \frac{\Pr \lp \{\sigma^k \cup \{v\} \in X_n \} \cap E' \cap E'' \: | \: T \subseteq X_n \rp }{\Pr \lp E' \cap E'' | T \subseteq X_n \rp} \\
& \geq \Pr \lp \{\sigma^k \cup \{v\} \in X_n \} \cap E' \cap E'' | T \subseteq X_n \rp.
\end{aligned}
\]
Because $E'$ happens with probability $1$ conditional on $T \subseteq X_n$, the last expression in (\ref{eq: 3.2.22}) is equal to
\[ \label{eq: 3.2.23} \tag{3.2.23}
\Pr \lp \{\sigma^k \cup \{v\} \in X_n \} \cap E'' \: | \: T \subseteq X_n \rp,
\]
which is further equal to
\[ \label{eq: 3.2.24} \tag{3.2.24}
\Pr \lp E'' \: | \: \sigma^k \cup \{v\} \in X_n,  T \subseteq X_n \rp \Pr \lp \sigma^k \cup \{v\} \in X_n \: | \: T \subseteq X_n \rp 
\]
Recall that the latter term of (\ref{eq: 3.2.24}) equals $\lambda/n + o(1)$ by the assumption (\ref{eq: 1.3.2 assumption 2}). Using the union bound, the former term can be estimated as follows:
\[ \label{eq: 3.2.25} \tag{3.2.25}
\begin{aligned}
\Pr \lp E'' \: | \: \sigma^k \cup \{v\} \in X_n,  T \subseteq X_n \rp & = 1 - \Pr \lp (E'')^c \: | \: \sigma_k \cup \{v\} \in X_n, T \subseteq X_n \rp \\
& \geq 1 - \sum_{\ell = 0}^{d-2} \sum_{\substack{1 \leq j \leq k-1, \\ \dim(\sigma_j \cap \sigma_k) = \ell}} \Pr \lp \sigma_j \cup \{v\} \in X_n \: | \: \sigma_k \cup \{v\} \in X_n, T \subseteq X_n \rp \\
& = 1 - \sum_{\ell = 1}^{d-1} \sum_{\substack{1 \leq j \leq k-1, \\ \dim(\sigma_j \cap \sigma_k) = \ell-1}} \Pr \lp \sigma_j \cup \{v\} \in X_n \: | \: \sigma_k \cup \{v\} \in X_n, T \subseteq X_n \rp.
\end{aligned}
\]
Similarly as we derived (\ref{eq: 3.2.16}) from (\ref{eq: 3.2.13}), we can obtain that the last expression of (\ref{eq: 3.2.25}) is lower bounded by
\[ \label{eq: 3.2.26} \tag{3.2.26}
1 - \sum_{\ell = 1}^{d-1} \sum_{\substack{\rho \subseteq \sigma_k, \\ \dim(\rho) = \ell - 1}} \deg_{\ell -1, d-1}^T(\rho) p_{d, \ell}.
\]
This completes the proof of the general statement of the lemma.

When the given model is the Linial-Meshulam model, the event $E''$ is independent of the event $\{\sigma^k \cup \{v\} \in X_n\}$, which proves the result.
\end{proof}

The key takeaway from Lemma \ref{lemma: forward estimate} and \ref{lemma: backward estimate} is the following:

\begin{cor}
\label{cor: probability of exploration processes}
Let $X_n \sim \MRSC_d(n; \bm{p})$ with the assumptions (\ref{eq: 1.3.1 assumption 1}) and (\ref{eq: 1.3.2 assumption 2}). Then, conditional on $\mathcal{X}_{k-1}$, the number of new forward $(d-1)$-simplices at each step $k$, denoted by $F_k$, satisfies
\[ \label{eq: 3.2.27} \tag{3.2.27}
F_k \geqst d \mathsf{Bin} \lp n - s_0(\mathcal{X}_{k-1}), p_k^F \rp,
\]
and the number of new backward $(d-1)$-simplices at each step $k$, denoted by $B_k$, satisfies
\[ \label{eq: 3.2.28} \tag{3.2.28}
B_k \geqst d \mathsf{Bin} \lp s_0(\mathcal{X}_{k-1}) - C \sum_{\ell = 1}^{d-1} \max_{\rho \in S_{\ell -1}(\mathcal{X}_{k-1}} \deg_{\ell -1, d-1}^{\mathcal{X}_{k-1}}(\rho), p_k^B \rp,
\]
where the constant $C > 0$ only depends on $d$. Furthermore, conditional on $\mathcal{X}_{k-1}$, $F_k$ and $B_k$ are independent, and it holds that
\[ \label{eq: 3.2.29} \tag{3.2.29}
F_k + B_k \geqst d  \mathsf{Bin} \lp n - C \sum_{\ell = 1}^{d-1} \max_{\rho \in S_{\ell - 1}(\mathcal{X}_{k-1})} \deg_{\ell - 1, d-1}^{\mathcal{X}_{k-1}}(\rho), \min\{p_k^F, p_k^B\} \rp.
\]
\end{cor}

On the other hand, conditional on $\mathcal{X}_{k-1}$, we have the following trivial stochastic upper bound for $E_k$: 
\[ \label{eq: 3.2.30} \tag{3.2.30}
E_k \leq_{st} d \mathsf{Bin} \lp n, \frac{\lambda}{n} \rp.
\]

\subsection{The local-weak convergence in probability of MRSCs}
\label{subsec: LWC of MRSC}

In this subsection, we establish the local-weak convergence in probability of the sequence of random simplicial complexes $X_n \sim \MRSC_d(n; \bm{p})$ satisfying the assumptions (\ref{eq: 1.3.1 assumption 1}) and (\ref{eq: 1.3.2 assumption 2}). To state it, we describe a special type of $(d-1)$ rooted random simplicial complex called the \textit{$(d-1)$-rooted Poisson tree} which is the local-weak limit of $(X_n)_{n \geq 1}$. 

\begin{defi}[$(d-1)$-rooted Poisson trees]
\label{def: Poisson tree}
A $(d-1)$-\emph{rooted Poisson tree} with parameter $\lambda > 0$ is a $(d-1)$-rooted random $d$-dimensional simplicial complex $(X, \tau)$ constructed as follows: Start with a root $(d-1)$-simplex $\tau$. At each step $k = 0, 1, \dots$, to every $(d-1)$-simplex $\tau'$ of distance $k$ from $\tau$, add $\mathsf{Poi}(\lambda)$ new vertices $v_1, \dots, v_m$ and the $d$-dimensional simplices $\tau' \cup \{v_1\}, \dots, \tau' \cup \{v_m\}$ along with all their lower-dimensional simplices to the simplicial complex constructed in the previous step. The process is independent across different $(d-1)$-simplices at each step. 
\end{defi}

Now, we are ready to state the local-weak convergence in probability of $X_n \sim \MRSC_d(n; \bm{p})$. 

\begin{thm}
\label{thm: LWC in Prob of MRSC}
Suppose $X_n \sim X_d(n; \bm{p})$ satisfies Assumption~\ref{assumption: prob parameters}. Then, as $n\to\infty$, 
\[ 
X_n \to (X, \tau) \; \text{locally weakly in probability},
\]
where $(X, \tau)$ is a $(d-1)$-rooted Poisson tree with parameter $\lambda$.
\end{thm}

As mentioned earlier in Section \ref{sec: prelim}, local-weak convergence in probability is a stronger concept than local-weak convergence. The local-weak convergence of random simplicial complexes following $\LM(n, \lambda/n)$ (\cite{Linial;Peled:2016}) or $\MRSC_c(n; \bm{p})$ under Assumption \ref{assumption: prob parameters} (\cite{Kanazwa:2022}) was already established in the literature. However, the present paper is the very first that explicitly establishes their local-weak convergence in probability.

Recall the definition of local-weak convergence in probability (\ref{eq: 2.1.3 local conv}). First we are going to show that the expectation of the left-hand side of (\ref{eq: 2.1.3 local conv}) converges to the right-hand side, and then we are going to show that the variance of it tends to $0$. By summing these two up, we can establish the desired convergence in probability. Since the proposed limiting object is a $(d-1)$-rooted Poisson tree, it suffices to prove the validity of the statement (\ref{eq: 2.1.3 local conv}) for an arbitrary $(d-1)$-rooted tree $T$ with finite depth. In fact, it turns out that the limiting behaviour of the local neighborhood at each $(d-1)$-simplex is governed by the forward exploration started at each simplex. 

Before discussing the detailed proof, we would like to remark that the total number of $(d-1)$-simplices in $X_n \sim \MRSC_d(n; \bm{p})$ is concentrated around its expected value $\binom{n}{d} q_d$. The following formal result was introduced in \cite[Lemma 23]{Kanazwa:2022}. 

\begin{lemma}
\label{lemma: total number of simplices}
Under the condition (\ref{eq: 1.3.1 assumption 1}) in Assumption \ref{assumption: prob parameters}, for any $r \in [1, \infty)$ it holds that
\[
\lim_{n \to \infty} \E \lb \: \left| \frac{s_{d-1} \lp X_n \rp}{n^d q_d} - \frac{1}{d!} \right|^r \: \rb = 0.
\]
In particular, $\lim_{n \to \infty} \Pr \lp s_{d-1} \lp X_n \rp > 0 \rp = 1$.
\end{lemma}

\begin{proof}[Proof of Theorem \ref{thm: LWC in Prob of MRSC} ] 
The proof is broken into four steps.

\vspace{.1in}

\noindent \textbf{Step 1: Preliminary setup.} Let $T$ be a fixed $(d-1)$-rooted tree with $s_{d-1}(T) = t$ and $S_0(T) \subseteq [n]$. Let $\rho_0$ denote the minimal element in $S_{d-1}(T)$ in the linear ordering induced by the lexicographic ordering on $S_{d-1}(\Delta_n)$. Let $(T_k)_{1 \leq k \leq t}$ be the sequence of deterministic simplicial complexes obtained by the breadth-first exploration started at $\rho_0$ and let $(\\rho_k)_{k=1}^t$ denote the sequence of the $(d-1)$-simplices in $T$ whose ordering is obtained by the breadth-first exploration on $T$ started at $\rho_0$. In other words, each $\rho_k$ is the $(d-1)$-simplex that is explored at Step $k$ in the exploration $(T_k)_{1 \leq k \leq t}$. Additionally, for each $1 \leq k \leq t$, let $x_k$ denote the number of newly detected $d$-simplices at Step $k$ of the exploration $(T_k)_{1 \leq k \leq t}$.

For convenience, let us fix $\sigma_0 := \{1, 2, \dots, d\}$, the minimal element in $S_{d-1}(\Delta_n)$. Without loss of generality, we analyze the neighborhood of $\sigma_0$ in $X_n$ conditional on $\sigma_0 \in X_n$. Let $(\mathcal{X}_k)_{k \geq 1}$ be the sequence of random simplicial complexes obtained by the breadth-first exploration of $X_n$ started at $\sigma_0$, which was described in Section~\ref{subsec: breadth-first exploration}. Recall that $(\sigma_k)_{k \geq 1}$ denote the $(d-1)$-simplces explored at each Step $k$ in the exploration $(\mathcal{X}_k)_{k \geq 1}$; here, in particular, $\sigma_1 = \sigma_0$. Additionally, recall the notations $\mathcal{F}_{d, k}, \mathcal{B}_{d, k}$, and $\mathcal{H}_{d, k}$ for the set of newly detected forward, backward, and sibling $d$-simplces at each step $k$, respectively.

Our goal is to prove that, for every fixed radius $r \in \N \cup \{0\}$, 
\[ \label{eq: 3.3.1} \tag{3.3.1}
\begin{aligned}
p^{(X_n)}(T) & := \frac{1}{s_{d-1}(X_n)} \sum_{\sigma \in S_{d-1}(X_n)} \1 \left\{ B_{X_n}(\sigma; r) \cong T \right\} \\
& \overset{\Pr}{\to} \mu \lp B_X(\tau; r) \cong T \rp
\end{aligned}
\]
as $n \to \infty$. Let $\#(T)$ denote the number of different ways to choose breadth-first orderings of $T$ and let $\bar{T}$ be the tree $T$ with fixed order given by the exploration starting at $\rho_0$. Since $\mu$ is the probability law of the $(d-1)$-rooted Poisson tree with parameter $\lambda$, we have
\[ \label{eq: 3.3.2} \tag{3.3.2}
\mu \lp B_X(\tau; r) \cong T \rp = \# (T) \mu \lp \bar{B}_X(\tau; r) = \bar{T} \rp = \#(T) \prod_{k \in [t] \: : \: \mathrm{dist}(\tau, \rho_k) < r} e^{-\lambda} \frac{\lambda^{x_k}}{x_k !},
\]
where $\bar{B}_X(\tau; r)$ is the neighborhood $B_X(\tau; r)$ equipped with the ordering of its $(d-1)$-simplices given by the breadth-first exploration of it starting at $\tau$.

\vspace{.1in}

\noindent\textbf{Step 2: First moment.} Write
\[ \label{eq: 3.3.3} \tag{3.3.3}
N_{n,r}(T) := \sum_{\sigma \in S_{d-1}(X_n)} \1 \left\{ \bar{B}_{X_n}(\sigma; r) = \bar{T} \right\},
\]
where $\bar{B}_{X_n}(\sigma; r)$ is the neighborhood $B_{X_n}(\sigma; r)$ equipped with the ordering of its $(d-1)$-simplices given by the breadth-first exploration on it starting at $\sigma$. Then, the expectation of $N_{n,r}(T)$ can be written as
\[ \label{eq: 3.3.4} \tag{3.3.4}
\begin{aligned}
\E \lb  N_{n,r}(T) \rb & = \E \lb \E \lb  N_{n,r}(T) \: | \: s_{d-1}(X_n) \rb \rb \\
& = \E \lb s_{d-1}(X_n) \E \lb  \1 \left\{ \bar{B}_{X_n}(\sigma_0; r) = \bar{T} \right\} \: | \: s_{d-1}(X_n) \rb \rb.
\end{aligned}
\]
Here, the second equality is obtained by the observation that the probability law of the neighborhood $\bar{B}_{X_n}(\sigma; r)$ of any $\sigma \in S_{d-1}(X_n)$ is equal to that of $\sigma_0$, the minimal element in $S_{d-1}(\Delta_n)$. From this, we obtain
\[ \label{eq: 3.3.5} \tag{3.3.5}
\begin{aligned}
\frac{d!}{n^{d} q_d} \E \lb N_{n,r}(T) \rb & =  \E \lb \frac{d! s_{d-1}(X_n)}{n^d q_d} \E \lb  \1 \left\{ \bar{B}_{X_n}(\sigma_0; r) = \bar{T} \right\} \: | \: s_{d-1}(X_n) \rb \rb \\
& = \E \lb \lb \frac{d! s_{d-1}(X_n)}{n^d q_d} -1 + 1 \rb \E \lb  \1 \left\{ \bar{B}_{X_n}(\sigma_0; r) = \bar{T} \right\} \: | \: s_{d-1}(X_n) \rb \rb \\
& = \E \lb \E \lb \1 \left\{ \bar{B}_{X_n}(\sigma_0; r) = \bar{T} \right\} \: | \: s_{d-1}(X_n) \rb \rb \\
& + \E \lb \lb \frac{d! s_{d-1} (X_n)}{n^d q_d} - 1 \rb \E \lb \1 \left\{ \bar{B}_{X_n}(\sigma_0; r) = \bar{T} \right\} \: | \: s_{d-1}(X_n) \rb \rb.
\end{aligned}
\]
The first term in the last expression in (\ref{eq: 3.3.5}) is equal to 
\[
\Pr \lp \bar{B}_{X_n}(b_1; r) = \bar{T} \rp.
\]
The latter term is bounded above by
\[
\E \lb \left| \frac{d! s_{d-1}(X_n)}{n^d q_d} - 1 \right| \rb,
\]
which tends to $0$ by Lemma~\ref{lemma: total number of simplices}.  Therefore, we have
\[ \label{eq: 3.3.6} \tag{3.3.6}
\E \lb N_{n,r}(T) \rb = \lp 1 + o(1) \rp \frac{n^d q_d}{d!} \Pr \lp \bar{B}_{X_n}(b_1; r) = \bar{T} \rp.
\]

Next, we analyze the probability $\Pr \lp \bar{B}_{X_n}(b_1; r) = \bar{T} \rp$. Let $F_{k,d} := |\mathcal{F}_{d,k}|$ if $\mathrm{dist}(\sigma_0, \sigma_k) < r$, and $F_{k,d} : = 0$ otherwise. Let $D_{k,d} := |\mathcal{B}_{d,k} \cup \mathcal{H}_{d,k}|$. Then, 
\[ \label{eq: 3.3.7} \tag{3.3.7}
\begin{aligned}
\Pr \lp \bar{B}_{X_n}(\sigma_0; r) = \bar{T} \rp & = \Pr \lp (F_{k,d}, D_{k,d}) = (x_k, 0) \text{ for all } k \in [t] \rp \\
& = \Pr \lp (F_{[k],d}, D_{[k],d}) = (x_{[k]}, 0_{[k]}) \text{ for all } k \in [t] \rp \\
& = \prod_{k=1}^t \Pr \lp (F_{k,d}, D_{k,d}) = (x_k, 0) \: | \: (F_{[k-1],d}, D_{[k-1],d}) = (x_{[k-1]}, 0_{[k-1]}) \rp,
\end{aligned}
\]
where we use the abbreviation $x_{[k]} = (x_1, \dots, x_k)$ ($F_{[k],d}$, $D_{[k],d}$, and $0_{[k]}$ likewise) for brevity. By Lemma \ref{lemma: forward estimate}, the probability that each forward $d$-simplex appears at Step $k$ is bounded above by $\lambda/n$ and below by $p_k^F$. Since $T$ is a fixed finite graph, $k$ and all the degree terms in the expression of $p_k^F$~(\ref{eq: forward prob 3.2.3}) are $O(1)$ as $n \to \infty$. Thus, we have
\[ \label{exp: 3.3.8} \tag{3.3.8}
\mathsf{Bin} \lp n_k, (1 - o(1)) \frac{\lambda}{n} \rp \leqst F_{d,k} \leqst \mathsf{Bin} \lp n_k, (1 + o(1))\frac{\lambda}{n} \rp,
\]
where $n_k = n - s_0(T_{k-1})$, i.e., $n_k$ denotes the number of vertices in $[n]$ that have not been detected until Step $(k-1)$ of the exploration process of $T$. On the other hand, the probability that each backward or sibling $d$-simplex appears is always $o(1)$ as $n \to \infty$. This rough bound gives us
\[ \label{eq: 3.3.9} \tag{3.3.9}
D_{d,k} \leqst \mathsf{Bin} \lp s_0(T_{k-1}), o(1) \rp.
\]
Again, recall that $s_0(T_{k-1}) = O(1)$ as $n \to  \infty$ since we fixed a finite simplicial complex $T$. Additionally, conditional on $\mathcal{X}_{k-1}$, $F_{d,k}$ and $D_{d,k}$ are independent. Therefore,
\[ \label{eq: 3.3.10} \tag{3.3.10}
\begin{aligned}
\Pr \lp \bar{B}_{X_n}(\sigma_0; r) = \bar{T} \rp & \sim \prod_{k\in [t] \: : \: \mathrm{dist}(\sigma_0, \sigma_k) < r} \Pr \lp \mathsf{Bin}(n_k, \lambda/n) = x_k \rp \times \prod_{k \in [t]} \lp 1 - \lambda/n \rp^{\sum_{j=0}^{k-1} x_j} \\
& \to \prod_{k \in [t] \: : \: \mathrm{dist}(\sigma_0, a_k) < r } e^{-\lambda} \frac{\lambda^{x_k}}{x_k !} \\
& = \mu \lp \bar{B}_X(\tau; r) = \bar{T} \rp.
\end{aligned}
\]
Here the symbol $\sim$ means that the ratio of the both sides of it tends to $1$ as $n \to \infty$. To sum up, we have obtained that
\[ \label{eq: 3.3.11} \tag{3.3.11}
\frac{d!}{n^d q_d} \E \lb N_{n,r}(T) \rb \to \mu \lp \bar{B}_X(\tau; r) = \bar{T} \rp.
\]

\vspace{.1in}

\noindent \textbf{Step 3: Second moment.} We have
\[ \label{eq: 3.3.12} \tag{3.3.12}
\begin{aligned}
& \E \lb N_{n,r}(T)^2 \rb  \\
& = \E \lb \sum_{\sigma \in S_{d-1}(X_n)} \1 \left\{ \bar{B}_{X_n}(\sigma; r) = \bar{T} \right\} + \sum_{\sigma \neq \sigma'} \1 \left\{ \bar{B}_{X_n}(\sigma; r) = \bar{T}, \bar{B}_{X_n}(\sigma'; r) = \bar{T} \right\} \rb.
\end{aligned}
\]
The first term in the right-hand side of (\ref{eq: 3.3.12}) satisfies
\[ \label{eq: 3.3.13} \tag{3.3.13}
\begin{aligned}
\E \lb \sum_{\sigma \in S_{d-1}(X_n)} \1 \{ \bar{B}_{X_n}(\sigma; r) = \bar{T} \} \rb & = \E \lb N_{n,r}(T) \rb \\
& = \frac{n^d q_d}{d!} \mu \lp \bar{B}_X(\tau; r) = \bar{T} \rp (1 + o(1)).
\end{aligned}
\]
by \textbf{Step 1}.

The latter term can be written as follows:
\[ \label{eq: 3.3.14} \tag{3.3.14}
\sum_{\ell = 0}^{d-2} \E \lb  \sum_{\sigma \neq \sigma'} \1 \{ \bar{B}_{X_n}(\sigma; r) = \bar{T}, \bar{B}_{X_n}(\sigma'; r) = \bar{T} \} \1\{|\sigma \cap \sigma'| = \ell \} \rb.
\]
For each $0 \leq \ell \leq d-1$, let $\sigma_{0, \ell}$ be the minimal element in $S_{d-1}(\Delta_n)$ that has $\ell$ common elements with $\sigma_0$. Then, we obtain
\[ \label{eq: 3.3.15} \tag{3.3.15}
\begin{aligned}
& \E \lb  \sum_{\sigma \neq \sigma'} \1 \{ \bar{B}_{X_n}(\sigma; r) = \bar{T}, \bar{B}_{X_n}(\sigma'; r) = \bar{T} \} \1\{|\sigma \cap \sigma'| = \ell \} \rb \\
& = \E \lb \E \lb  \sum_{\sigma \neq \sigma'} \1 \{ \bar{B}_{X_n}(\sigma; r) = \bar{T}, \bar{B}_{X_n}(\sigma'; r) = \bar{T} \} \1\{|\sigma \cap \sigma'| = \ell \} \: | \: s_{d-1}(X_n) \rb \rb \\
& = \E \lb s_{d-1}(X_n) \binom{d}{\ell} \binom{n}{d - \ell} \Pr \lp \bar{B}_{X_n}(\sigma_0; r) = \bar{T}, \bar{B}_{X_n}(\sigma_{0,\ell}; r) = \bar{T} \rp \rb.
\end{aligned}
\]
By Lemma~\ref{lemma: total number of simplices}, using a similar argument as that used in (\ref{eq: 3.3.5}), we can show that the last expression of (\ref{eq: 3.3.15}) is equal to 
\[ \label{eq: 3.3.16} \tag{3.3.16}
\lp 1 + o(1) \rp \frac{n^d q_d}{d!} \binom{d}{\ell} \binom{n}{d-\ell}  \Pr \lp \bar{B}_{X_n}(\sigma_0; r) = \bar{T}, \bar{B}_{X_n}(\sigma_{0,\ell}; r) = \bar{T} \rp.
\]

We will show that for every $0 \leq \ell \leq d-1$,
\[ \label{eq: 3.3.17} \tag{3.3.17}
\Pr \lp \bar{B}_{X_n}(\sigma_0; r) = \bar{T}, \bar{B}_{X_n}(\sigma_{0,\ell}; r) = \bar{T} \rp \to \mu \lp \bar{B}_{X}(\tau; r) = \bar{T} \rp^2.
\]
Let us split the left-hand side of the above as follows:
\[ \label{eq: 3.3.18} \tag{3.3.18}
\begin{aligned}
\Pr \lp \bar{B}_{X_n}(\sigma_0; r) = \bar{T}, \bar{B}_{X_n}(\sigma_{0,\ell}; r) \rp & = \Pr \lp \bar{B}_{X_n}(\sigma_0; r) = \bar{T}, \bar{B}_{X_n}(\sigma_{0,\ell}; r) = \bar{T}, \mathrm{dist}_{X_n} (\sigma_0, \sigma_{0,\ell}) > 2r \rp \\
& + \Pr \lp \bar{B}_{X_n}(\sigma_0; r) = \bar{T}, \bar{B}_{X_n}(\sigma_{0,\ell}; r) = \bar{T}, \mathrm{dist}_{X_n} (\sigma_0, \sigma_{0,\ell}) \leq 2r \rp
\end{aligned}
\]
The first term in the right-hand side of (\ref{eq: 3.3.18}) is trivially upper bounded by
\[ \label{eq: 3.3.19} \tag{3.3.19}
\Pr \lp \mathrm{dist}_{X_n}(\sigma_0, \sigma_{0,\ell}) \leq 2r \rp
\]
Recall that the number of $(d-1)$-simplices that share $\ell$ vertices with $\sigma_0$ is $\binom{d}{\ell} \binom{n}{d - \ell}$. Because ehy are all exchangeable, the expression in (\ref{eq: 3.3.19}) is upper bounded by
\[ \label{eq: 3.3.20} \tag{3.3.20}
\frac{1}{\binom{d}{\ell} \binom{n}{d - \ell}} \E \lb s_{d-1} \lp B_{X_n}(\sigma_0; 2r) \rp \rb. 
\]
Recall that the number of newly detected $(d-1)$-simplices at each step (the sum of forward,  backward, and sibling $(d-1)$-simplices) is dominated by $d \mathsf{Bin}(n, \lambda/n)$, conditional on each previous step. Hence, the expectation $\E\lb s_{d-1}(B_{X_n}(\sigma_0; 2r)) \rb$ is upper bounded by the expected number of total progeny in a branching process with progeny distribution $d \mathsf{Bin}(n, \lambda/n)$. This means that $\E \lb s_{d-1}(B_{X_n}(\sigma_0; 2r) ) \rb \sim \sum_{j = 0}^{2r} (d \lambda)^{j}$ as $n \to \infty$. Since $r$ is fixed, this quantity is finite. Thus, we have
\[ \label{eq: 3.3.21} \tag{3.3.21}
\Pr \lp \mathrm{dist}_{X_n}(\sigma_0, \sigma_{0, \ell}) \leq 2r \rp = O \lp n^{-(d - \ell)} \rp,
\]
which implies that 
\[ \label{eq: 3.3.22} \tag{3.3.22}
\Pr \lp \mathrm{dist}_{X_n}(\sigma_0, \sigma_{0, \ell}) \leq 2r \rp \to 0
\]
as $n \to \infty$ for every $0 \leq \ell \leq d-1$.

As for the first term on the right-hand side of (\ref{eq: 3.3.18}), let us express it as
\[ \label{eq: 3.3.23} \tag{3.3.23}
\Pr \lp  \bar{B}_{X_n}(\sigma_{0, \ell}; r) = \bar{T}, \mathrm{dist}_{X_n} (\sigma_0, \sigma_{0, \ell}) > 2r  \: | \: \bar{B}_{X_n}(\sigma_0; r) = \bar{T} \rp \Pr \lp \bar{B}_{X_n}(\sigma_0; r) = \bar{T} \rp. 
\]
The second term on the right-hand side of (\ref{eq: 3.3.23}) converges to $\mu \lp \bar{B}_{X}(\tau; r) = \bar{T} \rp$ by (\ref{eq: 3.3.10}). The first term can be further written as
\[ \label{eq: 3.3.24} \tag{3.3.24}
\begin{aligned}
& \Pr \lp \bar{B}_{X_n}(\sigma_{0, \ell}; r) = \bar{T} \: | \: \mathrm{dist}_{X_n}(\sigma_0, \sigma_{0, \ell}) > 2r, \bar{B}_{X_n}(\sigma_0; r) = \bar{T} \rp \\
& \qquad \qquad \qquad \qquad \qquad \times \Pr \lp \mathrm{dist}_{X_n}(\sigma_0, \sigma_{0, \ell}) > 2r \: | \: \bar{B}_{X_n}(\sigma_0; r) = \bar{T} \rp.
\end{aligned}
\]
Conditional on $\bar{B}_{X_n}(\sigma_0; r) = \bar{T}$, the event $\mathrm{dist}_{X_n}(\sigma_0, \sigma_{0, \ell}) > 2r$ took place if and only if $\sigma_{0, \ell}$ does not belong to $B_{X_n}(\sigma_0; 2r)$ given that $\bar{B}_{X_n}(\sigma_0; r) = \bar{T}$. Similarly as discussed in the derivation of (\ref{eq: 3.3.22}), we can show that $s_{d-1}(B_{X_n}(\sigma_0; 2r))$ is finite w.h.p. as $n \to \infty$. Hence, 
\[ \label{eq: 3.3.25} \tag{3.3.25}
\Pr \lp \mathrm{dist}_{X_n}(\sigma_0, \sigma_{0, \ell}) > 2r \: | \: \bar{B}_{X_n}(\sigma_0; r) = \bar{T} \rp \to 1
\]
as $n \to \infty$. Finally, conditional on the event $\{\mathrm{dist}_{X_n}(\sigma_0, \sigma_{0, \ell}) > 2r\} \cap \{ \bar{B}_{X_n}(\sigma_0; r) = \bar{T}) \}$, the probability that $\bar{B}_{X_n}(\sigma_{0, \ell}; r) = \bar{T}$ is equal to the probability that the neighborhood $B_{X_n}(\sigma_{0, \ell}; r)$ is isomorphic to $T$ conditional on that the breadth-first exploration of $B_{X_n}(\sigma_{0, \ell}; r)$ only uses vertices that do not create any $d$-simplices adjacent to any $(d-1)$-simplices in $T$. Since $T$ is a fixed finite simplicial complex, the number of vertices used in this exploration is $n - s_0(T)$, which is asymptotically equivalent to $n$. This gives us
\[ \label{eq: 3.3.26} \tag{3.3.26}
\begin{aligned}
\Pr \lp \bar{B}_{X_n}(\sigma_{0, \ell}; r) = \bar{T} \: | \: \mathrm{dist}_{X_n}(\sigma_0, \sigma_{0, \ell}) > 2r, \bar{B}_{X_n}(\sigma_0; r) = \bar{T} \rp & = \Pr \lp \bar{B}_{X_n}(\sigma_{0, \ell}) = \bar{T} \rp + o(1) \\
& \to \mu \lp \bar{B}_{X}(\tau; r) = \bar{T} \rp
\end{aligned}
\]
as $n \to \infty$. Hence, the expression in (\ref{eq: 3.3.16}) is equal to
\[ \label{eq: 3.3.27} \tag{3.3.27}
\lp 1 + o(1) \rp \frac{n^d q_d}{d!} \binom{d}{\ell} \binom{n}{d - \ell} \lp o(1) + \mu \lp \bar{B}_X(\tau; r) = \bar{T} \rp^2 \rp.
\]
Plugging this back to the summation in (\ref{eq: 3.3.14}), we obtain that the expression is asymptotically equivalent to
\[ \label{eq: 3.3.28} \tag{3.3.28}
\sum_{\ell =0 }^{d-2} \frac{n^d q_d}{d!} \binom{d}{\ell} \binom{n}{d - \ell} \mu \lp \bar{B}_X(\tau; r) = \bar{T} \rp^2.
\]
This concludes that
\[ \label{eq: 3.3.29} \tag{3.3.29}
\E \lb N_{n,r}(T)^2 \rb = \E \lb N_{n,r}(T) \rb + \sum_{\ell = 0}^{d-1} \lp 1 + o(1) \rp \frac{n^d q_d}{d!} \binom{d}{\ell} \binom{n}{d-\ell} \mu \lp \bar{B}_X(\tau; r) = \bar{T} \rp^2 \lp 1 + o(1) \rp.
\]

\vspace{.1in}

\noindent \textbf{Step 4: Completion of the proof.} Combining (\ref{eq: 3.3.11}) and (\ref{eq: 3.3.29}), we obtain 
\[ \label{eq: 3.3.30} \tag{3.3.30}
\frac{\V \lp N_{n,r}(T) \rp}{\E \lb N_{n,r}(T) \rb^2} \to 0.
\]
Thus, by the Chebyshev inequality, we can conclude that
\[ \label{eq: 3.3.31} \tag{3.3.31}
\frac{N_{n,r}(T)}{\E \lb N_{n,r}(T) \rb} \overset{\Pr}{\to} 1.
\]
Accordingly, 
\[ \label{eq: 3.3.32} \tag{3.3.32}
\frac{N_{n,r}(T)}{s_{d-1} (X_n)} \overset{\Pr}{\to} \mu \lp \bar{B}_X(\tau; r) = \bar{T} \rp,
\]
which implies that
\[ \label{eq: 3.3.33} \tag{3.3.33}
p^{(X_n)}(T) = \#(T) \frac{N_{n,r}(T)}{s_{d-1} (X_n)} \overset{\Pr}{\to} \#(T) \mu \lp \bar{B}_X(\tau; r) = \bar{T} \rp = \mu \lp B_X(\tau; r) \cong T \rp,
\]
This completes the proof.
\end{proof}

As a result, we obtain the following convergence result of the number of $d$-dimensional connected components in $X_n$:

\begin{prop}
\label{prop: number of components in MRSC}
Let $X_n \sim \MRSC_d(n; \bm{p})$ satisfying Assumption~\ref{assumption: prob parameters}. Let $C_n$ denote the number of $d$-dimensional connected components in $X_n$. Then, it holds that
\[
\frac{C_n}{s_{d-1}(X_n)} \overset{\Pr}{\to} \gamma_{\lambda} - \lambda \frac{d}{d+1} \gamma_{\lambda}^{d+1},
\]
where $\gamma_{\lambda}$ is the extinction probability of a branching process with offspring distribution $d \mathsf{Poi}(\lambda)$, i.e., $\gamma_{\lambda}$ is the smallest solution to $\gamma = e^{-\lambda(1 - \gamma^d)}$.
\end{prop}

\begin{proof}[Proof of Proposition~\ref{prop: number of components in MRSC}]

The proposed results immediately follows Theorem~\ref{thm: LWC in Prob of MRSC} and Proposition~\ref{prop: number of components}. By Theorem~\ref{thm: LWC in Prob of MRSC}, we know that $(X_n)_{n \geq 1}$ converges locally weaky in probability to $(d-1)$-rooted Poisson tree $(X, \tau)$ with parameter $\lambda$. Thus, it suffices to calculate
\[ \label{eq: 3.3.34} \tag{3.3.34}
\E \lb \frac{1}{s_{d-1} \lp \mathcal{C}(\tau) \rp} \rb.
\]
Note that the probability law of $s_{d-1} \lp \mathcal{C}(\tau) \rp$ is equal to that of the total progeny of a branching process with offspring distribution $d \mathsf{Poi}(\lambda)$. $\mathcal{C}$ denote the total progeny the branching process. Let $G(z)$ be the probability generating function of $\mathcal{C}$, defined by
\[ \label{eq: 3.3.35} \tag{3.3.35}
G(z) := \E \lb z^{\mathcal{C}}\1 \{ \mathcal{C} < \infty\} \rb \quad z \in [0, 1].
\]
We can write
\[ \label{eq: 3.3.36} \tag{3.3.36}
\begin{aligned} 
G(z) = \E \lb z^{1 + \sum_{i=1}^{Z_1} \mathcal{C}_i} \1\{ \mathcal{C} < \infty \} \rb = z \E \lb \E \lb \prod_{i=1}^{Z_1} z^{\mathcal{C}_i} \1\{ \mathcal{C}_i < \infty \} \: | \: Z_1 \rb \rb & = z \E \lb G(z)^{Z_1} \rb \\
& = z g(G(z)),
\end{aligned}
\]
where $g(\gamma) := e^{\lambda (\gamma^d -1)}$ is the probability generating function of $Z_1 \sim d \mathsf{Poi}(\lambda)$.Thus, we have
\[ \label{eq: 3.3.37} \tag{3.3.37}
G(z) = z e^{\lambda (G(z)^d - 1)}.
\]
Note that $G(1) = \Pr (\mathcal{C} < \infty) =: \gamma_{\lambda}$ is the extinction probability of the branching process, which is the smallest solution to $\gamma = g(\gamma)$.

Next, we express $\E[1/\mathcal{C}]$ in terms of $G(z)$. We can write
\[ \label{eq: 3.3.38} \tag{3.3.38}
\E \lb \frac{1}{\mathcal{C}} \rb = \sum_{n=1}^{\infty} \frac{1}{n} \Pr \lp \mathcal{C} = n \rp = \sum_{n=1}^{\infty} \Pr \lp \mathcal{C} = n \rp \int_{t=0}^1 t^{n-1} \d t.
\]
Here, the last equality is given by the equation 
\[
\frac{1}{n} = \int_{t=0}^1 t^{n-1} \d t
\]
for every $n \in \N$. By the monotone convergence theorem, the rightmost expression of (\ref{eq: 3.3.38}) is equal to
\[ \label{eq: 3.3.39} \tag{3.3.39}
\int_{t=0}^1 \sum_{n=1}^{\infty} \Pr \lp \mathcal{C} = n \rp t^{n-1} \d t = \int_{t=0}^1 \frac{1}{t} \sum_{n=1}^{\infty} \Pr \lp \mathcal{C} = n \rp t^n \d t = \int_{t=0}^1 \frac{G(t)}{t} \d t.
\]

To calculate the integral in the rightmost expression in (\ref{eq: 3.3.39}), let us apply the change of variable formula given by $u = G(t)$. Then, we obtain
\[ \label{eq: 3.3.40} \tag{3.3.40}
\int_{t=0}^1 \frac{G(t)}{t}  \d t = \int_{u=0}^{\gamma_{\lambda}} (1 - \lambda d u^d) \d u = \gamma_{\lambda} - \lambda \frac{d}{d+1} \gamma_{\lambda}^{d+1}.
\]
This completes the proof.

\end{proof}

According to Proposition~\ref{prop: upper bound for the giant}, we obtain the following upper bound for $s_{d-1} \lp \mathcal{C}_{\max} \rp$ immediately. 

\begin{prop}
\label{prop: MRSC giant size upper bound}
Let $X_n \sim \MRSC_d(n; \bm{p})$ satisfying Assumption \ref{assumption: prob parameters}. Then, it holds that for every $\e > 0$,
\[
\Pr \lp s_{d-1} \lp \mathcal{C}_{\max} \rp \leq s_{d-1} (X_n) (\zeta_{\lambda} + \e) \rp \to 1
\]
as $n \to \infty$, where $\zeta_{\lambda}$ is the survival probability of the $(d-1)$-rooted Poisson tree with parameter $\lambda$.
\end{prop}

\section{Proofs of the main results}
\label{sec: proof}

In this section, we prove our main results about the phase transition of the giant $2$-dimensional components in $X_n \sim \MRSC_2(n; \bm{p}_n)$. 

\subsection{Concentration in the supercritical case: Proof of Theorem \ref{thm: supercritical MRSC}}

To prove the concentration of $s_{d-1}(\mathcal{C}_{\max})$, we need to prove that $X_n \sim \MRSC_d(n; \bm{p})$ satisfies the property (\ref{eq: 2.2.3}) in Proposition~\ref{prop: side condition for the giant} which establishes the lower bound of $s_{d-1}(\mathcal{C}_{\max})$ that matches the upper bound obtained in Proposition~\ref{prop: MRSC giant size upper bound}. Under appropriate conditions, the property (\ref{eq: 2.2.3}) is equivalent to the following:
\[ \label{eq: 4.1.1} \tag{4.1.1}
\lim_{r \to \infty} \limsup_{n \to \infty} \Pr \lp \sigma_1 \overset{\sss{d}}{\longleftrightarrow} \sigma_2, s_{d-1} \lp \partial B(\sigma_1; r) \rp \geq r, s_{d-1} \lp \partial B(\sigma_2; r) \rp \geq r \rp = 1.
\]
Recall that $\sigma_1$ and $\sigma_2$ are $(d-1)$-simplices in $X_n$ chosen uniformly at random and independent of each other. For brevity, let us write the event $s_{d-1} \lp \partial B(\sigma_i; r) \rp \geq r$ as $\mathcal{L}_n^{(i)}(r)$ for $i = 1, 2$ and call it ``local survival'' of $\sigma_i$ for radius $r$. 

Although the property (\ref{eq: 4.1.1}) (and the original property (\ref{eq: 2.2.3}) itself is a direct extension of its version for random graphs~\cite{Hofstad:2021, Hofstad:2024}, we find that how this property takes place in random simplicial complexes is intrinsically different from situations in random graphs. To prove (\ref{eq: 4.1.1}), we have to show that conditional on local survival $\mathcal{L}_n^{(1)}(r) \cap \mathcal{L}_n^{(2)}(r)$ for large enough $r$, both components $\mathcal{C}(\sigma_1)$ and $\mathcal{C}(\sigma_2)$ grow further to be large enough so that they become connected eventually. The matter is to figure out how large each component needs to be and how to justify that each component reaches that size w.h.p, conditional on local survival. Let $\sigma$ and $\sigma'$ are two disjoint $(d-1)$-simplices on the vertex set $[n]$. Then, conditional on $\sigma, \sigma' \in X_n$, the probability that $(d-1)$-simplices $\sigma$ and $\sigma'$ are connected by some $d$-dimensional path in $X_n \sim \MRSC_d(n; \bm{p})$ is bounded below by $\prod_{i=1}^{d} p_i^{i \binom{d+1}{i+1}} = \Theta \lp n^{-\sum_{i=1}^d \alpha_i i \binom{d+1}{i+1}} \rp$, which is the ``cheapest way'' to connect them. For example, when $d = 2$, this way requires three new edges and two new triangles, which takes place with probability $p_1^3 p_2^2 = \Theta \lp n^{-3 \alpha_1 - 2\alpha_2} \rp$. Thus, naively speaking, it suffices to prove the following: conditional on large local survival, each component reaches a stage in which 
\[ \label{eq: 4.1.2} \tag{4.1.2}
s_{d-1} \lp \partial B(\sigma_i; r_n) \rp \geq \omega(n^{ \frac{\sum_{i=1}^d \alpha_i i \binom{d+1}{i+1}}{2}} ) \quad \text{ for } i = 1, 2
\]
for some radius $r_n$ depending on $n$ w.h.p.

When $d = 1$, $X_n \sim G(n, \lambda/n)$ and it suffices to show that each of the components reaches the stage in which its boundary contains $\omega(n^{1/2})$ vertices w.h.p., conditional on local survival. The formation of neighborhoods up to such a stage is still governed by the probability law of the breadth-first exploration of local neighborhoods. So to speak, the formation of the giant component is ``almost local'' in Erd\H{o}s-Rényi random graphs, which was dubbed by van der Hofstad~\cite{Hofstad:2021}.

On the other hand, for $d \geq 2$, the situation changes dramatically. For instance, when $X_n \sim \LM_d(n, \lambda/n)$, we need to show that $s_{d-1} \lp \partial B(\sigma_i; r_n) \rp \geq \omega(n^{d/2})$ for $i = 1,2$ for some radius $r_n$ w.h.p, conditional on local survival. As analyzed in Section~\ref{subsec: breadth-first exploration-stochastic properties}, the explorations of local neighborhoods is governed by forward simplices. However, as appeared in their probability law (Corollary~\ref{cor: probability of exploration processes}), their contribution is no longer significant once the exploration reaches the stage where a considerable number of vertices are included in the explored part. This reveals that in the formation of the giant $d$-dimensional component for $d \geq 2$, the exploration of backward and sibling simplices play a significant role, whose contribution was not meaningful to the formation of local neighborhoods. In this regard, we might say that the higher-dimensional giant component in random simplicial complexes is not almost local.

The probability laws of backward and sibling simplices at each step $k$ of the breadth-first exploration highly depend on how dense the explored part $\mathcal{X}_{k-1}$ is, which can be described in terms of degrees of simplices (Definition~\ref{def: degree}) in $\mathcal{X}_{k-1}$. As appeared in Lemma~\ref{lemma: backward estimate}, it is necessary to analyze the denseness of $\mathcal{X}_{k-1}$ for all dimensions lower than $d$. Such lower dimensional connections gets denser as the top dimension $d$ gets higher and the parameter $\lambda$ gets larger, which makes the problem much more challenging for large $d$ and $\lambda$ much larger than the critical value $1/d$. In fact, if $\lambda$ is very close to $1/d$, the denseness of the explored part is still rather tamed. For instance, if $\lambda = (1 + \e_n) /d$ with $\e_n \to 0$ as $n \to \infty$, which we call the barely supercritical regime, the degrees do not get overly large. This was the key ingredient of \cite{Cooley;Kang;Koch:2018} to obtain the concentration of the giant component in uniform hypergraphs in the barely supercritical regime, which can be translated to our problem for $X_n \sim \LM_d(n, \lambda/n)$ in the barely supercritical regime. However, beyond the barely supercritical regime, in particular for $\lambda > 1$, the lower-dimensional connections become so much denser that the known techniques are no longer applicable, even in $d = 2$. 

In this paper, we only focus on $d = 2$. As stated in Theorem~\ref{thm: supercritical MRSC}, we would like to remark that when $X_n \sim \LM_2(n, \lambda/n)$ we obtain the uniqueness and the concentration of the size of the $2$-dimensional giant component in the full supercritical regime $\lambda > 1/d$. When $X_n \sim \MRSC_2(n; \bm{p})$ which is not $\LM_2(n, \lambda/n)$, i.e., $p_1 = \Theta(n^{-\alpha_1})$ with $0 < \alpha_1 \leq 1/2$, large $\lambda$ makes the lower-dimensional structure more intricate and relatively denser compared with the order of the total number of edges $n^{2 - \alpha_1}$, which is smaller than that in $\LM_2(n, \lambda/n)$. Due to such an intrinsic challenge, our consideration does not cover the full supercritical regime. However, we would like to remark that the upper bound of $\lambda$ below which our result holds is much larger than $1$. Our primary goal in this section is to prove the following:

\begin{prop}
\label{prop: MRSC satisfies the side condition}
Assume that $d = 2$ with $p_1$ and $p_2$ parametrized by $p_1 = \Theta(n^{-\alpha_1})$ and $p_2 = \Theta(n^{-\alpha_2})$ satisfying $2\alpha_1 + \alpha_2 = 1$. Assume further that $0 \leq \alpha_1 < \frac{2 \log 2}{\log \lambda + 2 \log 2}$. Then, $X_n \sim \MRSC_2(n; p_1, p_2)$ satisfies the condition (\ref{eq: 4.1.1}) (equivalently, the condition (\ref{eq: 2.2.3}). More precisely,
\[ \label{eq: 4.1.3} \tag{4.1.3}
\lim_{r \to \infty} \limsup_{n \to \infty} \Pr \lp \sigma_1 \centernot\longleftrightarrow \sigma_2,  s_{1} \lp \partial B_{X_n} (\sigma_1; r) \rp,  s_{1} \lp \partial B_{X_n}(\sigma_2; r) \rp \geq r \rp = 0.
\]
\end{prop}

\noindent For each $r \geq 0$ and $b_0^{(1)}, b_0^{(2)} \geq r$ and $w_0^{(1)}, w_0^{(2)} \in \N$, define the event $\mathcal{L}_n(r; \bm{b}_0, \bm{w}_0)$ by 
\[ \label{eq: 4.1.4} \tag{4.1.4}
\mathcal{L}_n(r; \bm{b}_0, \bm{w}_0) := \left\{ s_{1} \lp B_{X_n}(\sigma_1; r-1) \rp = w_0^{(i)}, s_{1} \lp \partial B_{X_n}(\sigma_i; r) \rp = b_0^{(i)}, i = 1, 2 \right\}
\]
Then, the probability term in the left-hand side of (\ref{eq: 4.1.3}) can be written as
\[ \label{eq: 4.1.5} \tag{4.1.5}
\sum_{b_0^{(1)}, b_0^{(2)} \in \N, w_0^{(1)}, w_0^{(2)} \geq r} \Pr \lp \sigma_1 \centernot\longleftrightarrow \sigma_2 \: | \: \mathcal{L}_n(r; \bm{b}_0, \bm{w}_0) \rp \Pr \lp \mathcal{L}_n (r; \bm{b}_0, \bm{w}_0) \rp.
\]
Thus it suffices to prove the following: for a given $\e > 0$, 
\[ \label{eq: 4.1.6} \tag{4.1.6}
\Pr \lp \sigma_1 \centernot\longleftrightarrow \sigma_2 \: | \: \mathcal{L}_n(r; \bm{b}_0, \bm{w}_0) \rp \leq \e
\]
for all $b_0^{(1)}, b_0^{(2)} \in \N$ and $w_0^{(1)}, w_0^{(2)} \geq r$.

To provide the much clearer picture of the formal steps of the proof, we would like to connect the notion of boundaries to the breadth-first exploration process. Let $\sigma \in S_{d-1}(X_n)$ and suppose that we consider the breadth-first exploration $(\mathcal{X}_k)_{k \geq 0}$ of the connected component $\mathcal{C}(\sigma)$. On the other hand, let us consider the neighborhoods $\lp B_{X_n}(\sigma; r) \rp_{r \geq 0}$ of $\sigma$. By the definition of the breadth-first exploration (Section~\ref{subsec: breadth-first exploration}), for every $r$, we start explore $(d-1)$-simplices in $\partial B_{X_n}(\sigma; r)$ once we have explored all $(d-1)$-simplices in $B_{X_n}(\sigma; r-1)$. In other words, at Step $k = s_{d-1}(B_{X_n}(\sigma; r-1))$ of the exploration, the number of active $(d-1)$-simplices is equal to $s_{d-1} \lp \partial B_{X_n}(\sigma; r) \rp$. Therefore, showing that the exploration can reach a stage with large boundary is equivalent to showing that it can reach a stage with many active $(d-1)$-simplices. We will show that to prove Proposition~\ref{prop: MRSC satisfies the side condition} (i), it suffices to show that each of the explorations starting at $\sigma_1$ and $\sigma_2$ reach the stage with $T n^{1 - \alpha_1/2}$ active edges w.h.p., conditional on $\mathcal{L}_n(r; \bm{b}_0, \bm{w}_0)$ with large enough constant $T$ that depends on $\e$. 

The proof of \eqref{eq: 4.1.6}, which proves Proposition~\ref{prop: MRSC satisfies the side condition}, is broken down into three steps. First, we prove that under the assumed conditions in Proposition~\ref{prop: MRSC satisfies the side condition}, the degree $\deg_{0, 1}^{\mathcal{X}_k} (v)$ does not grow too fast in $n$ for every $v \in \mathcal{X}_k$ up to large enough exploration step $k$ as $n \to \infty$. This will allow us to approximate the distribution on the right-hand side of (\ref{eq: 3.2.29}) for $d = 2$ by $2 \mathsf{Bin}(n, \lambda/n)$ until the stage in which we obtain $T n^{1 - \alpha_1/2}$ active edges with constant $T$ chosen appropriately. Next step is to prove that conditional on large local survival of both $\sigma_1$ and $\sigma_2$, both neighborhoods can actually reach the stage we obtain many enough active edges. The final step is to show that the two large neighborhoods eventually connect to each other w.h.p.

First, we control the order of the vertex-edge degree through the breadth-first exploration of a $2$-dimensional connected component. 

\begin{lemma}
\label{lemma: upper bounds for degrees d =2}
Let $\sigma_0$ be an arbitrary edge in $X_n \sim \MRSC_d(n; p_1, p_2)$ satisfying the conditions in Proposition~\ref{prop: MRSC satisfies the side condition} (i). Let $(\mathcal{X}_k)_{k \geq 1}$ be the simplicial complexes obtained by the breadth-first exploration starting at $\sigma$. Then, for any fixed constant $T > 0$, as long as $k \leq T n^{1- \alpha_1/2}$, it holds that
\[ \label{eq: 4.1.7} \tag{4.1.7}
\max_{v \in S_0(\mathcal{X}_{k-1})} \deg_{0, 1}^{\mathcal{X}_{k-1}} (v) = o(n^{1 - \alpha_1}) \quad w.h.p.
\]
as $n \to \infty$.
\end{lemma}

\begin{proof}
Let $v \in [n]$ be an arbitrary vertex. At each step $k = 0, \dots, T n^{1- \alpha_1/2}$, the degree can increase in one of the following two scenarios: 1) the edge $\sigma_k$ that is explored at Step $k$ does not contain $v$ as its vertex or 2) $\sigma_k$ contains $v$ as its vertex. For each $k$, let $D_k(v)$ denote the increment $\deg_{0,1}^{\mathcal{X}_k}(v) - \deg_{0, 1}^{\mathcal{X}_{k-1}}(v)$. Then we can write $D_k(v) = D_k^{(1)}(v) + D_k^{(2)}(v)$, where $D_k^{(1)}(v)$ and $D_k^{(2)}(v)$ are the increments made by scenario $1$ and $2$, respectively. Note that for each $k$, only one of $D_k^{(1)}(v)$ and $D_k^{(2)}(v)$ is non-zero.

In the first scenario, given $\mathcal{X}_{k-1}$, the degree of $v$ increases by $1$ if $v$ is incident to exactly one of the vertices of $\sigma_k$ by an edge in $\mathcal{X}_{k-1}$ and the triangle $\{v\} \cup \sigma_k$ appears at Step $k$ with probability $p_1 p_2$. It increases by $2$ if $v$ is not incident to any of the vertices of $\sigma_k$ by edges in $\mathcal{X}_{k-1}$ and the  triangle $\{v\} \cup \sigma_k$ appears at Step $k$ with probability $p_1^2 p_2$, conditional on $\mathcal{X}_{k-1}$. Thus, we can define a sequence of  random variables $(W_k(v))_{k \geq 1}$ satisfying the following properties: $D_k^{(1)}(v) \leq W_k(v)$ for all $k \geq 0$ and $W_k(v) | \mathcal{X}_{k-1} \sim 2 \mathsf{Ber}(p_1 p_2)$ (Here $W_0(v) := 0$). The sequence of random variables $(W_k(v) - 2 p_1p_2)_{k \geq 0}$ defines a bounded martingale difference with respect to the filtration generated by $(\mathcal{X}_k)_{k \geq 0}$. By the Azuma-Hoeffding inequality (e.g., \cite[Theorem 3.2.1]{Roch:2023}), we have
\[ \label{eq: 4.1.8} \tag{4.1.8}
\Pr \lp \sum_{k=0}^{\rfloor T n^{1-\alpha_1/2} \rfloor} W_k(v) - 2(\lfloor T n^{1-\alpha_1/2} \rfloor +1)p_1 p_2  \geq t \rp \leq \exp \lp - \frac{2 t^2}{C\lfloor T n^{1-\alpha_1/2} \rfloor } \rp
\]
for a constant $C > 0$. Take $t = n^{1- \alpha_1 - \delta}$ with small enough $\delta > 0$ such that $1 - \frac{3}{2} \alpha_1 - 2 \delta > 0$. Since $W_k(v) \geq D_k^{(1)}(v)$ for all $k$, we obtain
\[ \label{eq: 4.1.9} \tag{4.1.9}
\begin{aligned}
& \Pr \lp \max_{v \in [n]} \sum_{k=0}^{T n^{1-\alpha_1/2}} D_k^{(1)}(v) \geq 2 \lp \lfloor T n^{1 - \alpha_1/2} \rfloor + 1 \rp p_1 p_2 + n^{1-\alpha_1 - \delta} \rp \\
& \qquad \qquad \qquad \qquad  \leq \sum_{v \in [n]} \Pr \lp \sum_{k=0}^{T n^{1-\alpha_1/2}} D_k^{(1)}(v) \geq 2 \lp \lfloor T n^{1 - \alpha_1/2} \rfloor + 1 \rp p_1 p_2 + n^{1-\alpha_1 - \delta} \rp \\
& \qquad \qquad \qquad \qquad  \leq \sum_{v \in [n]} \Pr \lp \sum_{k=0}^{T n^{1-\alpha_1/2}} W_k(v) \geq 2 \lp \lfloor T n^{1 - \alpha_1/2} \rfloor + 1 \rp p_1 p_2 + n^{1-\alpha_1 - \delta} \rp.
\end{aligned}
\]
Here, the first inequality is obtained by taking the union bound. By (\ref{eq: 4.1.8}), the last expression of (\ref{eq: 4.1.9}) is upper bounded by
\[ \label{eq: 4.1.10} \tag{4.1.10}
n \exp \lp - \frac{2 n^{2 - 2\alpha_1 - 2 \delta}}{C \lfloor T n^{1 - \alpha_1/2} \rfloor} \rp  \leq C_1 n \exp \lp - C_2 n^{1-\frac{3}{2} \alpha_1 - 2 \delta} \rp
\]
for some constants $C_1, C_2 > 0$. Because of our choice of $\delta$, the right-hand side of the above tends to $0$ as $n \to \infty$. Recall that $p_1 p_2 = \Theta (n^{-\alpha_1 - \alpha_2})$. This concludes that the contribution from the scenario 1 is of order $O(n^{\alpha_1/2}) + O(n^{1 - \alpha_1 - \delta}) = o(n)$ w.h.p. as $n \to \infty$. 

Before going into the detailed proof for the second scenario, we would like to recall that the sibling edges and triangles cannot have significant contribution in the exploration process as long as $\deg_{0, 1}^{\mathcal{X}_k} (w) = o(n^{1 - \alpha_1})$, where $w$ is a vertex contained in the edge being explored at Step $k$. This means that until the forward and backward explorations carry the exploration to the step where the vertex-edge degree is of order $n^{ 1 - \alpha_1}$, the sibling exploration is negligible. Thus, it suffices to prove (\ref{eq: 4.1.7}) for the contribution of forward and backward explorations. 

For every edge $\sigma \in \mathcal{X}_k$ that is incident to $v$ and obtained by Scenario 2, we say that $\sigma$ is a \textit{strong descendant} of a \textit{root Scenario 2 edges} if the following holds.  There is a finite sequence of edges $\sigma_0, \dots, \sigma_m = \sigma$ in $\mathcal{X}_k$ such that (i) $\sigma_j$ is obtained by exploring $\sigma_{j-1}$ for each $j = 1, \dots, m$ (ii) all $\sigma_j$, $j = 0, \dots, m-1$ are incident to $v$, i.e., $\sigma_1, \dots, \sigma_{m-1}$ are obtained by Scenario 2. (iii) $\sigma_0$, which we call a root Scenario 2 edge, is obtained by either Scenario 1 or as a starting point of the exploration process. Note that every such edge $\sigma$ is a strong descendant of exactly one of root Scenario 2 edges. Let $\mathcal{K} = \{\sigma^{(2)}_1, \dots, \sigma^{(2)}_m \}$ denotes the set of roots Scenario 2 edges. Additionally, we say that an edge $\sigma' \in \mathcal{X}_k$ (not necessarily incident to $v$) is a \textit{descendant} of a root Scenario 2 edge $\sigma_0$ if there is a finite sequence of edges $\sigma_1, \dots, \sigma_m = \sigma'$ such that $\sigma_j$ is obtained by exploring $\sigma_{j-1}$ for $j = 1, \dots, m-1$. Note that all $\sigma_j$, $j = 1, \dots, m$ are not necessarily incident to $v$. 

For each $j = 1, \dots, m$, let $\mathcal{T}_j'$ denote the set of descendants of $\sigma^{(2)}_j$ and $\mathcal{T}_j$ denote the set of strong descendants of $\sigma^{(2)}_j$; it is obvious that $\mathcal{T}_j \subseteq \mathcal{T}_j'$. Recall the (estimated) probability law of the sum of forward and backward explorations presented in Coroallary~\ref{cor: probability of exploration processes}, which is stochastically bounded by $2 \mathsf{Bin}(n, \lambda/n + o(1))$. Additionally, note that every forward or backward triangle obtained in Scenario 2 increases the degree by $1$. Therefore, the probability laws of $\mathcal{T}_1', \dots, \mathcal{T}_m'$ are stochastically bounded by those of independent binomial branching processes with common offspring distribution $2 \mathsf{Bin}(n, \lambda/n)$. Moreover, the probability law of each $\mathcal{T}_j$ is stochastically bounded by that of a binomial branching process with offspring distribution $\mathsf{Bin}(n, \lambda/n)$ embedded in the upper bounding branching process of $\mathcal{T}_j'$. Let $\tilde{\mathcal{T}}_j$ and $\tilde{\mathcal{T}}_j'$ denote those upper bounding branching process for each $j$. 

Let us consider the case $\lambda < 1$ first. In this case, all $\tilde{\mathcal{T}}_j'$ are subcritical branching processes with common offspring distribution $\mathsf{Bin}(n, \lambda/n)$. Since the cardinality of the set $\mathcal{K}$ is bounded by the contribution of Scenario 1 to the degree, we have
\[ \label{eq: 4.1.11} \tag{4.1.11}
\sum |\tilde{\mathcal{T}}_j'| = O \lp |\mathcal{K}| \rp = o \lp n^{1 - \alpha_1/2} \rp \quad w.h.p.
\]

When $\lambda \geq 1$, each $\tilde{\mathcal{T}}_j$ is a supercritical branching process and $\tilde{\mathcal{T}}_j'$ may be a supercritical branching process as well. If all $\tilde{\mathcal{T}}_j$ do not survive, the situation boils down to the case considered in the previous paragraph. Thus, it suffices to consider the situation at least one of them survives. Without loss of generality, let us suppose that $\tilde{\mathcal{T}}_1'$ survives, and denote the event of survival by $\mathcal{S}^{(n)}$. For each $N \geq 0$, let $Y_N^{(n)}$ and $Z_N^{(n)}$ denote the number of progeny in the $N$th generation of $\tilde{\mathcal{T}}_j'$ and $\tilde{\mathcal{T}}_j$, respectively. Additionally, for $R \geq 0$, let $S_R^{(n)} := \sum_{N=0}^R Y_N^{(n)}$ and $W_R^{(n)} := \sum_{N=0}^R Z_N^{(n)}$. Then, it suffices to prove the following. There exists a sequence $(R_n)_{n \geq 1}$ and $\beta \in (0, 1 - \alpha_1)$ such that
\[ \label{eq: 4.1.12} \tag{4.1.12}
\Pr \lp S_{R_n}^{(n)} \geq T n^{1 - \alpha_1/2} \: | \: \mathcal{S}^{(n)} \rp \to 1 
\]
and
\[ \label{eq: 4.1.13} \tag{4.1.13}
\Pr \lp W_{R_n}^{(n)} \geq n^{\beta} \: | \: \mathcal{S}^{(n)} \rp \to 0
\]
as $n \to \infty$. In words, this means that we can find a radius $R_n$ such that, conditional on survival of $\tilde{\mathcal{T}}_j'$, the degree $\deg_{0, 1}^{\mathcal{X}_k} (v)$ is of order $o(n^{1 - \alpha_1})$ until we detect at least $T n^{1 - \alpha_1/2}$ edges throughout the exploration. 

Intuitively, the sequence $(2 \lambda)^{-R} S_{R}^{(n)}$ concentrates around a positive constant as $R \to \infty$ by the martingale convergence theorem (\cite[Theorem 3.9]{Hofstad:2017}) and the Kesten-Stigum theorem (\cite[Theorem 3.10]{Hofstad:2017}). This tells us that $R_n = \log_{2 \lambda} C n^{1 - \alpha_1/2}$ for some constant $C > 0$ is a promising choice satisfying (\ref{eq: 4.1.12}). Moreover, following the same intuitive logic, $\lambda^{-R} W_R^{(n)}$ concentrates around a positive constant, conditional on $\mathcal{S}^{(n)}$, which tells us that $W_{R_n}^{(n)}$ is more or less $\lambda^{R_n} = o(n^{1 - \alpha_1})$. In fact, it will turn out that this $R_n$ is a correct choice (up to logarithmic adjustment). However, there is technical subtlety that prevents us from utilizing this intuitive logic directly. The limiting random variable of $(2 \lambda)^{-R} S_R^{(n)}$ obtained by the martingale convergence theorem depends on each $n$; hence, we cannot obtain nice enough concentration of $(2 \lambda)^{-R_n} S_{R_n}^{(n)}$, universal for all large $n$. The remaining part of the formal proof is dedicated to deal with this technical subtlety.

Let us choose $R_n := \log_{2 \lambda} ( n^{1 - \alpha_1/2} \log n)$. Let $\mathcal{T} = (Y_N)_{N \geq 0}$ denote a branching process with offspring distribution $2 \mathsf{Poi}(\lambda)$ and let $\mathcal{S}$ denote the event that this branching process survives. Let us write $q_n := \Pr ( \mathcal{S}^{(n)})$ and $q := \Pr ( \mathcal{S})$. Then we have $q_n \to q$ as $n \to \infty$. Let $\e > 0$ be given. First, we prove the following: there exists an integer $M > 0$ such that whenever $m > M$, it holds that
\[ \label{eq: 4.1.14} \tag{4.1.14}
\Pr \lp Y_m^{(n)} \geq m \: | \: \mathcal{S}^{(n)} \rp \geq 1 - \e
\]
for all large $n$. To prove it, choose $M > 0$ such that whenever $m > M$, 
\[ \label{eq: 4.1.15} \tag{4.1.15}
\Pr \lp Y_m \geq m \: | \: \mathcal{S} \rp \geq 1 - \e.
\]
This is possible because $(2 \lambda)^{-M} Y_M$ converges a.s. to a finite and integrable positive random variable by the martingale convergence theorem (\cite[Theorem 3.9]{Hofstad:2017}) and the Kesten-Stigum theorem (\cite[Theorem 3.10]{Hofstad:2017}). On the other hand, for each fixed $m > M$, we can write
\[ \label{eq: 4.1.16} \tag{4.1.16}
\Pr \lp Y^{(n)}_m \geq m \: | \: \mathcal{S}^{(n)} \rp = \frac{\Pr \lp X^{(n)}_m \geq m, \mathcal{S}^{(n)} \rp}{\Pr \lp \mathcal{S}^{(n)} \rp}.
\]
The numerator of (\ref{eq: 4.1.16}) can be written as 
\[ \label{eq: 4.1.17} \tag{4.1.17}
\begin{aligned}
\Pr \lp Y^{(n)}_m \geq m, \mathcal{S}^{(n)} \rp & = \E \lb \E \lb \1\{ Y^{(n)}_m \geq m \} \1_{\mathcal{S}^{(n)}} \: | \: Y_m^{(n)} \rb \rb \\
& = \E \lb \1\{ Y_m^{(n)} \geq m \} \E \lb \1 \{ \vee_{i = 1}^{Y_m^{(n)}} \mathcal{S}_i \} \: | \: Y_m^{(n)} \rb \rb,
\end{aligned}
\]
where each $\mathcal{S}_i$ denotes the event that a branching process started from one of $Y_m^{(n)}$ elements in the $m$th generation survives. Conditional on $Y_m^{(n)}$, all events $\mathcal{S}_i$ are independent, so the last expression of (\ref{eq: 4.1.17}) is equal to
\[ \label{eq: 4.1.18} \tag{4.1.18}
\E \lb \1\{Y_m^{(n)} \geq m \} \lp 1 - (1-q_n)^{Y_m^{(n)}} \rp \rb.
\]
Since $m$ was fixed, $Y_m^{(n)}$ converges in distribution to $Y_m$ as $n \to \infty$. Combined with the convergence $q_n \to q$, this implies that the expression in (\ref{eq: 4.1.18}) converges to 
\[
\E \lb \1\{ Y_m \geq m \} \lp 1 - ( 1 - q)^{Y_m} \rp \rb = \Pr \lp Y_m \geq m, \mathcal{S} \rp.
\]
This further implies that 
\[
\Pr \lp Y_m^{(n)} \geq m \: | \: \mathcal{S}^{(n)} \rp \to \Pr \lp Y_m \geq m \: | \: \mathcal{S} \rp.
\]
Therefore, for every $m > M$, there exists large $N >0$ such that whenever $n > N$, $\Pr \lp Y_m^{(n)} \geq m \: | \: \mathcal{S}^{(n)} \rp \geq 1 - c \e$ for some constant $c > 0$. This proves (\ref{eq: 4.1.14}).

Next, for each fixed $m > 0$, we are going to consider the probability
\[ \label{eq: 4.1.19} \tag{4.1.19}
\Pr \lp S_{R_n}^{(n)} \geq T n^{1 - \alpha_1/2} \: | \: Y_M \geq M \rp.
\]
for large enough $M$ which will be specified later. Suppose that $\e >0$ was chosen so small that $2\lambda (1 - \e) > 1$. For each $\ell \geq 0$, let $b_{\ell} := m \lp 2 \lambda ( 1- \e) \rp^{\ell}$ and $c_{\ell} := m \lp 2 \lambda ( 1 + \e) \rp^{\ell}$. Additionally, define the event
\[ \label{eq: 4.1.20} \tag{4.1.20}
\mathcal{E}_{M, \ell} := \{ b_{\ell} \leq Y_{m + \ell} \leq c_{\ell} \}
\]
and
\[ \label{eq: 4.1.21} \tag{4.1.21}
\mathcal{E}_{M, [R]} := \bigcap_{\ell = 1}^R \mathcal{E}_{M, \ell}
\]
for each $R \in \N$. Note that we can write
\[ \label{eq: 4.1.22} \tag{4.1.22}
\begin{aligned}
\Pr \lp \mathcal{E}_{M, [R]} \: | \: Y_M \geq M \rp & = \prod_{\ell \in [R]} \Pr \lp \mathcal{E}_{M, \ell} \: | \: \mathcal{E}_{M, [\ell-1]}, Y_M \geq M \rp \\
& = \prod_{\ell \in [R]} \lp 1 - \Pr \lp \mathcal{E}_{M, \ell}^c \: | \: \mathcal{E}_{M, [\ell-1]}, Y_M \geq M \rp  \rp \\
& \geq 1 - \sum_{\ell \in [R]}  \Pr \lp \mathcal{E}_{M, \ell}^c \: | \: \mathcal{E}_{M, [\ell-1]}, Y_M \geq M \rp
\end{aligned}
\]
For every $\ell \geq 1$, $Y_{M + \ell}$ conditional on $Y_{M + \ell-1}$ follows the same distribution as $\sum_{j=1}^{Y_{M + \ell-1}} B_j$, where $B_j$ are i.i.d. random variables whose common distribution is $2 \mathsf{Bin}(n, \lambda/n)$. This gives us
\[ \label{eq: 4.1.23} \tag{4.1.23}
\E \lb Y_{M + \ell} \: | \: \mathcal{E}_{M, [\ell - 1],}, Y_M \geq M \rb = 2 \lambda Y_{M + \ell -1} \in [M (2 \lambda)^{\ell} (1 - \e)^{\ell -1}, M (2 \lambda)^{\ell} (1 + \e)^{\ell -1 } ].
\]
Hence, if $Y_{M + \ell} \not\in [b_{\ell}, c_{\ell}]$, then 
\[ \label{eq: 4.1.24} \tag{4.1.24}
\left| Y_{M + \ell} - \E \lb Y_{M + \ell} \: | \: \mathcal{E}_{M, [\ell - 1]}, Y_M \geq M \rb \right| \geq \frac{\e}{100} \E \lb [Y_{M+\ell} \: | \: \mathcal{E}_{M, [\ell -1]}, Y_M \geq M \rb.
\]
By the large deviation concentration inequality of the binomial distribution (\cite[Lemma 2.35]{Hofstad:2024}), we have
\[ \label{eq: 4.1.25} \tag{4.1.25}
\begin{aligned}
& \Pr \lp \mathcal{E}_{M, \ell}^c \: | \: \mathcal{E}_{M, [\ell-1]}, Y_M \geq M \rp \\
& \leq  \Pr \lp \left| Y_{M + \ell} - \E \lb Y_{M+\ell} \: | \: \mathcal{E}_{M, [\ell-1]}, Y_M \geq M \rb \right|  \geq \frac{\e}{100} \E \lb Y_{M, \ell} \: | \: \mathcal{E}_{M,[\ell -1]}, Y_M \geq M \rb \rp \\
& \leq 2 \exp \lp - \frac{\e^2 \E \lb Y_{M, \ell} \: | \: \mathcal{E}_{M, [\ell -1]}, Y_M \geq M \rb}{2 \times 10^4  (1 + \e/100)} \rp \\
& \leq 2 \exp \lp - \frac{M \e^2 (2 \lambda)^{\ell} (1 - \e)^{\ell -1}}{2 \times 10^4 (1 + \e/100)} \rp.
\end{aligned}
\]
Combining (\ref{eq: 4.1.25}) and (\ref{eq: 4.1.22}), we obtain
\[ \label{eq: 4.1.26} \tag{4.1.26}
\Pr \lp \mathcal{E}_{M, [R]} \: | \: Y_M \geq M \rp \geq 1 - 2 \sum_{\ell =1}^R \exp \lp - \frac{M \e^2 (2 \lambda)^{\ell} (1 - \e)^{\ell -1}}{2 \times 10^4 (1 + \e/100)} \rp.
\]
Now set $R = R_n - M = \log_{2 \lambda} (n^{1 - \alpha_1/2} \log n) - M$ for all large enough $n$ and take large enough $M$ (depending on $\e)$ so that 
\[ \label{eq: 4.1.27} \tag{4.1.27}
\sum_{\ell =1}^{R_n} \exp \lp - \frac{M \e^2 (2 \lambda)^{\ell} (1 - \e)^{\ell -1}}{2 \times 10^4 (1 + \e/100)} \rp \leq \e.
\]
On the event $\mathcal{E}_{M, [R_n - M]}$, it holds that
\[ \label{eq: 4.1.28} \tag{4.1.28}
\begin{aligned} 
S_{R_n}^{(n)} = \sum_{N = 0}^{R_n} Y_N^{(n)} \geq C \lp 2 \lambda ( 1 - \e) \rp^{R_n} & = C \lp n^{1 - \alpha_1/2} \log n \rp (1 - \e)^{\log_{2 \lambda} (n^{1- \alpha_1/2} \log n)} \\
& \geq T n^{1 - \alpha_1/2}
\end{aligned}
\]
for all large enough $n$ and small enough $\e$. Here, $C > 0$ is a constant that only depends on $\lambda$. This proves that for every $m > M$, 
\[ \label{eq: 4.1.29} \tag{4.1.29}
\Pr \lp S_{R_n}^{(n)} \geq T n^{1 - \alpha_1/2} \: | \: Y_M \geq M \rp \geq 1 - \e
\]
for all large $n$.

Finally, it remains to prove (\ref{eq: 4.1.13}). Note that for every $R > 0$ and $p \geq 1$, 
\[ \label{eq: 4.1.30} \tag{4.1.30}
\E \lb | W_R^{(n)} |^p \rb \leq C_p (\lambda^R )^p,
\]
which gives us
\[ \label{eq: 4.1.31} \tag{4.1.31}
\E \lb | W_{R_n}^{(n)}|^p \rb \leq C_p n^{(1 - \alpha_1/2) \frac{\log \lambda}{\log \lambda + \log 2}} (\log n)^{\frac{\log \lambda}{\log \lambda + \log 2}}.
\]
Under the assumption $\lambda < 2^{2(1-\alpha_1)/\alpha_1}$, we have
\[ \label{eq: 4.1.32} \tag{4.1.32}
\lp 1 - \frac{\alpha_1}{2} \rp \frac{\log \lambda}{\log \lambda + \log 2} < 1 - \alpha_1,
\]
Choose $\gamma$ so that $(1 - \alpha_1/2) \frac{\log \lambda}{\log \lambda + \log 2} < \gamma < 1 - \alpha_1$. Then,  $\E \lb |W_R^{(n)}|^p \rb = O(n^{p \gamma})$. We will specify the choice of $p \geq 1$ later. Choose $\beta$ so that $\gamma < \beta < 1 - \alpha_1$. Then, by the Chebyshev inequality, we have
\[ \label{eq: 4.1.33} \tag{4.1.33}
\Pr \lp W_{R_n}^{(n)} \geq n^{\beta} \rp \leq \frac{\E \lb |W_{R_n}^{(n)} |^p \rb}{n^{p \beta}} = O \lp n^{-p( \beta - \gamma} \rp.
\]
Choose $p$ so large that $p (\beta - \gamma) > 1$. Then, we can conclude that
\[ \label{eq: 4.1.34} \tag{4.1.34}
\Pr \lp W_{R_n}^{(n)} \geq n^{\beta} \rp \leq \frac{\E \lb |W_{R_n}^{(n)} |^p \rb}{n^{p \beta}} = O(n^{-\delta})
\]
for some $\delta > 1$. 

The only remaining part is to obtain the upper bound for the maximum over all vertices. Using the union bound, the probability that the total contribution of Scenario 2, taken maximum over all vertices $v \in [n]$ is $\geq n^{\beta}$ is bounded above by
\[ \label{eq: 4.1.35} \tag{4.1.35}
n \Pr \lp W_{R_n}^{(n)} \geq n^{\beta} \rp \leq \frac{\E \lb |W_{R_n}^{(n)} |^p \rb}{n^{p \beta}}  = n O(n^{-\delta}) = o(1).
\]
This completes the proof of the lemma.

\end{proof}

To prove (\ref{eq: 4.1.6}), we first consider the breadth-first exploration of the connected component $\mathcal{C}(\sigma_1)$ containing $\sigma_1$ under $\mathcal{L}_n(\bm{b}_0, \bm{w}_0)$ with $b_0^{(1)}, b_0^{(2)} \geq r$ (we will specify the choice of $r$ later). Let us take $\mathcal{X}_0^{(1)} := B_{X_n}(\sigma_1; r)$ and $\sigma_1^{(1)}$ to be the least element in $S_1 \lp \partial B_{X_n}(\sigma_1; r) \rp$ with respect to the linear ordering on $S_2(\Delta_n)$. In words, we start the exploration of $\mathcal{C}(\sigma_1)$ in the initial state where $B_{X_n}(\sigma_1; r-1)$ is already fully explored and $\partial B_{X_n}(\sigma_1; r)$ is active. Additionally, we explore $\mathcal{C}(\sigma_1)$ in $X_n$ avoiding all the vertices contained in $B_{X_n}(\sigma_2; r)$. Let $(\mathcal{X}_k^{(1)})_{k \geq 1}$ be the sequence of simplicial complexes obtained by this exploration and let $\mathcal{A}^{(1)}_k$ denote the set of active edges at each Step $k$. Recall that the exploration terminates when $A^{(1)}_k := |\mathcal{A}_k^{(1)}| = 0$. Let $\tau^{(1)}$ denotes the stopping time at which the exploration terminates, conditional on $\mathcal{L}_n(r; \bm{b}_0, \bm{w}_0)$, i.e.,
\[ \label{eq: 4.1.36} \tag{4.1.36}
\tau^{(1)} := \{k \geq 0 \: : \: A_k = 0 \} \quad \text{ on  } \mathcal{L}_n(r; \bm{b}_0, \bm{w}_0).
\]

As we mentioned in the outline of the proof, we need to prove that the exploration process $(\mathcal{X}_k^{(1)})_{k \geq 1}$ reaches the stage, say Step $k^*$, at which $A_{k*} = T n^{1 - \alpha_1/2}$ for any constant $T > 0$ w.h.p., conditional on $\mathcal{L}_n(r; \bm{b}_0, \bm{w}_0)$ with large enough $r$ (And, after that, we will prove the same thing for the exploration of $\sigma_2$, avoiding the explored simplicial complex $\mathcal{X}_{k^*}^{(1)}$ of $\sigma_1$).  We will see that the exploration reaches such a stage as long as $\tau^{(1)} \geq \delta n^{1- \alpha_1/2}$ for some $\delta > 0$. In general, this can happen with non-zero probability because $\sigma_1$ can be chosen among the edges that belong to small connected components. However, we claim that the large local survival $\mathcal{L}_n(r; \bm{b}_0, \bm{w}_0)$ rules out such a possibility.

\begin{lemma}
\label{lemma: local survival implies larger survival}
Let $\e > 0$ and $\delta > 0$ be given. Then, there exists $r = r(\e) > 0$ such that
\[ \label{eq: 4.1.37} \tag{4.1.37}
\Pr \lp \tau^{(1)} \geq \delta n^{1 - \alpha_1/2} \: | \: \mathcal{L}_n(r; \bm{b}_0, \bm{w}_0) \rp \geq 1 - \e.
\]
\end{lemma}

\begin{proof}
Because it suffices to consider a lower bound of $A_k$, we consider the active edges obtained from the forward and backward exploration only. More precisely, let us define $(\tilde{A}^{(1)}_k)$ as follows: $\tilde{A}^{(1)}_k = A^{(1)}_k$ for $1 \leq k \leq s_1 \lp \partial B_{X_n}(\sigma_1; r) \rp$. For $k \geq s_1 \lp \partial B_{X_n}(\sigma_1; r) \rp$, $\tilde{A}_{k+1} := \tilde{A}_k$ if the edge explored at Step $k$ has been obtained from a sibling triangle; otherwise, $\tilde{A}^{(1)}_{k+1} := \tilde{A}^{(1)}_k + F_k^{(1)} + B_k^{(1)} - 1$, where $F_k^{(1)}$ and $B_k^{(1)}$ denote the number of forward and backward edges obtained at Step $k$ through the exploration process $(\mathcal{X}_k^{(1)})_{k \geq 1}$

By Lemma~\ref{lemma: upper bounds for degrees d =2}, it is guaranteed that $\deg_{0, 1}^{\mathcal{X}^{(1)}} (v)$ is of order $o(n^{1-\alpha_1})$ with probability at least $1 - \e$ for all large $n$ and $v \in \mathcal{X}_k^{(1)}$ until we explore $T n^{1 - \alpha_1/2}$ edges for any large constant $T > 0$. By Corollary~\ref{cor: probability of exploration processes}, there is an event of probability $\geq 1 - \e$ on which 
\[
F_k^{(1)} + B_k^{(1)} \geqst 2 \mathsf{Bin} \lp n - s_0 \lp B_{X_n}(\sigma_2; r) \rp - o(n), \frac{\lambda}{n} (1 - o(1)) \rp,
\]
conditional on $\mathcal{X}_{k-1}^{(1)}$ for all $k \leq Tn^{1 - \alpha_1/2}$, however large the fixed $T > 0$ is. Take $\gamma > 0$ so small that $2 \lambda (1- \gamma)^2 > 1$. Then there exists a large $N > 0$ such that whenever $n \geq N$, we have
\[ \label{eq: 4.1.38} \tag{4.1.38}
F_k^{(1)} + B_k^{(1)} \geqst 2 \mathsf{Bin} \lp \lfloor (1 - \gamma)n \rfloor, \frac{\lambda}{n} (1 - \gamma) \rp,
\]
conditional on $\mathcal{X}^{(1)}_{k-1}$ with probability $ \geq 1 - \e$. Let us write the mean of the the distribution on the right-hand side of (\ref{eq: 4.1.38}) as $2 \lambda' > 1$. Let $(Z_k^{(n)})_{k \geq 0}$ be the sequence of random variables adapted to the filtration $(\sigma(\mathcal{X}_k^{(1)}))_{k \geq 1}$ whose law is specified as follows  and for each $k \geq 1$, conditional on $\mathcal{X}_{k-1}^{(1)}$, we have
\[
Z_k^{(n)} - Z_{k-1}^{(n)} \sim 2 \mathsf{Bin} \lp \lfloor n (1 - \gamma) \rfloor, \frac{\lambda}{n} (1 - \gamma) \rp - 1.
\]
For each $m \geq 0$, let us define the hitting time
\[ \label{eq: 4.1.39} \tag{4.1.39}
T_m := \{k \geq 1 \: : \: Z_k^{(n)} = m \}.
\]
Then, proving (\ref{eq: 4.1.37}) boils down to proving the following: 
\[ \label{eq: 4.1.40} \tag{4.1.40}
\Pr \lp T_0 > T_{\delta  n^{1 - \alpha_1/2}} \: | \: Z_0^{(n)} = b_0^{(1)} \rp,
\]
for $b_0^{(1)} \geq r$. Note that the above problem is essentially identical to the gambler's ruin problem for a random walk with a positive drift (see, for example, \cite[Theorem 4.8.9]{Durrett:2019}).

Let $W \sim 2 \mathsf{Bin} \lp \lfloor n (1 - \gamma) \rfloor, \frac{\lambda}{n} (1 - \gamma) \rp$ and let $G_n(s)$ be the probability generating function of $W$, i.e.,
\[
G_n(s) := \E \lb s^{W} \rb \quad s \in [0, 1].
\]
Let $q_n$ be the smallest solution in $[0, 1]$ to the equation $G_n(s) = s$. Since $W$ follows a binomial distribution with mean bigger than $1$, $q_n \in (0, 1)$ for every $n$. Let us define
\[ \label{eq: 4.1.41} \tag{4.1.41}
M_k^{(n)} := q_n^{Z_k^{(n)}}.
\]
Then, $(M_k^{(n)})_{k \geq 1}$ is a martingale adapted to the filtration $(\sigma(\mathcal{X}_k^{(1)}))_{k \geq 1}$. In fact,
\[
\begin{aligned}
\E \lb M_k^{(n)} \: | \: \mathcal{X}_{k-1}^{(1)} \rb  = \E \lb q_n^{Z_{k-1} + (Z_k - Z_{k-1})} \: | \: \mathcal{X}_{k-1}^{(1)} \rb & = q_n^{Z_{k-1}} \E \lb q_n^{W-1} \rb \\
& = M_{k-1}^{(n)}.
\end{aligned}
\]
Note that the last equality follows our choice of $q_n$. Set $\tau := T_0 \wedge T_{\delta n^{1 - \alpha_1/2}}$. By the optimal stopping theorem (\cite[Theorem 4.8.3]{Durrett:2019}), we obtain
\[ \label{eq: 4.1.42} \tag{4.1.42}
\E \lb M_{\tau} \: | \: Z_0 = b_0^{(1)} \rb = \E \lb M_0 \: | \: Z_0 = b_0^{(1)} \rb = q^{b_0^{(1)}} \leq q_n^r.
\]
On the other hand, we have
\[ \label{eq: 4.1.43} \tag{4.1.43}
\begin{aligned}
\E \lb M_{\tau} \: | Z_0 = b_0^{(1)} \rb & = \E \lb q_n^{M_{\tau}} 1\{ T_0 < T_{\delta n^{1 - \alpha_1/2}} \} \: | \: Z_0 = b_0^{(1)} \rb + \E \lb q_n^{M_{\tau}} 1\{ T_0 > T_{\delta n^{1 - \alpha_1/2}} \} \: | \: Z_0 = b_0^{(1)} \rb \\
& \geq \E \lb q_n^{M_{T_0}} 1\{ T_0 < T_{\delta n^{1 - \alpha_1/2}} \} \: | \: Z_0 = b_0^{(1)} \rb \\
& = \Pr \lp T_0 < T_{\delta n^{1 - \alpha_1/2}}  \: | \: Z_0 = b_0^{(1)} \rp.
\end{aligned}
\]
Hence, we obtain
\[ \label{eq: 4.1.44} \tag{4.1.44}
\Pr \lp T_0 < T_{\delta n^{1 - \alpha_1/2}}  \: | \: Z_0 = b_0^{(1)} \rp \leq q_n^r.
\]
Note that the random variable $W$ converges in distribution to $2 \mathsf{Poi}(\lambda')$ as $n \to \infty$. This implies that $q_n \to q \in (0, 1)$ as $n \to \infty$. Here, $q$ is the smallest solution to the equation $G(s) = s$, where $G(s)$ is the probability generating function of the distribution $2 \mathsf{Poi}(\lambda')$. Choose small enough $\gamma' > 0$ so that $q(1 + \gamma') < 1$. Then, for all large enough $n$, $q_n < q(1 + \gamma')$. Now, choose $r$ so large that it satisfies $( q ( 1 + \gamma'))^r < \e$ for the $\e > 0$ initially given. This completes the proof.

\end{proof}

Next, using Lemma~\ref{lemma: local survival implies larger survival}), we prove that the process $(A^{(1)}_k)_{k \geq 1}$ reaches a stage $k^*$ at which $A_{k^*} \geq T n^{1 - \alpha_1/2}$ for any constant $T > 0$ w.h.p., conditional on large local survival. 

\begin{lemma}
\label{lemma: exploration of the first component}
Let $\e > 0$ and $y_1 > 0$ be given. Then, there exist $R = R(\e) >0$ and  $T_1 = T_1(y_1) > 0$ such that whenever $r \geq R(\e)$, 
\[ \label{eq: 4.1.45} \tag{4.1.45}
\liminf_{n \to \infty} \Pr \lp A^{(1)}_{\lfloor T_1 n^{1 - \alpha_1/2} \rfloor} \geq y_1 n^{1 - \alpha_1/2} \: | \: \mathcal{L}_n(r; \bm{b}_0, \bm{w}_0) \rp \geq 1 - \e.
\]
\end{lemma}

\begin{proof}
Choose $0 < \delta < 1$. By Lemma~\ref{lemma: local survival implies larger survival}, there exists large $R(\e) > 0$ such that whenever $r \geq R(\e)$, 
\[ \label{eq: 4.1.46} \tag{4.1.46}
\Pr \lp \tau^{(1)} \geq \delta n^{1 - \alpha_1/2} \: | \: \mathcal{L}_n(r; \bm{b}_0, \bm{w}_0) \rp \geq 1 - \e.
\]
For brevity, let the event $\{\tau^{(1)} \geq \delta n^{1 - \alpha_1/2} \}$ be denoted by $\mathcal{T}^{(1)}$. Then the probability term in (\ref{eq: 4.1.45}) can be written as
\[
\begin{aligned}
& \Pr \lp A^{(1)}_{\lfloor T_1 n^{1 - \alpha_1/2} \rfloor} \geq y_1 n^{1 - \alpha_1/2} \wedge \mathcal{T}^{(1)}  \: | \: \mathcal{L}_n(r; \bm{b}_0, \bm{w}_0) \rp \\
& \qquad \qquad \qquad \qquad \qquad + \Pr \lp A^{(1)}_{\lfloor T_1 n^{1 - \alpha_1/2} \rfloor} \geq y_1 n^{1 - \alpha_1/2} \wedge (\mathcal{T}^{(1)})^c  \: | \: \mathcal{L}_n(r; \bm{b}_0, \bm{w}_0) \rp.
\end{aligned}
\]
By (\ref{eq: 4.1.46}), the latter term of the above is $\leq \e$. Thus, it suffices to prove that for $r \geq R(\e)$, 
\[ \label{eq: 4.1.47} \tag{4.1.47}
\liminf_{n \to \infty} \Pr \lp A^{(1)}_{\lfloor T_1 n^{1 - \alpha_1/2} \rfloor} \geq y_1 n^{1 - \alpha_1/2} \wedge \mathcal{T}^{(1)} \: | \: \mathcal{L}_n(r; \bm{b}_0, \bm{w}_0) \rp \geq 1 - \e.
\]

Since we only consider up to step of order $n^{1 - \alpha_1/2}$, we can again use the lower bound $(\tilde{A}^{(1)}_k)_{k \geq 1}$. For each $t > 0$, define
\[ \label{eq: 4.1.48} \tag{4.1.48}
\tilde{\alpha}_n^{(1)}(t) := \frac{\tilde{A}^{(1)}_{\lfloor t n^{1 - \alpha_1/2} \rfloor}}{n^{1 - \alpha_1/2}}.
\]
For brevity, let us write $\tilde{E}_k^{(1)} := F_k^{(1)} + B_k^{(1)}$. Recall the recursion formula of $\tilde{A}_k^{(1)}$:
\[ \label{eq: 4.1.49} \tag{4.1.49}
\begin{aligned}
\tilde{A}^{(1)}_k & = \tilde{A}^{(1)}_{k-1} + \tilde{E}_k^{(1)} - 1 \\
& = \tilde{A}_0^{(1)} + \sum_{j=1}^k \tilde{E}_j^{(1)} - k.
\end{aligned}
\]
Recall that Corollary~\ref{cor: probability of exploration processes} and Lemma~\ref{lemma: upper bounds for degrees d =2} tell us that there is an event of probability $\geq 1 - \e$ on which 
\[ \label{eq: 4.1.50} \tag{4.1.50}
\tilde{E}_k^{(1)} \big| \mathcal{X}_{k-1}^{(1)} \geqst 2 \mathsf{Bin} \lp n - o(n), \frac{\lambda}{n} (1 - o(1)) \rp
\]
for all $k \leq T n^{ 1- \alpha_1/2}$, however large the fixed $T > 0$ is. We only focus on this event. Define the function $\tilde{m}_n^{(1)}(t)$ and $\tilde{\delta}^{(1)}_n(t)$ by
\[ \label{eq: 4.1.51} \tag{4.1.51}
\tilde{m}_n^{(1)}(t) := \frac{1}{n^{1 - \alpha_1/2}} \sum_{k=1}^{\lfloor t n^{1 - \alpha_1/2} \rfloor} \lp \tilde{E}_k^{(1)} - \E \lb \tilde{E}_k^{(1)} \: | \: \mathcal{X}_{k-1}^{(1)} \rb \rp
\]
and
\[ \label{eq: 4.1.52} \tag{4.1.52}
\tilde{\delta}^{(1)}_n(t) := \frac{1}{n^{1 - \alpha_1/2}} \sum_{k=1}^{\lfloor t n^{1 - \alpha_1/2} \rfloor} \E \lb \tilde{E}_k^{(1)} \: | \: \mathcal{X}_{k-1}^{(1)} \rb - \frac{\lfloor t n^{1 - \alpha_1/2} \rfloor}{n^{1 - \alpha_1/2}} + \frac{1}{n^{1 - \alpha_1/2}} \tilde{A}_0^{(1)}.
\]
respectively. Then, we can write
\[ \label{eq: 4.1.53} \tag{4.1.53}
\begin{aligned}
\tilde{\alpha}^{(1)}_n(t) & = \frac{1}{n^{1 - \alpha_1/2}} \sum_{k=1}^{\lfloor t n^{1 - \alpha_1/2} \rfloor} \lp \tilde{E}_k^{(1)} - \E \lb \tilde{E}_k^{(1)} \: | \: \mathcal{X}_{k-1}^{(1)} \rb \rp \\
& + \frac{1}{n^{1 - \alpha_1/2}} \sum_{k=1}^{\lfloor t n^{1 - \alpha_1/2} \rfloor} \E \lb \tilde{E}_k^{(1)} \: | \: \mathcal{X}_{k-1}^{(1)} \rb - \frac{\lfloor t n^{1 - \alpha_1/2} \rfloor}{n^{1 - \alpha_1/2}} + \frac{1}{n^{1 - \alpha_1/2}} \tilde{A}_0^{(1)}.
\end{aligned}
\]
In other words, we decompose $\tilde{\alpha}_n^{(1)}(t)$ into the sum of its martingale and drift parts.

Let us analyze the martingale term $\tilde{m}_n^{(1)}(t)$ first. We claim that its influence eventually gets negligible. 

\begin{lemma}
\label{lemma: martingale term bound}
Let $T > 0$ and $\e' > 0$ be fixed. Then, whenever $r \geq R(\e)$ that satisfies (\ref{eq: 4.1.46}), it holds that
\[ \label{eq: 4.1.54} \tag{4.1.54}
\lim_{n \to \infty} \Pr \lp \sup_{t \in [0, T]} | \tilde{m}_n^{(1)}(t) | > \e' \: | \: \mathcal{L}_n(r; \bm{b}_0, \bm{w}_0) \rp = 0.
\]
\end{lemma}

\begin{proof}[Proof of Lemma~\ref{lemma: martingale term bound}]
By Doob's inequality (\cite[Theorem 3.8]{Karatzas;Shreve:1998}) and the Burkholder-Davis-Gundy inequality (\cite[Theorem 3.28]{Karatzas;Shreve:1998}), we have
\[ \label{eq: 4.1.55} \tag{4.1.55}
\Pr \lp \sup_{t \in [0, T]} |\tilde{m}^{(1)}_n(t)| > \e' \: | \: \mathcal{L}_n(r; \bm{b}_0, \bm{w}_0) \rp \leq C \frac{\E \lb \langle \tilde{m}_n^{(1)} \rangle_T \: | \: \mathcal{L}_n(r; \bm{b}_0, \bm{w}_0) \rb }{\e'^2}
\]
for some universal constant $C > 0$, where $\langle \tilde{m}_n^{(1)} \rangle$ denotes the quadratic variation process of $\tilde{m}_n^{(1)}(t)$. The quadratic variation is written as
\[ \label{eq: 4.1.56} \tag{4.1.56}
\langle \tilde{m}_n^{(1)} \rangle_T = \frac{1}{n^{2 - \alpha_1}} \sum_{k=1}^{\lfloor T n^{1 - \alpha_1/2} \rfloor} \V \lb \tilde{E}_k^{(1)} \: | \: \mathcal{X}_{k-1}^{(1)}, \mathcal{L}_n(r; \bm{b}_0, \bm{w}_0) \rb.
\]
Recall that $\tilde{E}_k^{(1)}$ satisfies the following trivial upper bound: $\tilde{E}_k^{(1)} \leqst 2 \mathsf{Bin}(n, \lambda/n)$. Since the variance is bounded above by the second moment, the right-hand side of (\ref{eq: 4.1.56}) is bounded by
\[ \label{eq: 4.1.57} \tag{4.1.57}
\frac{1}{n^{2 - \alpha_1}} \sum_{k=1}^{T n^{1 - \alpha_1/2}} 4 \lp \lambda \lp  1 - \frac{\lambda}{n} \rp + \lambda^2 \rp \leq C n^{-1 + \alpha_1/2}
\]
for some constant $C > 0$ that depends on $T$ and $\lambda$. Applying (\ref{eq: 4.1.57}) to (\ref{eq: 4.1.55}), we obtain the desired result (\ref{eq: 4.1.54}).
\end{proof}

Next, let us analyze the drift term $\delta^{(1)}_n(t)$. Define $\delta^{(1)}(t) := (2 \lambda - 1) t$ for every $t > 0$. We claim that the drift terms $\tilde{\delta}^{(1)}_n(t)$ do not go too much below the function $\delta^{(1)}$ w.h.p.

\begin{lemma}
\label{lemma: drift term convergence}
Let $T > 0$ and $0 < \e' < 1$ be fixed. Then, whenever $r \geq R(\e)$ that satisfies (\ref{eq: 4.1.46}), it holds that
\[ \label{eq: 4.1.58} \tag{4.1.58}
\liminf_{n \to \infty} \Pr \lp \tilde{\delta}^{(1)}_n(t) \geq \delta^{(1)}(t) - \e' \: \forall t \in [0, T] \: | \: \mathcal{L}_n(r; \bm{b}_0, \bm{w}_0) \rp \geq 1 - \e.
\]

\end{lemma}

\begin{proof}[Proof of Lemma~\ref{lemma: drift term convergence}]
Recall that with probability $\geq 1 - \e$, the lower bound of $\tilde{E}_k^{(1)}$ given in (\ref{eq: 4.1.50}) is valid for all $k \leq T n^{1 - \alpha_1/2}$. On that event, we have 
\[ \label{eq: 4.1.59} \tag{4.1.59}
\begin{aligned}
\tilde{\delta}^{(1)}_n(t) & \geq \frac{1}{n^{1 - \alpha_1/2}} \sum_{k=1}^{\lfloor t n^{1 - \alpha_1/2} \rfloor} (2 \lambda (1 - o(1)) - \frac{\lfloor t n^{1- \alpha_1/2} \rfloor}{n^{1 - \alpha_1/2}} + \frac{\tilde{A}_0^{(1)}}{n^{1 - \alpha_1/2}} \\
& = 2 \lambda t \lp 1 - o(1) \rp - t + o(1).
\end{aligned}
\]
Here, note that $\tilde{A}_0^{(1)} = b_0^{(1)}$, a fixed constant, under the condition $\mathcal{L}_n(r; \bm{b}_0, \bm{w}_0)$. We have
\[
2 \lambda t \lp 1 - o(1) \rp - t + o(1) - ( 2 \lambda (t - 1) - \e') > 0
\]
uniformly for all $t \in [0, T]$ for all large enough $n$. This proves (\ref{eq: 4.1.58}). 

\end{proof}

Now, we are ready to complete the proof. Choose $\e' > 0$ small enough so that $(2 \lambda  - 1) \delta - 2 e ' > 0$, and for given $y_1 > 0$, choose $T_1$ large enough so that $(2\lambda -1) T_1 - 2\e' > y_1$. By Lemma~\ref{lemma: martingale term bound}, we have $\sup_{t \in [0, T_1]} |\tilde{m}_n^{(1)}(t)| \leq \e'$ with probability $\geq 1 - \e$ for all large enough $n$. On the other hand, with probability $\geq 1 - \e$, the drift term keeps increasing since its slope $(2 \lambda - 1)$ is positive. On the event $\mathcal{T}^{(1)}$, the process $\tilde{\alpha}^{(1)}_n(t)$ is positive for $t \in [0, \delta]$. After the stage $t = \delta$, the drift term keeps increasing and the size of the martingale term is so small that the process $\tilde{\alpha}_n^{(1)}(t)$ lies above the function $\delta^{(1)}(t) - 2 \e'$ with probability $\geq 1 - \e$ for all $t \in [\delta, T_1]$. Since $ \Pr \lp \mathcal{T}^{(1)} \: | \: \mathcal{L}_n(r; \bm{b}_0, \bm{w}_0) \rp \geq 1 - \e$ for all large enough $n$ whenever $r \geq R(\e)$, we complete the proof.

\end{proof}

The next step is to explore the connected component of $\sigma_2$, avoiding a large enough explored part of $\mathcal{C}(\sigma_1)$. In order to make sure that the second exploration is also driven by the supercritical branching process we have to explore $\mathcal{C}(\sigma_1)$ not too large. For $y_1 > 0$, let $k_* := \lfloor T_1 n^{1-\alpha_1/2} \rfloor$, where $T_1$ satisfies the statement of Lemma~\ref{lemma: exploration of the first component}. Fix $y_1$ small enough so that $s_0(\mathcal{X}^{(1)}_{k_*}) \leq \delta n$ w.h.p.  as $n \to \infty$ for some $\delta > 0$ that satisfies $2\lambda (1 - \delta) > 1$. Let us define $(\mathcal{X}_k^{(2)})_{k \geq 0}$ to be the sequence of simplicial complexes obtained by the breadth-first exploration of $\mathcal{C}(\sigma_2)$ that starts at $B_{X_n}(\sigma_2; r)$ and avoids $\mathcal{X}_{k_*}^{(1)}$. More precisely, we take $\mathcal{X}_0^{(2)} = B_{X_n}(\sigma_2; r)$ with initial set of active edges $\mathcal{A}_0^{(2)} := S_1(\partial B_{X_n}(\sigma_2; r)$. The breadth-first exploration starts at the least edge in $\mathcal{A}_0^{(2)}$ with respect to the linear ordering on $S_1(\Delta_n)$ and the exploration only uses vertices in $[n] \setminus S_0(\mathcal{X}_{k_*}^{(1)}$. Then, following the same argument we used to obtain Lemma~\ref{lemma: exploration of the first component}, we obtain the following:

\begin{lemma}
\label{lemma: exploration of the second component}
Let $\e > 0$ and $y_2$ be given. Then, there exists $T_2 = T_2(y_2) > 0$ such that whenever $r \geq R(\e)$, 
\[ \label{eq: 4.1.60} \tag{4.1.60}
\liminf_{n \to \infty} \Pr \lp A^{(1)}_{\lfloor T_2 n^{1 - \alpha_1/2} \rfloor} \geq y_2 n^{1 - \alpha_1/2} \: | \: \mathcal{X}_{k_*}^{(1)}, \mathcal{L}_n(r; \bm{b}_0, \bm{w}_0) \rp \geq 1 - \e.
\]
\end{lemma}

We omit the detailed proof of this lemma because it is identical to that of Lemma~\ref{lemma: exploration of the first component}. 

The final step is to show that with appropriate choice of $y_1$ and $y_2$, we can show that $\sigma_1 \overset{\sss{2}}{\leftrightarrow} \sigma_2$ given that we are having $y_1 n^{1-\alpha_1/2}$ and $y_2 n^{1 - \alpha_1/2}$ active edges in each of the components $\mathcal{C}(\sigma_1)$ and $\mathcal{C}(\sigma_2)$, respectively. 

\begin{lemma}
\label{lemma: connection happens}
Let $\e > 0$ be given. For $y_1$ and $y_2$ that satisfy $(1 + y_1 y_2 \lambda^2)^{-1} < \e$, it holds that
\[
\limsup_{n \to \infty} \Pr \lp \sigma_1 \centernot\longleftrightarrow \sigma_2 \: | \: A^{(1)}_{\lfloor T_1 n^{1-\alpha_1/2} \rfloor} \geq y_1 n^{1-\alpha_1/2}, A^{(2)}_{\lfloor T_2 n^{1-\alpha_1/2} \rfloor} \geq y_2 n^{1-\alpha_1/2}, \mathcal{L}_n(r; \bm{b}_0, \bm{w}_0) \rp < \e
\]
whenever $r \geq R(\e)$. 
\end{lemma}

Recall that we chose $y_1 > 0$ small enough that $2\lambda (1 -  \delta) > 1$ for some $\delta > 0$. At the end, we will choose $y_2 > 0$ large enough so that $(1 + y_1 y_2 \lambda^2)^{-1} < \e$ for the given $\e > 0$.

\begin{proof}[Proof of Lemma~\ref{lemma: connection happens}]
For brevity, let $\mathcal{E}^c$ denote the event $\{A^{(1)}_{\lfloor T_1 n^{1-\alpha_1/2} \rfloor} \geq y_1 n^{1-\alpha_1/2}, A^{(2)}_{\lfloor T_2 n^{1-\alpha_1/2} \rfloor} \geq y_2 n^{1-\alpha_1/2} \}$. Let us consider a pair of active edges $\sigma$ and $\sigma'$ that are connected to $\sigma_1$ and $\sigma_2$, respectively. Note that they are disjoint by our definition of the exploration processes for $sigma_1$ and $\sigma_2$. The shortest $2$-dimensional path connecting $\sigma$ and $\sigma'$ consists of two adjacent triangles, say $\tau$ and $\tau'$, such that $\sigma \subseteq \tau$ and $\sigma' \subseteq \tau'$. There are two possible choices of such a pair $(\tau, \tau')$. Conditional on $\mathcal{E}^*$, the probability that $\tau$ and $\tau'$ appear in $X_n$ needs the appearance of three new edges and two new triangles. Since there are two such possibilities and the intersection of the two shares the appearance of two edges, the probability that at least one of such connections takes place is
\[
\begin{aligned}
p^*  := 2 p_1^3 p_2 - p_1^4 p_2^4 & = 2 p_1 p_2 \lp \frac{\lambda}{n} + o(1) \rp - \lp p_2 \lp \frac{\lambda}{n} + o(1) \rp \rp^2 \\
& = 2 \lambda^2 \Theta (n^{-2 + \alpha_1}) - \lambda^4 \Theta (n^{-4 + 4 \alpha_1}) \\
& = 2 \lambda^2 n^{-2 + \alpha_1} \lp 1 - \Theta (n^{-2 + 3 \alpha_1}) \rp.
\end{aligned}
\]
Here, we used the assumptions $p_1 = \Theta(n^{-\alpha_1})$, $p_2 = \Theta(n^{-\alpha_1})$ for $\alpha_1, \alpha_2 \geq 0$ and $n p_1^2 p_2 \to \lambda$. Recall that this assumption implies that $\alpha_1 \leq 1/2$. 

For brevity, let us write $\mathcal{A}^{(1)} := \mathcal{A}^{(1)}_{\lfloor T_1 n^{1 - \alpha_1/2} \rfloor}$ and $\mathcal{A}^{(2)} := \mathcal{A}^{(2)}_{\lfloor T_2 n^{1 - \alpha_1/2} \rfloor}$. Additionally, for any pair of edges $\sigma$ and $\sigma'$, let us write
\[
I(\sigma, \sigma') := \1 \{ \sigma \overset{\sss{2}}{\longleftrightarrow} \text{ by two adjacent triangles} \}.
\]
Then, we have
\[
\begin{aligned}
& \Pr \lp \sigma_1 \centernot\longleftrightarrow \sigma_2 \: | \: \mathcal{E}^*, \mathcal{L}_n(r; \bm{b}_0, \bm{w}_0) \rp \\
& \qquad \qquad \leq \Pr \lp \bigcap_{\sigma \in \mathcal{A}^{(1)}, \sigma' \in \mathcal{A}^{(2)}} \{ \sigma \centernot\longleftrightarrow \sigma' \text{ by two adjacent triangles} \: | \: \mathcal{E}^*, \mathcal{L}_n(r; \bm{b}_0, \bm{w}_0) \rp \\
& \qquad \qquad = \Pr \lp \sum_{\sigma \in \mathcal{A}^{(1)}, \sigma' \in \mathcal{A}^{(2)}} I(\sigma, \sigma') = 0 \: | \: \mathcal{E}^*, \mathcal{L}_n(r; \bm{b}_0, \bm{w}_0) \rp.
\end{aligned}
\]
We obtain an upper bound for the last expression of the above by using the second moment method. The first moment satisfies
\[ \label{eq: 4.1.61} \tag{4.1.61}
\begin{aligned}
\E \lb \sum_{\sigma \in \mathcal{A}^{(1)}, \sigma' \in \mathcal{A}^{(2)}} I(\sigma, \sigma') \: | \: \mathcal{E}^*, \mathcal{L}_n(r; \bm{b}_0, \bm{w}_0) \rb & \gtrsim y_1 y_2 n^{2 - \alpha_1/2} \lambda^2 n^{-2 + \alpha_1} \lp 1 - \Theta (n^{-2 + 3 \alpha_1}) \rp \\
& = y_1 y_2 \lambda^2  \lp 1 - o(1) \rp.
\end{aligned}
\]
The second moment can be written as the sum of the following two terms:
\[ \label{eq: 4.1.62} \tag{4.1.62}
\E \lb \sum_{(\sigma, \sigma') \in \mathcal{A}^{(1)} \times \mathcal{A}^{(2)}} I(\sigma, \sigma') \: | \: \mathcal{E}^*, \mathcal{L}_n(r; \bm{b}_0, \bm{w}_0) \rb 
\]
and
\[ \label{eq: 4.1.63} \tag{4.1.63}
\E \lb \sum_{ \substack{(\sigma, \rho), (\sigma', \rho') \in \mathcal{A}^{(1)} \times \mathcal{A}^{(2)} \\ (\sigma, \rho) \neq (\sigma', \rho')}} I(\sigma, \sigma') I(\rho, \rho') \: | \: \mathcal{E}^*, \mathcal{L}_n(r; \bm{b}_0, \bm{w}_0) \rb.
\]
Obviously, the first term (\ref{eq: 4.1.62}) is equal to the first moment considered in (\ref{eq: 4.1.61}). To analyze the second term (\ref{eq: 4.1.63}) further, note that $I(\sigma, \sigma')$ and $I(\rho, \rho')$ are independent unless $\rho \cap \rho \neq \emptyset$ and $\sigma' \cap \rho' \neq \emptyset$ simultaneously. Let us write
\[
\mathcal{D} := \{ (\sigma, \rho, \sigma', \rho') \in \mathcal{A}^{(1)} \times \mathcal{A}^{(1)} \times \mathcal{A}^{(2)} \times \mathcal{A}^{(2)} \: | \: (\sigma, \sigma') \neq (\rho, \rho'), \sigma \cap \rho \neq \emptyset, \sigma' \cap \rho' \neq \emptyset \}
\]
and
\[
\mathcal{D}' := \left\{ (\sigma, \rho, \sigma', \rho') \in \mathcal{A}^{(1)} \times \mathcal{A}^{(1)} \times \mathcal{A}^{(2)} \times \mathcal{A}^{(2)} \: | \: (\sigma, \sigma') \neq (\rho, \rho') \wedge \lp \sigma \cap \rho = \emptyset \vee \sigma' \cap \rho' = \emptyset \rp \right\}.
\]
Then, the expression (\ref{eq: 4.1.63}) can be written as
\[
\E \lb \sum_{(\sigma, \rho, \sigma', \rho') \in \mathcal{D}} I(\sigma, \sigma') I(\rho, \rho') \: \mathcal{E}', \mathcal{L}_n(r; \bm{b}_0, \bm{w}_0) \rb +\E \lb \sum_{(\sigma, \rho, \sigma', \rho') \in \mathcal{D}'} I(\sigma, \sigma') I(\rho, \rho') \: \mathcal{E}', \mathcal{L}_n(r; \bm{b}_0, \bm{w}_0) \rb.
\]
By independence, the second term of the above is upper bounded by
\[
(y_1 y_2)^2 n^{4 - 2 \alpha_1} \lambda^4 n^{-4 + 2 \alpha_1} \lp 1 - \Theta (n^{-2 + 3 \alpha_1} ) \rp = y_1^2 y_2^2 \lambda^4 (1 - o(1)).
\]
As for the first term, recall that for each $\sigma \in \mathcal{A}^{(1)}$, the number of edges incident to one of its vertices in $\mathcal{X}^{(1)}_{\lfloor T_1 n^{1 - \alpha_1/2} \rfloor}$ is of order $o(n^{1 - \alpha_1})$ with probability $\geq 1 - \e$ by Lemma~\ref{lemma: upper bounds for degrees d =2}. Likewise, for each $\sigma' \in \mathcal{A}^{(2)}$, the number of edges incident to one of its vertices in $\mathcal{X}^{(2)}_{\lfloor T_2 n^{1-\alpha_1/2} \rfloor}$ is of order $o(n^{1 -\alpha_1/2})$ with probability $\geq 1 - \e$. If $\sigma$ and $\rho$ shares a vertex and $\sigma'$ and $\rho'$ share a vertex, then the probability that $I(\sigma, \sigma') I(\rho, \rho') = 1$ is the probability that we obtain four adjacent triangles connecting them, which needs five new edges and four new triangles. This happens with probability 
\[
p_1^5 p_2^4 = \lp \frac{\lambda}{n} + o(1) \rp^2 p_1 p_2^2 = \lambda^2 \Theta (n^{-4 + 3 \alpha_1}).
\]
Thus, the contribution of such cases is bounded above by
\[
y_1 y_2 n^{2 - \alpha_1} o(n^{2 - 2\alpha_1}) \lambda^2 \Theta (n^{-4 -3 \alpha_1}) = o(1).
\]
On the other hand, if $\sigma = \rho$ and $\sigma'$ and $\rho'$ share a vertex, the probability that $I(\sigma, \sigma') I(\rho, \rho') = 1$ is the probability that we obtain three adjacent triangles connecting them, which needs four new edges and three new triangles. This happens with probability
\[
p_1^4 p_2^3 = \lp \frac{\lambda}{n} + o(1) \rp^2 p_2 = \lambda^2 n^{-2} \Theta (n^{-1 + 2\alpha_1}) = \lambda^2 \Theta (n^{-3 + 2\alpha_1}).
\]
Thus, the contribution of such cases is bounded above by
\[
y_1 y_2 n^{2 - \alpha_1} o(n^{1 - \alpha_1}) \lambda^2 \Theta (n^{-3 +2\alpha_1}) = o(1).
\]
Therefore, the entire second moment (\ref{eq: 4.1.62}) is bounded above by
\[
y_1^2 y_2^2 \lambda^4 (1 - o(1)) + o(1)
\]
as $n \to \infty$. By the second moment method, we have
\[
\begin{aligned}
& \Pr \lp \sum_{\sigma \in \mathcal{A}^{(1)}, \sigma' \in \mathcal{A}^{(2)}} I(\sigma, \sigma') \geq 1 \: | \: \mathcal{E}^*, \mathcal{L}_n(r; \bm{b}_0, \bm{w}_0) \rp \\
& \hspace{180pt} \geq  \frac{\E \lb \sum_{\sigma \in \mathcal{A}^{(1)}, \sigma' \in \mathcal{A}^{(2)}} I(\sigma, \sigma') \: | \: \mathcal{E}^*, \mathcal{L}_n(r; \bm{b}_0, \bm{w}_0) \rb^2 }{\E \lb \lp \sum_{\sigma \in \mathcal{A}^{(1)}, \sigma' \in \mathcal{A}^{(2)}} I(\sigma, \sigma') \rp^2 \: | \: \mathcal{E}^*, \mathcal{L}_n(r; \bm{b}_0, \bm{w}_0) \rb }.
\end{aligned}
\]
The right-hand side of the above is lower bounded by
\[
\frac{y_1 y_2 \lambda^2}{1 + \frac{y_1^2 y_2^2 \lambda^4 (1 - o(1)) + o(1)}{y_1 y_2 \lambda_2}} = \frac{y_1 y_2 \lambda^2}{1 + y_1 y_2 \lambda^2 (1 - o(1)) + o(1)}
\]
with probability $\geq 1 - \e$. Since we chose $y_1$ and $y_2$ so that the right-hand side is bigger than $1 - \e$, as $n \to \infty$, the right-hand side is $\geq 1 - \e$. This completes the proof.
\end{proof}

\subsection{Proof of Theorem \ref{thm: subcritical regime}}

In this section, we prove that the size of the largest $d$-dimensional connected component in $X_n \sim \MRSC_d(n; \bm{p})$ is more or less $\log n$ in the subcritical regime $\lambda < 1/d$. In particular, there is no giant $d$-dimensional component in this regime. The overall  arguments are similar to that of subcritical Erd\H{o}s-R\'enyi random graphs~\cite[Theorem 4.4]{Hofstad:2021}.

Let $\sigma$ be an arbitrary $(d-1)$-simplex in $X_n$ and consider the breath-first exploration of the connected component $\mathcal{C}(\sigma)$. In this section, we write $(\mathcal{X}_k)_{k \geq 0}$ for the sequence of simplicial complexes obtained by the breadth-first exploration of $\mathcal{C}(\sigma)$ starting at $\sigma$. Recall that the number of new active $(d-1)$-simplices $E_k$ consists of the forward, backward, and sibling $(d-1)$-simplices: $E_k = F_k + B_k + H_k$. We claim that the contribution of the forward and backward exploration can make the component grow only up to size $\log n$. If that claim is true, the contribution of the sibling exploration is negligible to the size of the component since $H_k \leqst d \mathsf{Bin}(s_0(\mathcal{X}_{k-1}), p_{d, d-1})$ for each $k$, conditional on $\mathcal{X}_{k-1}$, with $s_0(\mathcal{X}_{k-1}) \lesssim \log n$ and $p_{d, d-1} = \Theta(n^{-\beta})$ for some $\beta > 0$. Therefore, it suffices to regard that the exploration takes place only using the forward and backward explorations. Recall that, conditional on $\mathcal{X}_{k-1}$,
\[ \label{eq: 4.2.1} \tag{4.2.1}
\tilde{E}_k \leqst d \mathsf{Bin} \lp n, \frac{\lambda}{n} \rp.
\]

\begin{proof}[Proof of Theorem \ref{thm: subcritical regime}]
For each $k \geq 0$, let $Z_{\geq k}$ be the number of $(d-1)$-simplices that are contained in $d$-dimensional connected components of size at least $k$:
\[
Z_{\geq k} := \sum_{\sigma \in S_{d-1}(X_n)} \1 \{ s_{d-1}\lp \mathcal{C}(\sigma) \rp \geq k \}.
\]
Then, we have
\begin{equation} \label{subcritical proof 2}
\{ s_{d-1} \lp \mathcal{C}_{\max} \rp \geq k \} = \{Z_k \geq k \}
\end{equation}
This equality~\eqref{subcritical proof 2}, along with the fact that $s_{d-1}(X_n)) \leq \binom{n}{d}$, gives us the following union bound:
\[ \label{eq: 4.2.3} \tag{4.2.3}
\begin{aligned}
\Pr \lp s_{d-1} \lp \mathcal{C}_{\max} \rp \geq k \rp & = \Pr \lp \sum_{\sigma \in S_{d-1}(X_n)} \1\{ s_{d-1}(\mathcal{C}(\sigma)) \geq k \} \rp   \\
&  \leq \binom{n}{d} \Pr \lp s_{d-1} \lp \mathcal{C}(\sigma) \rp \geq k \rp,
\end{aligned}
\]
where $\sigma$ is an arbitrarily fixed edge, say $\sigma = \{1, \dots, d\}$. Thus, it suffices to estimate the probability $\Pr \lp s_{d-1} \lp \mathcal{C}(\sigma)\rp \geq k \rp$. Let $T$ be the total progeny of a branching process with offspring distribution $d \mathsf{Bin} (n, \lambda/n)$. By \eqref{eq: 4.2.1}, we have
\[
\Pr \lp s_{d-1} \lp \mathcal{C}(\sigma) \rp \geq k \rp \leq \Pr \lp T \geq k \rp. 
\]
Let $\{W_i \}_{i \geq 1}$ be i.i.d.~random variables whose common distribution is $d \mathsf{Bin}(n, \lambda/n)$ and let 
\[
U_k := W_1 + \cdots W_k - (k-1), \quad k \geq 1.
\]
Then, $T > k$ if and only if $U_k > 0$. Thus, we can obtain the upper bound of the probability $\Pr \lp T > k \rp$ as follows:
\[ \label{eq: 4.2.4} \tag{4.2.4}
\begin{aligned}
\Pr \lp T > k \rp \leq \Pr \lp W_1 + \cdots + W_k \geq k \rp & = \Pr \lp \hat{W}_1 + \cdots + \hat{W}_k \geq \frac{k}{d} \rp \\
& \leq e^{-\frac{k}{d} I_{\lambda}},
\end{aligned}
\]
where $\{\hat{W}_i \}_{i \geq 1}$ are i.i.d. random variables whose common distribution is $\mathsf{Bin}(n, \lambda/n)$ and $I_{\lambda}$ is the large deviations rate function for the Poisson distribution of mean $\lambda$, given by
\[ \label{eq: 4.2.5} \tag{4.2.5}
I_{\lambda} = \lambda - 1 - \log \lambda.
\]
Now, take $k = k(n) = a \log n$ so that $a > d^2/I_{\lambda}$. Then,
\[ \label{eq: 4.2.6} \tag{4.2.6}
\Pr \lp s_{d-1} \lp \mathcal{C}_{\max} \rp > a \log n \rp \leq n^d \Pr \lp T > a \log n \rp \leq n^{d - \frac{a I_{\lambda}}{d}} = n^{- \delta},
\]
where $\delta = -d + a I_{\lambda}/d > 0$. This establishes the proposed upper bound. 

To establish the proposed lower bound, we only consider the contribution of the forward explorations. By the obtained upper bound, we know that $s_{d-1}(\mathcal{C}_{\max}) \leq C \log n$ for some constant with probability $\geq 1- \e$ for all large enough $n$. This implies that the exploration of any connected component terminates $\log n$ steps up to constants and the number of vertices contained in any component is also bounded by $\log n$ up to constants with probability $\geq 1 - \e$. On this event, the forward exploration of any connected component $\mathcal{C}(\sigma)$ satisfies that, conditional on $\mathcal{X}_{k-1}$,
\[
F_k \geqst d \mathsf{Bin} \lp n - C \log n, \frac{\lambda}{2n} \rp
\]
for all large $n$. Then, obviously, the size of $\mathcal{C}(\sigma)$ is lower bounded by that of the component $\mathcal{C}'(\sigma)$ that only considers forward simplices, which we call a forward component. We are going to establish a lower bound for the largest forward component $\mathcal{C}'_{\max}$. Note that if $s_{d-1}\lp \mathcal{C}'_{\max} \rp \geq c \log n$ for some $c > 0$, it implies that there exists at least one $d$-dimensional component whose size is $\geq c \log n$, which in turn implies $s_{d-1} \lp \mathcal{C}_{\max} \rp \geq c \log n$. 

Let us define
\[
Z'_{\geq k} := \sum_{\sigma \in S_{d-1}(X_n)} \1\{ s_{d-1} \lp \mathcal{C}'(\sigma) \rp \geq k \}.
\]
Similarly as \eqref{eq: 4.2.1}, for every $k$, we have 
\[ \label{eq: 4.2.7} \tag{4.2.7}
\left\{ s_{d-1} \lp \mathcal{C}'_{\max} \rp < k \right\} = \{Z'_{\geq k} = 0 \}.
\]
This implies that
\[ \label{eq: 4.2.8} \tag{4.2.8}
\Pr \lp s_{d-1} \lp \mathcal{C}_{\max} \rp < k \rp \leq \Pr \lp s_{d-1} \lp \mathcal{C}'_{\max} \rp < k \rp = \Pr \lp Z'_{\geq k} = 0 \rp.
\]
The remaining part is to take $k = c \log n$ and show that the right-most probability term in (\ref{eq: 4.2.8}) tends to $0$ as $n \to \infty$ by using the second moment method. Because the remaining part of the proof is identical to that of \cite[Theorem 4.5]{Hofstad:2017}, we omit the detail.

\end{proof}

\subsection{Proof of Theorem \ref{thm: vertex}}

The result for the subcritical regime $\lambda < 1/d$ immediately follows Theorem~\ref{thm: subcritical regime} because the number of vertices is bounded by the number of $(d-1)$-simplices up to a constant. 

In the supercritical regime, we are going to analyze the number of vertices in $\mathcal{C}_{\max}$ by means of the process $\lp s_0 (\mathcal{X}_k\rp_{k \geq 0}$, where $(\mathcal{X}_k)_{k \geq 0}$ denotes the sequence of simplicial complexes obtained by the breadth-first exploration of $\mathcal{C}_{\max}$ starting at some $\sigma \in S_{d-1}(\mathcal{C}_{\max})$. We would like to remark that a unique $d$-dimensional giant component exist in $(X_n)_{n \geq 1}$ when $\frac{1}{d} < \lambda < \frac{1}{d-1}$. Unlike the uniqueness and the concentration of the size of the giant component, the property \eqref{thm eq: vertex supercritical C_max} is an ``increasing property'', so to speak. More precisely, if $\lambda > \lambda'$ and $X'_n \sim \LM_d(n, \lambda'/n)$ satisfies \eqref{thm eq: vertex supercritical C_max}, then so does $X_n \sim \LM_d(n, \lambda/n)$ because $X_n$ and $X_n'$ can be coupled appropriately so that $X_n' \subseteq X_n$, which implies that the number of vertices in the largest component in $(X'_n)_{n \geq 1}$ is smaller than or equal to that in $(X_n)_{n \geq 1}$ under such an appropriate coupling. Therefore, it suffices to prove \eqref{thm eq: vertex supercritical C_max} for $\frac{1}{d} < \lambda < \frac{1}{d-1}$.

Let $X_n \sim \LM_d(n, \lambda/n)$ with $\frac{1}{d} < \lambda < \frac{1}{d-1}$ and choose $\sigma \in S_{d-1}(\mathcal{C}_{\max})$. Recall that for any $\e > 0$, there exists $\delta > 0$ such that
\[
\Pr \lp s_{d-1} \lp \mathcal{C}_{\max} \rp \geq \delta n^d \rp \geq 1 - \e
\]
for all large enough $n$. In particular, the breadth-first exploration of $\mathcal{C}(\sigma) = \mathcal{C}_{\max}$ does not terminate at least until step $\delta n^d$ with probability $\geq 1 - \e$. We restrict our attention to this event. Let $(\mathcal{X}_k)_{k \geq 0}$ be the sequence of simplicial complexes obtained by the breadth-first exploration. Then, every Step $k$, the number of newly appeared forward $d$-simplices $F_{k, d}$ satisfies the following (Lemma~\ref{lemma: forward estimate}): conditional on $\mathcal{X}_{k-1}$,
\[ \label{eq: 4.3.1} \tag{4.3.1}
F_{k,d} \sim \mathsf{Bin} \lp n - s_0(\mathcal{X}_{k-1}), \frac{\lambda}{n} \rp.
\]
On the other hand, the number of vertices in $\mathcal{X}_k$ satisfies the following:
\[ \label{eq: 4.3.2} \tag{4.3.2}
s_0(\mathcal{X}_k) = s_0(\mathcal{X}_{k-1}) + F_{k, d}
\]
with $s_0(\mathcal{X}_0) = d$. For brevity, let us write $V_k := s_0(\mathcal{X}_k)$ for each $k \geq 0$. Then, \eqref{eq: 4.3.1} and \eqref{eq: 4.3.2} gives us the following: conditional on $V_{k-1}$, 
\[ \label{eq: 4.3.3} \tag{4.3.3}
V_k - V_{k-1} \sim \mathsf{Bin} \lp n - V_{k-1}, \frac{\lambda}{n} \rp
\]

Define the scaled process $v_n(t)$ as
\[ \label{eq: 4.3.4} \tag{4.3.4}
v_n(t) := \frac{V \lfloor tn \rfloor}{n} \quad t > 0.
\]
We are going to investigate the behavior of this process. Let us write $W_k := V_k - V_{k-1}$ for each $k \geq 1$. Then, we have
\[ \label{eq: 4.3.5} \tag{4.3.5}
v_n(t) = \frac{1}{n} \sum_{k=1}^{\lfloor t n \rfloor} W_k + \frac{1}{n} V_0.
\]
Let
\[ \label{eq: 4.3.6} \tag{4.3.6}
m_n^v(t) := \frac{1}{n} \sum_{k=1}^{\lfloor tn \rfloor} \lp W_k - \E \lb W_k \: | \: V_{k-1} \rb \rp 
\]
and
\[ \label{eq: 4.3.7} \tag{4.3.7}
\delta_n^v(t) := \frac{1}{n} \sum_{k=1}^{\lfloor tn \rfloor} \E \lb W_k \: | \: V_{k-1} \rb + \frac{V_0}{n}.
\]
This gives us 
\[
v_n(t) = m_n^v(t) + \delta_n^v(t). 
\]
Let us first analyze the martingale term $m_n^v(t)$. Since $W_k \sim \mathsf{Bin}(n - V_{k-1}, \lambda/n)$, conditional on $V_{k-1}$, we can estimate its quadratic variation as follows:
\[ \label{eq: 4.3.8} \tag{4.3.8}
\begin{aligned}
\langle m_n^v \rangle_T = \frac{1}{n^2} \sum_{k=1}^{\lfloor T n \rfloor} \V \lb W_k \: | \: V_{k-1} \rb  \leq \frac{1}{n^2} \sum_{k=1}^{\lfloor Tn \rfloor} \E \lb W_k^2 \: | \: V_{k-1} \rb & \leq \frac{1}{n^2} \sum_{k=1}^{\lfloor Tn \rfloor} \lp n \frac{\lambda}{n} \lp 1 - \frac{\lambda}{n} \rp + \lambda^2 \rp \\
& = O \lp \frac{1}{n} \rp
\end{aligned}
\]
for any fixed constant $T > 0$. Here, the constant suppressed in the last expression depends on $T$ and $\lambda$. Hence, by the Doob inequality (\cite[Theorem 3.8]{Karatzas;Shreve:1998}) and the Burkholder-Davis-Gundy inequality (\cite[Theorem 3.28]{Karatzas;Shreve:1998}), we obtain the following: for any $\e' > 0$, 
\[ \label{eq: 4.3.9} \tag{4.3.9}
\lim_{n \to \infty} \Pr \lp \sup_{t \in [0, T]} |m_n^v(t)| > \e' \: | \: s_{d-1}\lp \mathcal{C}_{\max} \rp \geq \delta n^d \rp = 0.
\]
On the other hand, the drift term $\delta^v_n(t)$ can be written as
\[ \label{eq: 4.3.10} \tag{4.3.10}
\delta_n^v(t) = \frac{1}{n} \sum_{k=1}^{\lfloor t n \rfloor} (n - V_{k-1}) \frac{\lambda}{n} + \frac{d}{n} = \lambda \lp t - \frac{1}{n} \sum_{k=1}^{\lfloor tn \rfloor} \frac{V_{k-1}}{n} \rp + O \lp \frac{1}{n} \rp.
\]
Note that the function $s \mapsto n^{-1} V_{\lfloor s \rfloor}$ is an increasing step function whose jumps take place only possibly at $s \in \Z$ on $[0, tn]$ almost surely. This implies that the function is Riemann integrable and we can write the right-most expression of the above as
\[
\lambda \lp t - \frac{1}{n}\int_{s = 0}^{\lfloor tn \rfloor -1} \frac{V_{\lfloor s \rfloor}}{n} \d s \rp + O \lp \frac{1}{n} \rp.
\]
Take $s = nu$, then the above can be written as
\[
\lambda \lp t -  \int_{u = 0}^{\frac{\lfloor tn \rfloor -1}{n}} \frac{V_{\lfloor nu \rfloor}}{n} \d u \rp + O \lp \frac{1}{n} \rp,
\]
which is equal to
\[ \label{eq: 4.3.11} \tag{4.3.11}
\lambda \lp t - \int_{u=0}^t v_n(u) \d u + \int_{u = \frac{\lfloor tn \rfloor -1}{n}}^t v_n(u) \d u \rp + O \lp \frac{1}{n} \rp.
\]
Since $v_n(t) \leq 1$ for all $t > 0$, we can write
\[ \label{eq: 4.3.12} \tag{4.3.12}
\delta_n^v(t) = \lambda \lp t - \int_{u=0}^t m_n^v(u) \d u + \int_{u=0}^t \delta_n^v(u) \d u + r_n \rp + a_n,
\]
where $r_n$ is a sequence of non-negative random variables that satisfy $|r_n| \leq b_n$ for some deterministic sequence $b_n = O(n^{-1})$ and $a_n$ is a deterministic sequence of order $O(n^{-1})$.

On the other hand, let $\delta(t)$ be the solution to the following equation:
\[ \label{eq: 4.3.13} \tag{4.3.13}
\delta(t) = \lambda \lp t - \int_{u= 0}^t \delta(u) \d u \rp,
\]
which is explicitly given by
\[ \label{eq: 4.3.14} \tag{4.3.14}
\delta(t) = 1 - e^{-\lambda t}.
\]
We claim that, conditional on $s_{d-1} \lp \mathcal{C}_{\max} \rp \geq \delta n^d$,  for any fixed $T > 0$, 
\[ \label{eq: 4.3.15} \tag{4.3.15}
\sup_{t \in [0, T]} \left| v_n(t) - \delta(t) \right|  \overset{\Pr}{\to} 0
\]
as $n \to \infty$. Combining \eqref{eq: 4.3.12} and \eqref{eq: 4.3.13}, we obtain
\[ \label{eq: 4.3.16} \tag{4.3.16}
\begin{aligned}
\left| \delta_n^v(t) - \delta(t) \right| & = \left| -\lambda \lp \int_{u=0}^t \lp \delta_n^v(u) - \delta(u) \rp \d u - \int_{u=0}^t m_n^v(u) \d u  + r_n \rp + a_n \right| \\
& \leq \lambda \int_{u=0}^t \left| \delta_n^v(u) - \delta(u) \right| \d u + \int_{u=0}^t |m_n^v(u)| \d u + c_n,
\end{aligned}
\]
where $c_n := |a_n| + |b_n| = O(n^{-1})$. For a given $T > 0$, let us write $M_n := \sup_{t \in [0, T]} |m_n^v(t)|$, then we have
\[ \label{eq: 4.3.17} \tag{4.3.17}
|\delta_n^v(t) - \delta(t)| \leq  \lambda \int_{u=0}^t \left| \delta_n^v(u) - \delta(u) \right| \d u + T M_n + c_n.
\]
By the Gronwall inequality (\cite[Theorem 1.10]{Tao:2006}), we obtain
\[ \label{eq: 4.3.18} \tag{4.3.18}
\sup_{t \in [0, T]} |\delta_n^v(t) - \delta(t)| \leq \lp T M_n + c_n \rp e^{\lambda T}.
\]
By \eqref{eq: 4.3.9}, $M_n \overset{\Pr}{\to} 0$ as $n \to \infty$, conditional on $s_{d-1}\lp \mathcal{C}_{\max} \rp \geq \delta n^d$. This, along with $c_n \to 0$, implies that the right-hand side of \eqref{eq: 4.3.18} converges to $0$ in probability, conditional on $s_{d-1} \lp \mathcal{C}_{\max} \rp \geq \delta n^d$, i.e.,
\[ \label{eq: 4.3.19} \tag{4.3.19}
\sup_{t \in [0,T]} |\delta_n^v(t) - \delta(t)| \overset{\Pr}{\to} 0
\]
as $n \to \infty$, conditional on $s_{d-1} \lp \mathcal{C}_{\max} \rp \geq \delta n^d$.

Since $v_n(t) = m_n^v(t) + \delta_n^v(t)$, we have
\[ \label{eq: 4.3.20} \tag{4.3.20}
\sup_{t \in [0, T]} |v_n(t) - \delta(t)| \leq \sup_{t \in [0, T]} |\delta^v_n(t) - \delta(t)| + \sup_{t \in [0, T]} |m_n^v(t)|.
\]
Combining \eqref{eq: 4.3.9} and \eqref{eq: 4.3.19}, we obtain that
\[
\sup_{t \in [0, T]} |v_n(t) - \delta(t)| \overset{\Pr}{\to} 0,
\]
conditional on $s_{d-1} \lp \mathcal{C}_{\max} \rp \geq \delta n^d$, as $n \to \infty$.

Now, we are ready to complete the proof of Theorem~\ref{thm: vertex}. For any given $\e > 0$, choose $T > 0$ large enough so that $e^{-\lambda T} < \e/2$. Note that $s_0(\mathcal{C}_{\max}) \geq n v_n(t)$ for every $t \in [0, T]$. Thus, we have
\[
\begin{aligned}
\Pr \lp \frac{s_0 \lp \mathcal{C}_{\max} \rp}{n} \geq 1 - \e \rp & = \Pr \lp \frac{s_0 \lp \mathcal{C}_{\max} \rp}{n} \geq 1 - \e \: | \: s_{d-1} \lp \mathcal{C}_{\max} \rp \geq \delta n^d \rp \Pr \lp s_{d-1} \lp \mathcal{C}_{\max} \rp \geq \delta n^d \rp \\
& \geq \Pr \lp v_n(T) \geq 1 - \e \: | \: s_{d-1} \lp \mathcal{C}_{\max} \rp \geq \delta n^d \rp \Pr \lp s_{d-1} \lp \mathcal{C}_{\max} \rp \geq \delta n^d \rp.
\end{aligned}
\]
By \eqref{eq: 4.3.20}, we have
\[ 
\begin{aligned}
\Pr \lp v_n(T) \geq 1 - \e \: | \: s_{d-1} \lp \mathcal{C}_{\max} \rp \geq \delta n^d \rp & = \Pr \lp v_n(T) - (1 - e^{-\lambda T}) \geq \e - e^{-\lambda T} \: | \: s_{d-1} \lp \mathcal{C}_{\max} \rp \geq \delta n^d \rp \\
& \leq \Pr \lp v_n(T) - (1 - e^{-\lambda T}) \geq \frac{\e}{2} \: | \: s_{d-1} \lp \mathcal{C}_{\max} \rp \geq \delta n^d \rp \\
& \to 0
\end{aligned}
\]
as $n \to \infty$. Since $\Pr \lp s_{d-1} \lp \mathcal{C}_{\max} \rp \geq \delta n^d \rp \to 1$ as $n \to \infty$. This completes the proof of the theorem.


\bibliographystyle{amsalpha}
\bibliography{LWC_simplicial}

\appendix

\section{The metric space of rooted simplicial complexes}
\label{appendix: formulation of LWC}


\subsection{Construction of the metric}

Recall the definitions and notions introduced in Section~\ref{subsec: intro simplicial complex} and Secion~\ref{subsec: LWC}. Recall that $\mathcal{S}_{d-1}$ is the space of isomorphism classes $(d-1)$-rooted simplicial complexes locally finite above dimension $(d-1)$. Also, recall that we define, for any two $(d-1)$-rooted simplicial complexes $(X_1, \sigma_1), (X_2, \sigma_2) \in \mathcal{S}_{d-1}$, 
\[ \label{eq: A.1} \tag{A.1}
d_{\mathcal{S}_{d-1}} \lp (X_1, \sigma_1), (X_2, \sigma_2) \rp := \frac{1}{1 + R^*},
\]
We are going to prove that the space $(\mathcal{S}_{d-1}, d_{\mathcal{S}_{d-1}})$ is a Polish space.

\begin{prop}
\label{prop: appendix the metric is well-defined}
The function $d_{\mathcal{S}_{d-1}}$ is well-defined on $\mathcal{S}_{d-1} \times \mathcal{S}_{d-1}$ and it is a metric on $\mathcal{S}_{d-1}$.
\end{prop}

First, we show that the neighbourhoods of finite radii completely determines the whole isomorphism class.

\begin{lemma}
\label{lemma: appendix neighbourhood determines the whole}
Let $(X_1, \sigma_1)$ and $(X_2, \sigma_2)$ be two $d$-dimensional-connected $(d-1)$-rooted simplicial complexes locally finite above dimension $(d-1)$. Assume that $B_{X_1}(\sigma_1; r) \cong B_{X_2}(\sigma_2; r)$ for all $r \geq 0$. Then $(X_1, \sigma_1) \cong (X_2, \sigma_2)$.
\end{lemma}

\begin{proof}
Fix $r \geq 0$, and let $\phi_r : B_{X_1}(\sigma_1; r) \to B_{X_2}(\sigma_2; r)$ be an isomorphism, i.e., $\phi_r$ is a bijection between $S_0 \lp B_{X_1}(\sigma_1 r) \rp$ and $S_0 \lp B_{X_2}(\sigma_2; r) \rp$, and it preserves all simplices and the root. Let us extend it to $\psi_r$ defined on the whole simplical complex $(X_1, \sigma_1)$ by setting
\[ \label{eq: A.2} \tag{A.2}
\psi_r (v) := \begin{cases}
\phi_r(v) & \text{ if } v \in S_0 \lp B_{X_1}(\sigma_1; r) \rp, \\
v_2 & \text{ elsewhere,}
\end{cases}
\]
where $v_2$ is a fixed arbitrary vertex in $\sigma_2$.

For brevity, let $V_r^{(1)} := S_0 \lp B_{X_1}(\sigma_1; r) \rp$. Note that $V_0^{(1)} = \sigma_1 = \{v_0, \dots, v_{d-1} \}$ and 
\[
\left\{ \psi_r |_{V_0^{(1)}}(v_1), \dots, \psi_r |_{V_0^{(1)}} (v_{d-1}) \right\} = \sigma_2 \; \text{ for every } r.
\]
However, the number of distinct isomorphisms between $B_{X_1}(\sigma_1; 0)$ and $B_{X_2}(\sigma_2; 0)$ is $(d-1)!$. This means that there is at least one isomorphism repeated infinitely many times across all values of $r \geq 0$. Let $\phi_0^{'}$ denote such an isomorphism repeated infinitely many times, and let $\mathcal{N}_0$ denote the set of the values of $r$ for which
\[
\psi_r |_{V_0^{(1)}} = \phi_0^{'}.
\]

Next, note that $\psi_r |_{V_1^{(1)}}$ is an isomorphism between $B_{X_1}(\sigma_1; 1)$ and $B_{X_2}(\sigma_2; 1)$ for every $r \in \mathcal{N}_1$. Since $X$ is locally finite above dimension $(d-1)$, the neighbourhood $B_{X_1}(\sigma_1; 1)$ is a finite simplicial complex; hence, again, the number of distinct isomorphisms between $B_{X_1}(\sigma_1; 1)$ and $B_{X_2}(\sigma_2; 2)$ is finite, and there exists at least one isomorphism in the sequence $\lp \psi_r |_{V_1^{(1)}} \rp_{r \in \mathcal{N}_0}$ repeated infinitely many times, say, $\phi_1^{'}$. We would like to emphasise that $\phi_1^{'}$ and $\phi_0^{'}$ agree on $V_0^{(1)}$ because $\phi_1^{'}$ is obtained from the sequence indexed by $r \in \mathcal{N}_0$. Let $\mathcal{N}_1$ denote the set of the values of $r \in \mathcal{N}_0$ for which
\[
\psi_r |_{V_1^{(1)}} = \phi_1^{'}.
\]

Repeating this procedure, we can obtain a nested sequence of infinite sets $\mathcal{N}_1 \supseteq \mathcal{N}_2 \supseteq \cdots$ and corresponding sequence $\lp \psi_r |_{V_0^{(\ell)}} \rp_{r \in \mathcal{N}_{\ell}}$ for each $\ell \geq 0$. For each $\ell \geq 0$, take $\psi_{\ell}'$ to be the first element of the sequence $\lp \psi_r \rp_{r \in \mathcal{N}_{\ell}}$. Then it follows that $\psi_{\ell}^{'}(v) = \phi_{d-1}^{'} (v)$ for all $\ell \geq d-1$ and all $v \in V_{d-1}^{(1)}$. Let $U_0 := V_0^{(1)} = \sigma_1 = \{v_0, \dots, v_{d-1}\}$ and $U_k : = V_k^{(1)} \setminus V_{k-1}^{(1)}$ for each $k \geq 1$. Since the simplicial complex $X$ is $(k+1)$-connected, we can write $S_0(X) = \cup_{k \geq 0} U_k$, which is a disjoint union.

It follows that the sequence $(\psi_{\ell}^{'})_{\ell \geq 0}$ converges pointwise to
\[
\psi(v) := \lim_{\ell \to \infty} \psi_{\ell}^{'}(v) := \phi_k^{'}(v), \text{ where } v \in U_k.
\]
We claim that the pointwise limit $\psi$ is an isomorphism between $(X_1, \sigma_1)$ and $(X_2, \sigma_2)$. The map $\psi$ is obviously bijective. Let $m \geq 1$ and $u_0, \dots, u_m \in S_0(X_1)$. Write
\[
M := \max_{j = 0, \dots, m} \mathrm{dist}_{X_1}(\sigma_1, u_j),
\]
then $u_0, \dots, u_m \in V_M^{(1)}$. Since $\phi_M^{'}$ is an isomorphism between $B_{X_1}(\sigma_1; M)$ and $B_{X_2}(\sigma_2; M)$, it follows that $\phi_M^{'}(u_0), \dots, \phi_M^{'}(u_M) \in S_0 \lp B_{X_2}(\sigma_2; M) \rp$, and 
\[
\left\{ \phi_M^{'}(u_0), \dots, \phi_M^{'}(u_m) \right \} \in S_m \lp B_{X_2}(\sigma_2; M) \rp \text{ if and only if } \{u_0, \dots, u_m \} \in S_m \lp B_{X_1}(\sigma_1; M) \rp.
\]
This completes the proof.

\end{proof}

\begin{proof}[Proof of Proposition \ref{prop: appendix the metric is well-defined}, Part 1: being well-defined]
Suppose that $(\hat{X}_1, \hat{\sigma}_1) \cong (X_1, \sigma_1)$ and $(\hat{X}_2, \hat{\sigma}_2) \cong (X_2, \sigma_2)$. Then, for any $r \geq 0$, $B_{\hat{X}_1}(\hat{\sigma}_1; r) \cong B_{\hat{X}_2}(\hat{\sigma}_2; r)$ if and only if $B_{X_1}(\sigma_1; r) \cong B_{X_2}(\sigma_2; r)$. Therefore, the function $d_{\mathcal{S}_k}$ is well-defined independent of the choice of representatives from equivalence classes of simplicial complex isomorphisms.
\end{proof}

\begin{proof}[Proof of Proposition \ref{prop: appendix the metric is well-defined}, Part 2: being a metric] First, assume that $d_{\mathcal{S}_k} \lp (X_1, \sigma_1), (X_2, \sigma_2) \rp = 0$. Then it follows that $B_{X_1}(\sigma_1; r) \cong B_{X_2}(\sigma_2; r)$ for all $r \geq 0$. By Lemma \ref{lemma: appendix neighbourhood determines the whole}, it holds that $(X_1, \sigma_1) \cong (X_2, \sigma_2)$.

Symmetry is obvious.

As for transitivity, consider three elements $(X_1, \sigma_1), (X_2, \sigma_2)$, and $(X_3, \sigma_3)$ in $\mathcal{S}_k$. For $i, j \in \{1, \dots, 3\}$, let
\[
R_{ij} := \sup \left\{ r \geq 0 \: : \: B_{X_i}(\sigma_i; r) \cong B_{X_j}(\sigma_j; r) \right\}.
\]
In particular, $B_{X_1}(\sigma_1; r) \cong B_{X_3}(\sigma_3; r)$ for $r \leq R_{13}$ and $B_{X_2}(\sigma_2; r) \cong B_{X_3}(\sigma_3; r)$ for $r \leq R_{23}$. Take $R := \min \{R_{13}, R_{23} \}$, then $B_{X_1}(\sigma_1; r) \cong B_{X_2}(\sigma_2; r)$ for $r \leq R$, which implies that
\[
R_{12} \geq R = \min \{R_{13}, R_{23} \}.
\]
Thus, it follows that
\[
\frac{1}{1 + R_{12}} \leq \max \left\{ \frac{1}{1 + R_{13}}, \frac{1}{1 + R_{23}} \right\}.
\]
This completes the proof.
\end{proof}

\begin{remark}
In fact, the metric $d_{\mathcal{S}_{d-1}}$ is a ultrametric on $\mathcal{S}_{d-1}$.
\end{remark}

\subsection{Separability of the space}

\begin{prop}
\label{prop: separability}
The space $(\mathcal{S}_{d-1}, d_{\mathcal{S}_{d-1}})$ is separable, i.e., it has a countable dense subset.
\end{prop}

Let $\widehat{\mathcal{S}}_{d-1}$ denote the set of all $(d-1)$-rooted finite simplicial complexes, which is obviously countable. For any $(X, \sigma) \in \mathcal{S}_{d-1}$, the neighbourhood $B_X(\sigma; r)$ belongs to $\widehat{\mathcal{S}}_{d-1}$ for any $r \geq 0$. It satisfies that
\[
d_{\mathcal{S}_k} \lp B_X(\sigma; r), (X, \sigma) \rp \leq \frac{1}{1 + r};
\]
hence $B_X(\sigma; r)$ converges to $(X, \sigma)$ as $r \to \infty$ in $\mathcal{S}_k$. This proves that $\mathcal{S}_k$ contains a countable dense subset.

\subsection{Completeness of the space}

\begin{prop}
\label{prop: completeness}
The space $(\mathcal{S}_k, d_{\mathcal{S}_k})$ is complete, i.e., every Cauchy sequence in the space is convergent.
\end{prop}

\begin{lemma}
\label{lemma: compatible sequence}
Let $\lp (X_r, \sigma_r) \rp_{r \geq 0}$ be a sequence of $(k+1)$-connected $k$-rooted simplicial complexes locally finite above dimension $k$. Assume that the sequence $\lp (X_r, \sigma_r) \rp_{r \geq 0}$ is compatible in the sense that $B_{X_s}(\sigma_s; r) \cong (X_r, \sigma_r)$ for all $s \geq r$. Then there exists a $(k+1)$-connected $k$-rooted simplicial complex $(X, \sigma)$ that is locally finite above dimension $k$ and $(X_r, \sigma_r) \cong B_X(\sigma; r)$ for every $r \geq 0$. Moreover, $(X, \sigma)$ is unique up to isomorphisms.

\end{lemma}

\begin{proof}
For each $r$, let $N_r : = \left| S_0 \lp B_{X_r}(\sigma_r; r) \rp \right|$ and $V_r := [N_r] = \{1, \dots, N_r\}$.Note that each $N_r$ is finite. In particular, $V_0 = \{0, \dots, k\}$. We are going to construct a sequence of bijections $\phi_r : S_0 \lp B_{X_r}(\sigma_r; r) \rp \to V_r$ recursively as follows. Let $\phi_0$ be an arbitrary bijection between $S_0 \lp B_{X_0} (\sigma_0 ; 0) \rp = \sigma_0 = \{v_{00}, \dots, v_{0k} \}$ and $V_0 = \{0, \dots, k\}$. Let $\psi_r$ be an isomorphism between $(X_{r-1}, \sigma_{r-1})$ and $B_{X_r}(\sigma_r; r-1)$, and let $\eta_r$ be an arbitrary bijection between $S_0 (X_r) \setminus S_0 \lp B_{X_r}(\sigma_r; r-1) \rp$ and $V_r \setminus V_{r-1}$. Define
\[
\begin{aligned}
\phi_r(v) := \begin{cases}
\phi_{r-1} \lp \psi_r^{-1}(v) \rp & \text{ if } v \in S_0 \lp B_{X_r}(\sigma_r; r-1) \rp, \\
\eta_r(v) & \text{ if } v \in S_0(X_r) \setminus S_0 \lp B_{X_r}(\sigma_r; r-1) \rp.
\end{cases}
\end{aligned}
\]
Then $\phi_r$ is a bijection from $S_0(X_r)$ to $V_r$. We further define $(X_r^{'}, \sigma_r^{'}) = \lp \phi_r(X_r), \phi_r(\sigma_r) \rp$, where $\phi_r(X_r)$ is the simplicial complex that consists of the vertex set $\{\phi_r(v) \: : \: v \in S_0(X_r) \}$ and simplex sets $S_p \lp \phi_r(X_r) \rp = \{ \{ \phi_r(u_0), \dots, \phi_r(u_p) \} \: : \: \{u_0, \dots, u_p \} \in S_p(X_r) \}$ for every $p \geq 1$.

Notice that $\sigma_r^{'} = \{0, \dots, k\}$ for every $r \geq 0$. Furthermore, $B_{X_r^{'}}(\sigma_r^{'}; r) \cong B_{X_r}(\sigma_r; r) = (X_r, \sigma_r)$. 

Finally, we are ready to construct $(X, \sigma)$. Set $\sigma := \{0, \dots, k\}$ and $S_p(X) := \cup_{r \geq 0} S_p(X_r^{'})$ for every $p$. Then, one can check $B_X(\sigma; r) \cong B_{X_r^{'}}(\sigma_r^{'}; r) \cong B_{X_r}(\sigma_r; r) = (X_r, \sigma_r)$. And, $(X, \sigma)$ is $(k+1)$-connected and locally finite above dimension $k$. To check its uniqueness, suppose that there is another $k$-rooted simplicial complex $(Y, \rho)$ such that $B_{Y}(\rho; r) \cong B_X(\sigma_r; r) = (X_r, \sigma_r)$ for every $r \geq 0$. Then $B_Y(\rho; r) \cong B_X(\sigma; r)$ for every $r \geq 0$, which implies that $(Y, \rho) \cong (X, \sigma)$ by Lemma \ref{lemma: appendix neighbourhood determines the whole}.

\end{proof}

Now we are ready to prove Proposition \ref{prop: completeness}

\begin{proof}[Proof of Proposition \ref{prop: completeness}]
Let $\lp (X_n, \sigma_n) \rp_{n \geq 1}$ be a Cauchy sequence in $\mathcal{S}_k$---without loss of generality, we consider $(X_n, \sigma_n)$ a representative of the equivalence class of it. Let $\e > 0$, then there exists $N \in \N$ such that $n, m \geq N$ implies that
\[
d_{\mathcal{S}_k} \lp (X_n, \sigma_n), (X_m, \sigma_m) \rp < \e,
\]
which implies that
\[
B_{X_n}(\sigma_n; r) \cong B_{X_m}(\sigma_m; r)
\]
for all $r \leq 1/e - 1$ whenever $n,m \geq N$. Equivalently, for every $r \geq 0$, there exists $N_r \in \N$ such that, for all $n \geq N_r$,
\[
B_{X_n}(\sigma_n; r) \cong B_{X_{N_r}}(\sigma_{N_r}; r).
\]
Indeed, we can select $N_r$ so that $r \mapsto N_r$ is strictly increasing. Let $(X'_r, \sigma'_r) = B_{X_{N_r}}(\sigma_{N_r}; r)$, then the sequence $\lp (X'_r, \sigma'_r) \rp_{r \geq 0}$ forms a compatible sequence. By Lemma \ref{lemma: compatible sequence}, there exists a $k$-rooted simplicial complex $(X, \sigma)$ locally finite above dimension $k$, satisfying $B_X(\sigma; r) \cong (X'_r, \sigma'_r)$ for every $r$. On the other hand, it also holds that
\[
B_{X_n}(\sigma_n; r) \cong B_{X_{N_r}}(\sigma_{N_r}; r) = (X'_r, \sigma'_r)
\]
for $n \geq N_r$. Thus, for all $n \geq N_r$,
\[
d_{\mathcal{S}_k} \lp (X, \sigma), (X_n, \sigma_n) \rp \leq \frac{1}{1 + r} \leq \e.
\]
Since the choice of $\e$ was arbitrary, we conclude that $(X_n, \sigma_n)$ converges to $(X, \sigma)$ in $\mathcal{S}_k$. This completes the proof.
\end{proof}

Combining Proposition \ref{prop: separability} and \ref{prop: completeness}, we conclude that the metric space $(\mathcal{S}_k, d_{\mathcal{S}_k})$ is a Polish space.

\section{Proof of Propositions in Section~\ref{subsec: LWC and connected components}}
\label{appendix: properties of LWC}

First, we prove Proposition~\ref{prop: number of components}. The following proof is an extension of that presented in \cite[Corollary 2.21]{Hofstad:2024}.

\begin{proof}
Let $h$ be a function on $\mathcal{S}_d$ defined by
\[ \label{eq: B. 1} \tag{B.1}
h(X, \tau) := \frac{1}{s_{d-1} \lp \mathcal{C}(\tau) \rp}.
\]
We claim that $h$ is bounded and continuous (recall that $h(X, \tau) = 0$ when $s_{d-1} \lp \mathcal{C}(\tau) \rp = \infty$ by convention). It is obvious that $h$ is bounded since $0 \leq h \leq 1$. 

To prove its continuity, let $(X, \tau) \in \mathcal{S}_d$ and $\e > 0$ be given. If $s_{d-1}\lp \mathcal{C}(\tau) \rp = \infty$, choose large enough $R > 0$ so that $\delta := (1 + R)^{-1} < \e$. Then, any $(X', \tau') \in \mathcal{S}_d$ with $d_{\mathcal{S}_d} \lp (X, \tau), (X', \tau') \rp < \delta$ satisfies $B_X(\tau; R) \cong B_{X'}(\tau'; R)$. Since $s_{d-1} \lp \mathcal{C}(\tau) \rp = \infty$, $s_{d-1} \lp \partial B_{X'}(\tau'; r) \rp \geq 1$ for all $0 \leq r \leq R$. Thus, 
\[ 
0 \leq h(X', \tau') = \frac{1}{s_{d-1} \lp \mathcal{C}(\tau') \rp} \leq \frac{1}{s_{d-1} \lp B_{X'}(\tau'; R) \rp } \leq \frac{1}{1 + R} < \e.
\]
This implies that $h$ is continuous at every $(X, \tau)$ with $s_{d-1} \lp \mathcal{C}(\tau) \rp = \infty$. Next, suppose that $s_{d-1} \lp \mathcal{C}(\tau) \rp < \infty$. This means that there exists some finite $R > 0$ such that $\mathcal{C}(\tau) = B_X(\tau; R)$. Choose $\delta > 0$ small enough that $\delta < \frac{1}{1 + R}$. Then, for any $(X', \tau') \in \mathcal{S}_d$ with $d_{\mathcal{S}_d} \lp (X, \tau), (X', \tau') \rp < \delta$, we have $\mathcal{C}(\tau) \cong \mathcal{C}(\tau')$. This implies that $h(X, \tau) = h(X', \tau')$. Therefore, $h$ is continuous on $(X, \tau) \in \mathcal{S}_d$.

By the definition of the local-weak convergence in probability, we obtain
\[ \label{eq: B.2} \tag{B.2}
\E \lb h(X_n, \sigma_n) \: | \: X_n \rb \overset{\Pr}{\to} \E \lb h (X, \tau) \rb.
\]
Note that 
\[ \label{eq: B.3} \tag{B.3}
C_n = \sum_{\sigma \in S_{d-1}(X_n)} \frac{1}{s_{d-1} \lp \mathcal{C}(\sigma) \rp}.
\]
Thus, the left-hand side of (\ref{eq: B.2}) is equal to
\[ \label{eq: B.4} \tag{B.4}
\frac{C_n}{s_{d-1} (X_n)}.
\]
This completes the proof.
\end{proof}

Next, we prove Proposition~\ref{prop: upper bound for the giant}. The following proof is an extension of that presented in \cite[Corollary 2.1]{Hofstad:2021}.

\begin{proof}[Proof of Proposition~\ref{prop: upper bound for the giant}]
For each $k \in \N$, let us define
\[ \label{eq: B.5} \tag{B.5}
Z_{\geq k} := \sum_{\sigma \in S_{d-1}(X_n)} \1\{ s_{d-1} \lp \mathcal{C}(\sigma) \rp \geq k \}.
\]
Note that $\{s_{d-1} \lp \mathcal{C}(\tau) \rp \geq k \} = \{s_{d-1} \lp B_{X}(\tau; k) \rp \geq k \}$. We obtain
\[ \label{eq: B.6} \tag{B.6}
\frac{Z_{\geq k}}{s_{d-1} (X_n)} = \E \lb \1\{s_{d-1} \lp B_{X_n}(\sigma_n; k) \rp \geq k \} \: | \: X_n \rb \overset{\Pr}{\to} \zeta_{\geq k},
\]
where 
\[ \label{eq: B.7} \tag{B.7}
\zeta_{\geq k} := \mu \lp s_{d-1} \lp \mathcal{C}(\tau) \rp \geq k \rp.
\]
On the other hand, for every $k \in \N$, $\{s_{d-1} \lp \mathcal{C}_{\max} \rp \geq k \} = \{ Z_{\geq k} \geq k \}$. Additionally, we have $s_{d-1} \lp \mathcal{C}_{\max} \rp \leq Z_{\geq k}$. 

Note that $\zeta_{\geq k} \to \zeta := \mu \lp s_{d-1} \lp \mathcal{C}(\tau) \rp = \infty \rp$. Take large enough $K$ so that $\zeta \geq \zeta_{\geq K} - \e/2$ for all $k > K$. Then, for all large enough $n$, we obtain
\[ \label{eq: B.8} \tag{B.8}
\begin{aligned}
\Pr \lp s_{d-1} \lp \mathcal{C}_{\max} \rp \geq s_{d-1}(X_n) (\zeta + \e) \rp & \leq \Pr \lp Z_{\geq k} \geq s_{d-1}(X_n) (\zeta + \e) \rp \\
& \leq \Pr \lp Z_{\geq k} \geq s_{d-1} (X_n) ( \zeta_{\geq k} + \e/2) \rp.
\end{aligned} 
\]
By (\ref{eq: B.6}), the last expression in (\ref{eq: B.8}) tends to $0$ as $n \to \infty$. This completes the proof.
\end{proof}

The following proof of Proposition~\ref{prop: side condition for the giant} can be viewed as an extension of that presented in \cite[Theorem 2.2]{Hofstad:2021}.

\begin{proof}[Proof of Proposition~\ref{prop: side condition for the giant}]
When $\zeta = 0$, Proposition~\ref{prop: upper bound for the giant} has already proven the proposed result, so it suffices to focus on the case where $\zeta > 0$. Let $(C_i)_{i \geq 1}$ be the sequence of the sizes (the number of $(d-1)$-simplices included) of the $d$-dimensional connected components in $X_n$ ordered in size, from the largest to the smallest (with ties broken arbitrarily). In particular, $C_1 := s_{d-1} \lp \mathcal{C}_{\max} \rp$.

Since $C_i \leq C_1$ for all $i \geq 1$, for every $k \geq 1$, we have
\[ 
\sum_{i \geq 1} C_i^2 1\{C_i \geq k \} \leq C_1 \sum_{i \geq 1} C_i \1\{ C_i \geq k\} ,
\]
which gives us
\[ \label{eq: B.9} \tag{B.9}
\frac{C_1}{s_{d-1}(X_n)} \geq \frac{\frac{1}{s_{d-1}(X_n)^2}\sum_{i \geq 1} C_i^2 \1\{C_i \geq k\}}{\frac{1}{s_{d-1}(X_n)} \sum_{i \geq 1} C_i \1\{C_i \geq k\} }
\]
Let us consider the denominator of (\ref{eq: B.9}) first. We have
\[ \label{eq: B.10} \tag{B.10}
\begin{aligned}
\frac{1}{s_{d-1}(X_n)} \sum_{i \geq 1} C_i \1\{ C_i \geq k\} & = \frac{1}{s_{d-1}(X_n)} \sum_{\sigma \in S_{d-1}(X_n)} \1 \{ s_{d-1} \lp \mathcal{C}(\sigma) \rp \geq k \} \\
& = \frac{1}{s_{d-1} (X_n)} Z_{\geq k} \\
& \overset{\Pr}{\to} \zeta_{\geq k}
\end{aligned}
\]
as $n \to \infty$ by the local-weak convergence in probability of $(X_n)_{n \geq 1}$. Since $\zeta_{\geq k} \to \zeta$ as $k \to \infty$, we can write
\[ \label{eq: B.11} \tag{B.11}
\frac{1}{s_{d-1}(X_n)} \sum_{i \geq 1} C_i \1\{ C_i \geq k\} = \zeta + o_{k, \Pr}(1).
\]
Let us define
\[ \label{eq: B.12} \tag{B.12}
\begin{aligned}
X_{n, k} & := \frac{1}{s_{d-1}(X_n)^2} \sum_{i \geq 1} C_i^2 \1\{ C_i \geq k \} - \lp \frac{1}{s_{d-1}(X_n)} \sum_{i \geq 1} C_i \1 \{ C_i \geq k \} \rp^2 \\
& = \frac{1}{s_{d-1}(X_n)^2} \sum_{i \geq 1} C_i^2 \1\{ C_i \geq k \} + \lp \frac{1}{s_{d-1}(X_n) Z_{\geq k} } \rp^2,
\end{aligned}
\]
so that we can express the numerator of the right-hand side of (\ref{eq: B.9}) as
\[ \label{eq: B.13} \tag{B.13}
\lp \frac{1}{s_{d-1}(X_n)} Z_{\geq k} \rp^2 + X_{n,k}.
\]
By (\ref{eq: B.11}), we can write (\ref{eq: B.13}) as
\[ \label{eq: B.14} \tag{B.14}
\zeta^2 + o_{k, \Pr}(1) + X_{n,k}.
\]
The term $X_{n,k}$ can be further expressed as follows:
\[ \label{eq: B.15} \tag{B.15}
\begin{aligned}
X_{n,k} & = \frac{1}{s_{d-1}(X_n)^2} \sum_{i, j \geq 1, i \neq j} C_i C_j \1\{ C_i \geq k, C_j \geq k\} \\
& = \frac{1}{s_{d-1} (X_n)^2} \left| \left\{ (\sigma, \sigma') \in S_{d-1}(X_n)^2 \: : \: s_{d-1} \lp \mathcal{C}(\sigma) \rp \geq k, s_{d-1} \lp \mathcal{C}(\sigma') \rp \geq k, \sigma \centernot\longleftrightarrow \sigma' \right\} \right|
\end{aligned}
\]
Note that
\[ \label{eq: B.16} \tag{B.16}
\begin{aligned}
& \E \lb X_{n,k} \rb \\
& = \E \lb \frac{1}{s_{d-1} (X_n)^2} \left| \left\{ (\sigma, \sigma') \in S_{d-1}(X_n)^2 \: : \: s_{d-1} \lp \mathcal{C}(\sigma) \rp, s_{d-1} \lp \mathcal{C}(\sigma') \rp \geq k, \sigma \centernot\longleftrightarrow \sigma' \right\} \right| \rb \\
& = \E \lb \E \lb \frac{1}{s_{d-1} (X_n)^2} \left| \left\{ (\sigma, \sigma') \in S_{d-1}(X_n)^2 \: : \: s_{d-1} \lp \mathcal{C}(\sigma) \rp, s_{d-1} \lp \mathcal{C}(\sigma') \rp \geq k, \sigma \centernot\longleftrightarrow \sigma' \right\} \right| \: \big| \: X_n \rb \rb \\
& = \E \lb \E \lb \1\{s_{d-1} \lp \mathcal{C}(\sigma_1) \rp, s_{d-1} \lp \mathcal{C}(\sigma_2) \rp \geq k, \sigma_1 \centernot\longleftrightarrow \sigma_2 \} \: \big| \: X_n \rb \rb,
\end{aligned}
\]
where $\sigma_1$ and $\sigma_2$ are $(d-1)$-simplices in $X_n$ chosen uniformly at random and independent of each other. The last expression of (\ref{eq: B.16}) is equal to
\[
\Pr \lp s_{d-1} \lp \mathcal{C}(\sigma_1) \rp s_{d-1} \lp \mathcal{C}(\sigma_2) \rp \geq k, \sigma_1 \centernot\longleftrightarrow \sigma_2 \rp.
\]
Thus, by the assumed condition (\ref{eq: 2.2.3}), we have
\[ \label{eq: B.17} \tag{B.17}
\lim_{k \to \infty} \limsup_{n \to \infty} \E \lb X_{n,k} \rb = 0.
\]
Summing up (\ref{eq: B.13}) and (\ref{eq: B.17}), we obtain
\[ \label{eq: B.18} \tag{B.18}
\frac{C_1}{s_{d-1}(X_n)} \geq \frac{\zeta^2 + o_{k, \Pr}(1)}{\zeta + o_{k, \Pr}(1)} = \zeta + o_{k, \Pr}(1).
\]
Since $k$ was chosen arbitrarily, we can conclude that
\[ \label{eq: B.19} \tag{B.19}
C_1 \geq s_{d-1}(X_n) \zeta (1 + o_{\Pr}(1)).
\]
Combining this with the result of Proposition~\ref{prop: upper bound for the giant}, we obtain the following concentration of $C_1$:
\[ \label{eq: B.20} \tag{B.20}
\frac{C_1}{s_{d-1}(X_n)} \overset{\Pr}{\to} \zeta.
\]

To prove the second assertion, observe that the following holds w.h.p. as $n \to \infty$:
\[ \label{eq: B.21} \tag{B.21}
\frac{1}{s_{d-1}(X_n)} C_2 \1\{ C_2 \geq k\} \leq \frac{1}{s_{d-1}(X_n)} \sum_{i \geq 2} C_i \1\{ C_i \geq k \} = \frac{1}{s_{d-1}(X_n) }\sum_{i \geq 1} C_i \1\{ C_i \geq k \} - \frac{C_1}{s_{d-1}(X_n)}.
\]
Combining \eqref{eq: B.11} and \eqref{eq: B.20}, we obtain that the rightmost expression of \eqref{eq: B.21} is of order $o_{k, \Pr}(1)$. Since $k$ is arbitrary, this concludes that
\[
\frac{C_2}{s_{d-1}(X_n)} \overset{\Pr}{\to} 0
\]
as $n \to \infty$.

\end{proof}

\end{document}